\documentclass[10pt, twoside]{article}

\usepackage{times}
\usepackage{amssymb,amsmath,amsfonts,amsthm}
\usepackage{mhequ} 
\usepackage{mathtools, mathrsfs}
\usepackage[title]{appendix}
\usepackage{microtype}
\usepackage{tikz}
\usepackage{booktabs}
\usepackage{caption}
\usetikzlibrary{decorations.pathmorphing}
\usetikzlibrary{positioning}
\usepackage{color}
\usepackage{bbm}

\usepackage{hyperref}
\usepackage{enumerate}
\usepackage{enumitem}
\usepackage{bm}
\usepackage{dcolumn}
\usepackage{tabu}
\usepackage{subfig}
\usepackage{breakurl}
\usepackage{cases}
\usepackage{wasysym}
\usepackage{centernot}
\usepackage{comment}
\usepackage{soul}
\usepackage[includehead,includefoot,body=17cm,top=3.2cm,bottom=2.8cm, left=2.2cm]{geometry}
\usepackage{ stmaryrd }
\usepackage{authblk}

\theoremstyle{plain}
\newtheorem{lemma}{Lemma}[section]
\newtheorem{theorem}[lemma]{Theorem}
\newtheorem{corollary}[lemma]{Corollary}
\newtheorem{proposition}[lemma]{Proposition}
\newtheorem{assumption}[lemma]{Assumption}
\theoremstyle{definition}
\newtheorem{definition}[lemma]{Definition}
\newtheorem{remark}[lemma]{Remark}

\newtheorem*{notation*}{Notation}
\newtheorem*{plan*}{Plan of the paper}
\newtheorem*{ackno*}{Acknowledgements}

\makeatletter

\newcommand{\E}{\mathbb{E}}
\newcommand{\N}{\mathbb{N}}
\newcommand{\R}{\mathbb{R}}
\newcommand{\T}{\mathbb{T}}
\newcommand{\Z}{\mathbb{Z}}
\newcommand{\X}{\mathbb{X}}
\newcommand{\Prob}{\mathbb{P}}

\newcommand{\CA}{\mathcal{A}}
\newcommand{\CB}{\mathcal{B}}
\newcommand{\CC}{\mathcal{C}}
\newcommand{\CD}{\mathcal{D}}
\newcommand{\CF}{\mathcal{F}}
\newcommand{\CG}{\mathcal{G}}
\newcommand{\CH}{\mathcal{H}}

\newcommand{\CJ}{\mathcal{J}}
\newcommand{\CK}{\mathcal{K}}
\newcommand{\CL}{\mathcal{L}}
\newcommand{\CM}{\mathcal{M}}

\newcommand{\CO}{\mathcal{O}}

\newcommand{\CQ}{\mathcal{Q}}
\newcommand{\CR}{\mathcal{R}}
\newcommand{\CS}{\mathcal{S}}
\newcommand{\CT}{\mathcal{T}}
\newcommand{\CX}{\mathcal{X}}
\newcommand{\CW}{\mathcal{W}}

\newcommand{\FK}{\mathfrak{K}}
\newcommand{\FI}{\mathfrak{I}}

\newcommand{\s}{\mathfrak{s}}

\newcommand{\SF}{\mathscr{F}}

\newcommand{\SL}{\mathscr{L}}

\newcommand{\SR}{\mathscr{R}}

\newcommand{\ST}{\mathscr{T}}

\newcommand{\eps}{\varepsilon}

\newcommand{\Id}{\mathrm{Id}}
\newcommand{\Imma}{\mathrm{Im}}
\newcommand{\1}{\mathbbm{1}}
\newcommand{\one}{\textcolor{blue}{\mathbf{1}}}

\newcommand{\dd}{\mathrm{d}}      
\newcommand{\VERT}{\vert\kern-0.25ex\vert\kern-0.25ex\vert} 
\newcommand{\eqdef}{\stackrel{\mathclap{\mbox{\tiny def}}}{=}}
\newcommand{\BV}{\textcolor{blue}{\mathbf{V}}}
\newcommand{\BC}{\mathbf{C}}

\makeatother

\definecolor{darkred}{rgb}{0.9,0.1,0.1}
\definecolor{darkblue}{rgb}{0,0,0.7}
\definecolor{darkgreen}{rgb}{0,0.5,0}

\def\snorm#1{| #1|}
\def\onorm#1{| #1|_0}
\def\enorm#1{| #1|_{\eps}}
\def\pnorm#1{\bigl[ #1\bigr]}

\hypersetup{
	citecolor=blue,
	colorlinks=true, 
	urlcolor= blue, 
	linkcolor= black, 
	bookmarksopen=true, 
	pdftitle={DiscreteSPDEs}, 
}

\usetikzlibrary{external}

\tikzset{
    wavy/.style={decorate, decoration={snake,amplitude=0.2mm,segment length=1mm,post length=0.1cm}, ->, thick, -stealth, draw=black},
    verywavy/.style={decorate, decoration={snake,amplitude=0.7mm,segment length=0.7mm,post length=0.13cm}, ->, -stealth, draw=black},
    arrow/.style={->, thick, >=stealth, draw=black},
    dashed/.style={thick, densely dotted,->, >=stealth, draw=black},
    grayNode/.style={shape=circle,minimum size=0.06cm,inner sep=0, draw=gray, fill=gray},
    plainNode/.style={shape=circle,minimum size=0.06cm,inner sep=0, fill=black},
    deriv/.style={->>}
}

\tikzset{
	xi/.style={circle,fill=blue!10,draw=black,inner sep=0pt,minimum size=1.2mm},
	xib/.style={circle,fill=blue!10,draw=black,inner sep=0pt,minimum size=1.6mm},
	not/.style={circle,fill=black,draw=black,inner sep=0pt,minimum size=0.5mm,line width=0.5pt},
	>=stealth,
	}
\makeatletter
\def\DeclareSymbol#1#2#3{\expandafter\gdef\csname MH@symb@#1\endcsname{\tikz[baseline=#2,scale=0.15]{#3}}}
\def\<#1>{\csname MH@symb@#1\endcsname}
\makeatother

\DeclareSymbol{tree1}{0}{\draw (0,0.8) node[not] {};}
\DeclareSymbol{tree2}{0}{\draw (-0.3, 0.8) node[not] {}; \draw (0.3,0.8) node[not] {};}
\DeclareSymbol{tree11}{0}{\draw (0,0.8) node[not] {} -- (0,0); }
\DeclareSymbol{tree21}{0}{\draw (0,0.8) node[not] {} -- (0,0); \draw (0.5,0) node[not] {};}
\DeclareSymbol{tree22}{0}{\draw (-0.4,0.8) node[not] {} -- (0,0); \draw (0.4,0.8) node[not] {} -- (0,0); \draw (0.5,0) node[not] {};}
\DeclareSymbol{tree12}{0}{\draw (-0.4, 0.8) node[not] {} -- (0,0); \draw (0.4,0.8) node[not] {} -- (0,0);}
\DeclareSymbol{tree122}{0}{\draw (-0.4,0.8) node[not] {} -- (0,0); \draw (0.4,0.8) -- (0,0); \draw (0,1.6) node[not] {} -- (0.4,0.8); \draw (0.8,1.6) node[not] {} -- (0.4,0.8);}
\DeclareSymbol{tree124}{0}{\draw (-0.4,0.4) -- (0,0); \draw (0.4,0.4) -- (0,0); \draw (0.2,1.3) node[not] {} -- (0.4,0.4); \draw (0.9,1.3) node[not] {} -- (0.4,0.4); \draw (-0.9,1.3) node[not] {} -- (-0.4,0.4); \draw (-0.2,1.3) node[not] {} -- (-0.4,0.4);}
\DeclareSymbol{tree1222}{0}{\draw (-0.4,0.8) node[not] {} -- (0,0); \draw (0.4,0.8) -- (0,0); \draw (0,1.6) node[not] {} -- (0.4,0.8); \draw (0.8,1.6) -- (0.4,0.8); \draw (0.4,2.4) node[not] {} -- (0.8,1.6); \draw (1.2,2.4) node[not] {} -- (0.8,1.6);}

\DeclareSymbol{bigtree1}{0}{\draw (0,0.7) node[not, minimum size=0.1cm, fill=blue!80, draw=blue!80] {};}
\DeclareSymbol{bigtree21}{0}{\draw[line width=1pt, fill=blue!80, draw=blue!80] (0,1.5) node[not, minimum size=0.1cm, fill=blue!80, draw=blue!80] {} -- (0,0); \draw (0.8,0) node[not, minimum size=0.1cm, fill=blue!80, draw=blue!80] {};}
\DeclareSymbol{bigtree11}{0}{\draw[line width=1pt, fill=blue!80, draw=blue!80] (0,1.5) node[not, minimum size=0.1cm, fill=blue!80, draw=blue!80] {} -- (0,0); }
\DeclareSymbol{tree222}{0}{\draw (-0.4,0.8) node[not] {} -- (0,0); \draw (0.4,0.8) -- (0,0); \draw (0,1.6) node[not] {} -- (0.4,0.8); \draw (0.8,1.6) node[not] {} -- (0.4,0.8); \draw (0.7,0) node[not] {};}

\makeatother
\numberwithin{equation}{section}

\begin{document}

\title{Space-time discrete KPZ equation}
\date{}
\author{G. Cannizzaro\footnote{University of Warwick, \textit{G.Cannizzaro@warwick.ac.uk}}$\,\,$ and K. Matetski\footnote{University of Toronto, \textit{matetski@math.toronto.edu}}}
\date{\today}

\maketitle

\begin{abstract}
We study a general family of space-time discretizations of the KPZ equation and show that they converge to its solution. The approach we follow makes use of basic elements of the theory of regularity structures [M. Hairer, A theory of regularity structures, Invent. Math. 2014] as well as its discrete counterpart [M. Hairer, K. Matetski, Discretizations of rough stochastic PDEs, 2015]. Since the discretization is in both space and time and we allow non-standard discretization for the product, the methods mentioned above have to be suitably modified in order to accommodate the structure of the models under study. 
\end{abstract}
 
\tableofcontents


\section{Introduction}

The celebrated KPZ equation is a singular Stochastic Partial Differential Equation (SPDE) which was introduced by three physicists, Kardar, Parisi and Zhang to which it owes its name~\cite{KPZ86}, as a description of the fluctuations of a randomly growing surface. 
It is formally given by the expression
\begin{equ}\label{e:KPZ}
\partial_t h=\Delta h + (\partial_x h)^2 +\xi\;,
\end{equ}
in which $h$ is the {\it height function} (the surface mentioned above) and $\xi$ is a space-time white noise, i.e. a Gaussian random field on $L^2(\R^2)$ with covariance satisfying
\[
\E\big[\langle \xi,\varphi\rangle\langle \xi,\psi\rangle\big]=\langle \varphi, \psi\rangle_{L^2}\;,
\]
for $\varphi,\,\psi\in L^2(\R^2)$. It is well-known that the realizations of $\xi$, seen as a random distribution, belong almost surely to a H\"older-type space of distributions of regularity $\alpha<-\frac{3}{2}$ (see below for the exact definition of these spaces). The regularising properties  of the heat operator $\partial_t-\Delta$ then imply that the expected regularity of $h$ cannot be better than $\alpha+2<\frac{1}{2}$ which in turn makes it unclear what meaning to attribute to the square of its derivative.

That said, the KPZ equation is presumed (and in certain cases, shown) to be a universal object and, by now, people have developed various methods in order to make sense of it. The first significant work in this direction is due to Bertini and Giacomin~\cite{BG97}. 
They defined the solution to~\eqref{e:KPZ} as the logarithm of the solution to the linear multiplicative stochastic heat equation (known to be strictly positive~\cite{Mue91}), which, in turn, can be treated thanks to classical tools from stochastic analysis (see for example~\cite{Wal86}). They not only showed well-posedness but also proved that a suitably rescaled particle system (the Weakly Asymmetric Simple Exclusion Process, briefly WASEP) converges to it, thus confirming that their notion of solution is the physically correct one. 
Nevertheless, their approach is based on a non-linear transformation (the logarithm mentioned above), the so-called {\it Cole-Hopf transform}, that is in general difficult to implement at the discrete level (although it has been successfully done in various but specific situations, see~\cite{DT16, CT17, CST16, Lab16, CS16}). 

More recently, in~\cite{GJ10,GJ14}, Gon\c{c}alves and Jara introduced the notion of {\it energy solution}, which can be seen as a martingale problem-type formulation of the KPZ equation, and proved existence of such solution as well as the fact that all the limit points of a significant class of suitably rescaled interacting particle systems (among which WASEP and zero-range processes) are indeed energy solutions. 
The problem of uniqueness has been settled for a slight reformulation of this notion, given in~\cite{GJ13}, by Gubinelli and Perkowski in~\cite{GP15b} and from then on a number of results confirming the claim of universality has been obtained, e.g.~\cite{GS15,FGS15,BGS16,GJ16, GP16,DGP16}. 
The main difficulty of their method lies in the fact that, in order to apply it, it is necessary to know the invariant measure for the system explicitly and have a good control over it already at the discrete level. 

Almost contemporarily, a different way of making sense of~\eqref{e:KPZ} has been established. Thanks to the theory of rough paths first~\cite{Hai13} (introduced by Lyons in~\cite{Lyo98}) and paracontrolled calculus~\cite{Reloaded} (first in~\cite{GIP15}), regularity structures~\cite{Hai14} then, it is now possible to give a {\it pathwise} meaning to the KPZ equation which has the significant advantage of dealing with the equation directly (no need of the Cole-Hopf transform) and the invariant measure plays no role whatsoever. Such techniques are extremely stable and general enough to be applicable to a huge number of singular SPDEs.  The price to pay though is that, in this context, one has to work with strong topologies and the solution is identified in spaces constrained by certain non-trivial algebraic relations that are not always easily identifiable. Nevertheless, they have been proved to be very powerful and a significant body of literature is now available,~\cite{HQ15, HS15, Reloaded, FH16, Hos16}.  

The aim of the present paper is to show that a family of space-time discretizations of the KPZ equation converges to its solution once the discretization is removed. We will prefer to work with the Stochastic Burgers Equation (SBE) obtained by formally taking the derivative of $h$ in~\eqref{e:KPZ}, i.e. setting $u\eqdef\partial_x h$ so that $u$ solves
\begin{equ}\label{e:SBE}
\partial_t u=\Delta u +\partial_x u^2 + \partial_x \xi\,,\qquad u(0,\cdot)=u_0(\cdot)\;,
\end{equ}
where $u_0$ is the initial condition and $\xi$ a space-time white noise. An advantage of SBE is that no renormalization is required (see Remark~\ref{rem:Ren2}). We will assume throughout periodic boundary conditions, so that the space variable ranges over the one-dimensional torus $\T\eqdef\R/\Z$. Let $\eps>0$, $\xi$ be as above and define the discrete noise as
\begin{equ}[e:DNoise]
\xi^\eps(z)\eqdef \eps^{-3} \langle \xi, \1_{|\eps^{-\s}(\cdot-z)| \leq 1/2}\rangle
\end{equ}
for $z\in(\eps^2\Z)\times\T_\eps$ (with $\T_\eps\eqdef \T\cap(\eps\Z)$), where for $z=(t,x)\in\R^2$, we write $\eps^{-\s} z = (\eps^{-2} t, \eps^{-1} x)$ and $|z| = |t| \vee |x|$. Then, the family of space-time discretizations we have in mind is obtained by a forward explicit scheme, which, at scale $\eps$ reads 
\begin{equ}[e:DiscreteSBE]
\bar D_{t,\eps^2} u^\eps=\Delta_\eps u^{\eps} +D_{x,\eps} B_\eps(u^{\eps},u^{\eps}) +D_{x,\eps}\xi^\eps\,,\qquad u^\eps(0,\cdot)=u^\eps_0(\cdot)\;,
\end{equ}
where $(t,x)$ belongs to the space-time grid $(\eps^2\N)\times\T_\eps$, $u_0^\eps$ is the initial condition, $\{\xi^\eps(z)\}_z$ is as in~\eqref{e:DNoise}, $\bar D_{t,\eps^2}$ is the forward difference (i.e. $\bar D_{t,\eps^2}f(t)=\eps^{-2}(f(t+\eps^2)-f(t))$) and $\Delta_\eps,\,D_{x,\eps},\,B_\eps$ are, respectively, the discrete Laplacian, spatial derivative and product whose precise definition will be given in~\eqref{e:DOperators}. 
Only using the most basic elements of the theory of regularity structures~\cite{Hai14} and their discrete counterpart (see~\cite{HM15}), we can prove the following theorem, which represents the main result of this work. All the norms and ``distances'' used in the statement are defined in the following section and we provide a proof of this theorem in Section~\ref{sec:Conv}.

\begin{theorem}\label{t:Convergence}
Make Assumptions~\ref{a:nu},~\ref{a:pi},~\ref{a:MeasureMu} on the discrete operators $\Delta_\eps,\,D_{x,\eps},\,B_\eps$ and fix $\eta>-1$. Let $\xi$ be a space-time white noise on the probability space $(\Omega, \CF,\Prob)$ and $\{\xi^\eps(z)\}_z$, $z\in(\eps^2\Z)\times\T_\eps$, be given by~\eqref{e:DNoise}. For $\{u_0^\eps\}_\eps$, a sequence of random functions on $\T_\eps$ independent of $\xi^\eps$, let $u^\eps$ be the solution to~\eqref{e:DiscreteSBE} with respect to $\xi^\eps$, starting at $u_0^\eps$. If there exists $u_0\in\CC^\eta$ such that
\begin{equ}
\lim_{\eps\to 0} \|u_0^\eps,u_0\|_{\CC^\eta}^{(\eps)}=0
\end{equ}
in probability, then for every $\alpha<-\frac{1}{2}$ there exists a stopping time $T_\infty$ 
such that for all $T< T_\infty$ the limit
\begin{equ}
\lim_{\eps\to 0}\|u^\eps,u\|_{\CC^{\alpha}_{\eta, T}}^{(\eps)}=0
\end{equ} 
holds in probability, where $u$ is the unique solution to
\begin{equ}\label{e:SBERen1}
\partial_t u=\Delta u +\partial_x u^2 + C\partial_x u + \partial_x \xi\,,\qquad u(0,\cdot)=u_0(\cdot)\;,
\end{equ}
and a constant $C$ depends on the specific definition of the discrete Laplacian, derivative and product.
\end{theorem}

The presence of a finite constant in~\eqref{e:SBERen1} does not come as a surprise. If it is true in general that specific features of the discrete model can influence the shape of the limiting equation, it is not clear what is the physical meaning of the transport term in this context. The constant $C$, in turn, depends on the definition of the discrete operators $\bar D_{t,\eps^2},\,\Delta_\eps,\,D_{x,\eps}$ and $B_\eps$, it might happen that $C$ can be taken to be $0$ (for example, under the assumptions of Theorem~\ref{t:Convergence}, if $D_{x,\eps}$ is antisymmetric). We point out that a similar phenomenon has already been observed in~\cite{Reloaded} when taking discretizations of the type in~\eqref{e:DiscreteSBE} but involving only the space variable. 

A special role is played by Zabusky-type discretizations (see~\cite{KZ65}, in~\cite{Reloaded} the authors refer to them as discretizations of Sasamoto-Spohn type), which, in the continuous-time case, preserve Gaussian invariant measures of the discrete model.

\begin{remark}
The discretion mentioned above corresponds to the choice of $D_{x,\eps} \phi(x) = \eps^{-1}(\phi(x) - \phi(x -  \eps))$, the nearest-neighbour discrete Laplacian $\Delta_\eps \phi(x) = \eps^{-2} (\phi(x+\eps) - 2 \phi(x) + \phi(x-\eps))$ and the product
\begin{equ}
  B_\eps(\phi,\psi)(x) = \frac{1}{3} \left(\phi(x+\eps) \psi(x+\eps) + \frac{1}{2} \bigl(\phi(x) \psi(x+\eps) + \phi(x+\eps) \psi(x) \bigr) + \phi(x) \psi(x) \right)\;.
\end{equ}
Theorem \ref{t:Convergence} implies that convergence holds as $\eps \to 0$ with a suitable renormalization constant.
\end{remark}

Before proceeding to the proof of the previous theorem let us point out some of the key aspects of the present work. 

\begin{itemize}
\item To prove convergence of the discrete systems introduced above, we first have to clarify what we mean by solution for~\eqref{e:SBE} and, to do so, we solve SBE with a slightly new methodology which is applicable to a number of different subcritical singular SPDEs (the dimension plays no role whatsoever). In order to make sense of the nonlinearity in space, (and {\it only} for that!) we will exploit the most basic elements of the theory of Regularity Structures (in particular, the reconstruction theorem) while the Schauder estimates that one needs in order to close the fixed point are obtained by regarding the heat kernel as a rescaled test function, with time as the scaling parameter (see Section~\ref{sec:ContConv} for more on this). 
In particular, Lemma~\ref{l:Schwartz} and Proposition~\ref{p:HeatBounds}, that represent the core of the argument, are valid for more general space-time kernels, namely any for which one can show a sufficiently fast decay at infinity (for the purpose of this paper, we will consider as an example the off-the-grid extension of the space-time discrete heat kernel). This means that, on one side we can avoid the complicated construction carried out in~\cite[Sections 5]{Hai14}, on the other it leaves the possibility to treat time as a separate quantity, and this, in certain situations, can come at hand. 
The overall approach seems to us more direct (and similar in spirit to~\cite{GIP15}) even if definitely not as systematic as~\cite{Hai14}. 

\item The fact that we are discretizing the product through the operator $B_\eps$ will play an important role in our analysis. For two functions $f$ and $g$ on $\T_\eps$, $B_\eps$ is a {\it twisted product} given by the average of the product of $f$ and $g$ evaluated at different points (for the definition, see~\eqref{e:Dproduct}). It includes the usual pointwise product, general Zabusky-type discretizations (which are also refer to as Sasamoto-Spohn type).
As we pointed out before, in order to make sense of the ill-posed terms in the equation one has to identify a suitable subspace of the space of H\"older functions/distributions constrained by specific algebraic relations, depending on the structure of the equation itself. 
This twisted product naturally requires to modify such relations and all the constructions/definitions behind them in a non-trivial way. This is the first time in which this is done in the discrete setting in the context of regularity structures and hints at the way in which such a procedure might be carried out for actual interacting particle systems.  

\item In general, one cannot expect the product of gaussians to be invariant for the system in~\eqref{e:DiscreteSBE} for {\it any} choice of the discrete operators, and the discrete Cole-Hopf transform (known also as G\"artner transform~\cite{Gar88}) of $u^\eps$ does not seem to solve any suitable discrete version of the stochastic heat equation, so that it is not clear how one could apply these methods in order to obtain a result as Theorem~\ref{t:Convergence} in its full generality. 

\item The advantage of taking space-time discretizations lies also in the fact that such a scheme is in principle directly numerically implementable while when the discretization is only in space then another discretization step is needed. In this paper though we do not obtain explicit rates of convergence, leaving the question for future investigations. 
\end{itemize}

At last let us discuss previous works and related literature. Although discretizations in {\it both} space and time in the context of singular SPDEs and these new pathwise techniques have never been considered before, we point out that the overall strategy is similar to situations in which the discretization is only in space. In this case, we mention in particular~\cite{GP15b} where the authors, among the various, show convergence for a class of spatially discrete models to the KPZ equation, imposing assumptions on the discrete Laplacian, derivative and product very similar to the ones of the present paper. Their approach is based though on paracontrolled calculus (see also~\cite{ZZ14,ZZ15} for the space discretization of stochastic Navier-Stokes and $\Phi_3^4$ equations with similar techniques), which represents a significant difference with respect to our work. 
On the other hand, more closely related to ours is~\cite{HM15}, in which it is given a general framework to deal with space discretizations of singular SPDEs. Nevertheless, the discretization of the product there considered is only the pointwise one and it follows more closely~\cite{Hai14}. 

\begin{plan*}
In Section~\ref{sec:RS}, we recall the ingredients of the theory of Regularity Structures and its grid counterpart, we will need in order to make sense of the nonlinearity. In Section~\ref{s:Burgers}, we explain the ideas behind our approach and the procedure we follow, precisely introducing all the quantities of interest and stating all the main results. Section~\ref{sec:Analytic} contains the proof of the discrete and continuous Schauder estimates as well as the fixed point argument and Section~\ref{sec:Stochastic} is devoted to the stochastic estimates and the definition of the discrete and continuous stochastic processes that appear in the description of our solution. At last, in Section~\ref{sec:Conv}, we prove Theorem~\ref{t:Convergence}. 
\end{plan*}

\begin{ackno*}
We want to thank P.K. Friz, M. Hairer and N. Perkowski for many useful discussions, comments/suggestions and, the first two, for the financial support that allowed KM to visit the Technische Universit\"at Berlin and GC, the University of Warwick. Most of this work was carried out while GC was funded by DFG/RTG1845, that he kindly acknowledges. 
\end{ackno*}

\subsection*{Spaces and conventions}

In this section we want to introduce and recall the definition of the function spaces we will be using throughout the rest of the paper. 
We begin with some useful conventions that will make the notations lighter. 

Fix $T>0$. Let $t,\bar t\in(0,T]$, $x,\bar x\in\R$ and set $z=(t,x)$ and $\bar z=(\bar t,\bar x)$. Then, for a function $\zeta:(0,T]\times \R\to\R$, we define the increment operators in space, time and space-time respectively as
\begin{equ}[e:deltaOperator]
\delta^{(t)}_{x, \bar x} \zeta \eqdef \zeta(t, \bar x) - \zeta(t, x)\;, \quad \delta^{(t, \bar t)}_{x} \zeta \eqdef \zeta(\bar t, x) - \zeta(t, x)\,\quad \text{and}\quad \delta_{z, \bar z} \zeta \eqdef \zeta(\bar z) - \zeta(z)\;.
\end{equ}
Given $t,\,\bar t,\,x,\,\bar x\in\R$, let us introduce the following quantities
\begin{equ}
\onorm{t} \eqdef |t|^{1/2} \wedge 1\,,\qquad\onorm{t, \bar t} \eqdef \onorm{t} \wedge \onorm{\bar t}\qquad\text{and}\qquad|z-\bar z|_\s\eqdef |t-\bar t|^{1 / 2}\vee |x-\bar x|\;,
\end{equ}
where in the last we set $z=(t,x)$ and $\bar z=(\bar t, \bar x)$ and $|\cdot|_\s$ is the parabolic norm on $\R^2$. 

For $\alpha \in (0,1)$, $\eta \in \R$, we say that a function $\zeta:(0,T]\times \R\to\R$ is $\alpha$-H\"older in space with explosion rate $\eta$, briefly $\zeta\in\CC^{\alpha}_{\eta,T}$, if the quantity 
\begin{equ}[e:SpaceHolder]
\|\zeta\|_{\CC^{\alpha}_{\eta,T}}\eqdef \sup_{(t, x) \in (0,T]\times\R} \onorm{t}^{-(\eta \wedge 0)} |\zeta(t,x)|+ \sup_{t \in (0,T]} \sup_{\bar x \neq x \in \R}  \frac{|\delta^{(t)}_{x, \bar x} \zeta|}{\onorm{t}^{\eta - \alpha} | x - \bar x|^{\alpha}}
\end{equ} 
is finite. We say instead that the function is parabolic $\alpha$-H\"older continuous, i.e. $\zeta\in\CC^{\alpha,\s}_{\eta,T}$, if 
\begin{equ}[e:ParabolicHolder]
\|\zeta\|_{\CC^{\alpha,\s}_{\eta,T}}\eqdef\|\zeta\|_{\CC^{\alpha}_{\eta,T}}+\sup_{x \in \R}\sup_{\substack{\bar t \neq t \in (0,T]\\|t-\bar t|\leq\onorm{t,\bar t}^2}} \frac{|\delta^{(t, \bar t)}_{x} \zeta|}{\onorm{t, \bar t}^{\eta-\alpha}\, | t - \bar t|^{\alpha/2}}<\infty\,.
\end{equ}
In case $\zeta$ is (parabolic) $\alpha$-H\"older continuous but does not exhibit any blow-up as $t$ approaches $0$, we will simply write $\zeta\in\CC^{\alpha}_T$ (or $\CC^{\alpha,\s}_T$). The norm on the latter space is the same as before but with $\eta=0$ at the first summand of~\eqref{e:SpaceHolder}, and $\eta=\alpha$ in the second of both~\eqref{e:SpaceHolder} and~\eqref{e:ParabolicHolder}.  

For a point $z\in(0,T]\times\R$, we also introduce the following quantity
\begin{equ}[e:BracketNorm]
\pnorm{\zeta}_{\eta, \alpha; T, z} \eqdef \sup_{\substack{\bar z\neq z \in (0,T] \times \R\\|t-\bar t|\leq\onorm{t,\bar t}^2}} \frac{|\delta_{z, \bar z} \zeta|}{\onorm{t, \bar t}^{\eta}\, \snorm{\bar z - z}_\s^{\alpha}}\;,
\end{equ}
where the generic points $z,\,\bar z$ are given by $z=(t,x)$ and $\bar z=(\bar t,\bar x)$. We will use the previous seminorm also for functions on $((0,T]\times\R)^2$. If $R(\cdot,\cdot)$ is such a function, its $\pnorm{\cdot}_{\eta, \alpha; T, z}$ seminorm has the same definition as above but the increment $\delta_{z, \cdot} \zeta$ is replaced by $R(z,\cdot)$. Moreover, it is immediate to verify that  $\zeta \in \CC^{\alpha,\s}_{\eta,T}$ is equivalent to
\begin{equ}
\sup_{z \in (0,T] \times \R} \onorm{t}^{-(\eta \wedge 0)} |\zeta(z)| + \sup_{z \in (0,T] \times \R} \pnorm{\zeta}_{\eta - \alpha, \alpha; T, z} < \infty\;.
\end{equ}
where again the generic point $z$ is $(t,x)$. 

Let $\CS(\R)$ be the space of Schwartz functions and $\CS'(\R)$ its dual, i.e. the space of tempered distributions. 
In order to measure the regularity of functions on $(0,T]$ with values in $\CS'(\R)$, we introduce the spaces $\CC^{\alpha}_{\eta,T}$ and $\CC^{\alpha,\s}_{\eta,T}$, for $\alpha<0$ and $\eta\in\R$. 
$\CC^{\alpha}_{\eta,T}$ contains functions of time, $\zeta:(0,T]\to\CS'(\R)$, such that, for every $t\in(0,T]$, $\zeta(t,\cdot)$ belongs to the dual of $\CC^{r}$, with integer $r > - \lfloor \alpha \rfloor$ and 
\begin{equ}\label{e:HolderNegative}
\Vert \zeta \Vert_{\CC^\alpha_{\eta,T}} \eqdef \sup_{t\in(0,T]}\sup_{\varphi \in \CB^r_0} \sup_{x \in \R} \sup_{\lambda \in (0,1]} |t|_0^{-(\eta\wedge 0)}\lambda^{-\alpha} |\langle \zeta(t,\cdot), \varphi_x^\lambda \rangle|\;,
\end{equ}
where $\CB^r_0$ is the space of $\CC^r$ compactly supported functions whose $\CC^r$-norm is bounded by 1 and $\varphi^\lambda_x$ is the rescaled version of $\varphi\in\CB_0^r$ centered at $x\in\R$, i.e. $\varphi^\lambda_x(y)=\lambda^{-1}\varphi(\lambda^{-1}(y-x))$. In case $\eta=0$ we simply write $\CC^\alpha_T$, meaning with this the space $\CC((0,T],\CC^\alpha(\R))$. On the other hand, $\zeta\in\CC_T^{\alpha,\s}$ if $\zeta$ belongs to the dual of $\CC^{r,\s}_\infty$ (here, ``$\infty$'' refers to $T = \infty$) and
\begin{equ}\label{e:ParabolicHolderNegative}
\Vert \zeta \Vert_{\CC_T^{\alpha,\s}} \eqdef\sup_{\varphi \in \CB^{r,\s}_0} \sup_{z \in [-T,T]\times\R} \sup_{\lambda \in (0,1]} \lambda^{-\alpha} |\langle \zeta, \varphi_z^\lambda \rangle|\;,
\end{equ}
where $\CB^{r,\s}_0$ is the space of $\CC^{r,\s}_\infty$ compactly supported functions whose $\CC^{r,\s}_\infty$-norm is bounded by 1 and $\varphi^\lambda_z$ is the rescaled version of $\varphi\in\CB_0^{r,\s}$ centered at $z\in\R^2$, i.e. $\varphi^\lambda_z(\bar z)=\lambda^{-3}\varphi(\lambda^{-2}(\bar t-t), \lambda^{-1}(\bar x-x))$, where $z=(t,x)$ and $\bar z=(\bar t, \bar x)$. 

Sometimes we will work with families of distributions $\zeta_z \in \CS'(\R)$ parametrised by $z \in [-T,T]\times\R$. For $\alpha < 0$ we denote by $\CL^\alpha_T$ the space of such distributions, belonging to the dual space of $\CC^{r}_0$ for an integer $r > -\lfloor \alpha \rfloor$, and equipped with the seminorm
\begin{equ}
\Vert \zeta \Vert_{\CL^\alpha_T} \eqdef \sup_{\varphi \in \CB^r_0} \sup_{(t,x) \in [-T,T]\times\R} \sup_{\lambda \in (0,1]} \lambda^{-\alpha} |\langle \zeta_{(t,x)}, \varphi_x^\lambda \rangle|\;.
\end{equ}
For functions $\zeta_z : [-T,T]\times\R \to \R$, we will define the space $\CL^{\alpha, \s}_T$ with the seminorm 
\begin{equ}
\Vert \zeta \Vert_{\CL^{\alpha, \s}_T} \eqdef \sup_{\bar z\neq z \in [-T,T] \times \R} \frac{|\zeta_z(\bar z)|}{\snorm{\bar z - z}_\s^{\alpha}}\;.
\end{equ}

\subsubsection*{Space-time discrete norms}

We here define the analog of the norms introduced in the previous section but for functions defined on a grid. 
Let $T>0$, $N\in\N$ and set $\eps\eqdef 2^{-N}$. The discrete grids we will be working with are
\begin{equ}[e:Grids]
\Lambda_\eps\eqdef \eps\Z\,,\quad\Lambda_{\eps^2}\eqdef \eps^2\Z\,,\quad\Lambda_\eps^\s\eqdef \Lambda_{\eps^2}\times\Lambda_\eps\,,\quad \Lambda_{\eps^2,T}\eqdef\Lambda_{\eps^2}\cap(0,T]\,,\quad\Lambda_{\eps,T}^\s\eqdef\Lambda_{\eps^2,T}\times\Lambda_\eps\,.
\end{equ}
Given $t, \bar t > 0$, we define $\enorm{t} \eqdef |t|_0 \vee \eps$ and $\enorm{t, \bar t} \eqdef \enorm{t} \wedge \enorm{\bar t}$. Now, for $\alpha\in(0,1)$, $\eta\in\R$ and a function $\zeta:\Lambda_{\eps,T}^\s\to\R$, we introduce the discrete analog of~\eqref{e:SpaceHolder} and~\eqref{e:ParabolicHolder} as
\begin{equs}
\|\zeta\|_{\CC^{\alpha}_{\eta,T}}^{(\eps)} &\eqdef \sup_{(t, x) \in\Lambda_{\eps,T}^\s} \enorm{t}^{-(\eta \wedge 0)} |\zeta(t,x)|+ \sup_{t \in \Lambda_{\eps^2,T}} \sup_{\bar x \neq x \in \Lambda_\eps}  \frac{|\delta^{(t)}_{x, \bar x} \zeta|}{\enorm{t}^{\eta - \alpha} | x - \bar x|^{\alpha}}\;,\label{e:DSpaceHolder}\\
\|\zeta\|_{\CC^{\alpha,\s}_{\eta,T}}^{(\eps)} &\eqdef \|\zeta\|_{\CC^{\alpha}_{\eta,T}}^{(\eps)}+\sup_{x \in \Lambda_\eps}\sup_{\substack{\bar t \neq t \in \Lambda_{\eps^2,T}\\|t-\bar t|\leq\enorm{t,\bar t}^2}} \frac{|\delta^{(t, \bar t)}_{x} \zeta|}{\enorm{t, \bar t}^{\eta-\alpha}\, | t - \bar t|^{\alpha/2}}.\label{e:DParabolicHolder}
\end{equs}
By analogy to what has been done above, we define $\pnorm{\zeta}_{\eta, \alpha; T, z}^{(\eps)}$ according to~\eqref{e:BracketNorm}. 
For $\alpha<0$, we set 
\begin{equ}\label{e:DHolderNegative}
\Vert \zeta \Vert_{\CC^{\alpha}_{\eta,T}}^{(\eps)} \eqdef \sup_{t\in\Lambda_{\eps^2,T}}\sup_{\varphi \in \CB^r_0} \sup_{x \in \Lambda_\eps} \sup_{\lambda \in [\eps,1]} \enorm{t}^{-(\eta\wedge 0)}\lambda^{-\alpha} |\langle \zeta(t,\cdot), \varphi_x^\lambda \rangle_\eps|\;,
\end{equ}
where $\langle \cdot,\cdot\rangle_\eps$ indicates the discrete pairing, i.e. $\langle f,g\rangle_\eps=\eps\sum_{x\in\Lambda_\eps}f(x)g(x)$ and all the quantities above are the same as in the continuous case, and 
\begin{equ}\label{e:DParabolicHolderNegative}
\Vert \zeta \Vert_{\CC^{\alpha,\s}_T}^{(\eps)} \eqdef\sup_{\varphi \in \CB^{r,\s}_0} \sup_{z \in \Lambda_\eps^\s\cap [-T,T]\times\R} \sup_{\lambda \in [\eps,1]} \lambda^{-\alpha} |\langle \zeta, \varphi_z^\lambda \rangle_\eps|\;,
\end{equ}
where, in this case, $\langle \cdot,\cdot\rangle_\eps$ indicates the parabolic discrete pairing, i.e. $\langle f,g\rangle_\eps=\eps^3\sum_{z\in\Lambda_\eps^\s}f(z)g(z)$. 
For discrete families of functions $\zeta_z$ on the grid $\Lambda_\eps$ parametrized by $z\in\Lambda_\eps^\s$ we write
\begin{equ}
\Vert \zeta \Vert_{\CL^{\alpha}_T}^{(\eps)} \eqdef \sup_{\varphi \in \CB^r_0} \sup_{(t,x) \in \Lambda_\eps^\s\cap[-T,T]\times\R} \sup_{\lambda \in [\eps,1]} \lambda^{-\alpha} |\langle \zeta_{(t,x)}, \varphi_x^\lambda \rangle_\eps|\;,
\end{equ}
where $\alpha < 0$. And for $\zeta_z$ defined on $\Lambda_\eps^\s$ we introduce the quantity
\begin{equ}
\Vert \zeta \Vert_{\CL^{\alpha, \s}_T}^{(\eps)} \eqdef \sup_{\bar z\neq z \in \Lambda_\eps^\s\cap [-T,T] \times \R} \frac{|\zeta_z(\bar z)|}{\snorm{\bar z - z}_\s^{\alpha}}\;.
\end{equ}

At last, for $\alpha<0$ and $\eta\leq 0$, when comparing a discrete map, $\zeta^\eps$ on $\Lambda_{\eps,T}^\s$ (or $\Lambda_\eps$) with a map $\zeta\in\CC^{\alpha}_{\eta,T}$ (or in $\CC^\alpha$), with a slight abuse of notation, we will use the following norm
\begin{equ}[e:DCComparison]
\Vert \zeta^\eps,\zeta \Vert_{\CC^{\alpha}_{\eta,T}}^{(\eps)} \eqdef \sup_{t\in\Lambda_{\eps^2,T}}\sup_{\varphi \in \CB^r_0} \sup_{x \in \Lambda_\eps} \sup_{\lambda \in [\eps,1]} \enorm{t}^{-(\eta\wedge 0)}\lambda^{-\alpha} |\langle \zeta^\eps(t,\cdot), \varphi_x^\lambda \rangle_\eps-\langle \zeta(t,\cdot), \varphi_x^\lambda \rangle|
\end{equ}
(and, correspondingly, when $\zeta$ and $\zeta^\eps$ do not depend on time) where, inside the absolute value, $\langle\cdot,\cdot\rangle_\eps$ is the discrete pairing introduced above, while the second is the usual evaluation of a distribution on a test function. 

\begin{notation*} 
In the rest of the paper, to indicate that a sequence of maps $\zeta^\eps$ parametrized by $\eps$ satisfies $\|\zeta^\eps\|_{G}^{(\eps)}<\infty$ uniformly in $\eps$, where $G$ is any of the spaces defined above, we will write $\zeta^\eps\in G^{\eps}$. 
\end{notation*}

\section{Elements of regularity structures}\label{sec:RS}

The aim of this section is twofold. On one side we want to recall the elements of the theory of regularity structures we need in order to make sense of the ill-posed products appearing in the equation we want to treat. On the other, we will see how to modify these notions to be able to consider functions and ``objects" defined on the space-time dyadic grid. No claim of completeness is made and the interested reader is addressed to~\cite{Hai14,HM15} for the first part (for a thorough introduction see also~\cite{CW, FH14, Hai15}), and to~\cite{HM15} for the second. 

\subsection{Basic definitions and inhomogeneous models}

We begin by defining what a regularity structure is.

\begin{definition}\label{def:RS}
A {\it regularity structure} $\ST $ is a triplet $(\CA, \CT, \CG)$, in which $\CA$ is a locally finite bounded from below set of ``homogeneities" such that $0 \in \CA$, $\CT = \bigoplus_{\alpha \in \CA} \CT_\alpha$, the \textit{model space}, is a graded vector space and each of the $\CT_\alpha$ is a finite dimensional Banach space whose norm will be denoted by $\|\cdot\|$, $\CG$, the \textit{structure group}, is a set of linear transformations on $\CT$, such that for every $\Gamma~\in~\CG$, $\alpha \in \CA$ and $\tau \in \CT_\alpha$ one has
$\Gamma \tau - \tau \in \bigoplus_{\beta < \alpha} \CT_\beta$. Furthermore, we assume that $\CA \subset (-1, 1)$  and $\CT_0$ is generated by the basis vector $\one$, which represents the abstract unit. For $\tau\in\CT_\alpha$, we will denote by $|\tau| = \alpha$ the {\it homogeneity} of $\tau$.
\end{definition}

\begin{remark}\label{r:QOperators}
Given a regularity structure $\ST$ and an element $\tau\in\CT$, we will indicate by  $\CQ_\alpha \tau$ its canonical projection onto $\CT_\alpha$, and define $\| \tau \|_{\alpha} \eqdef \Vert \CQ_\alpha \tau \Vert$, i.e. the $\CT_\alpha$-norm of the component of $\tau$ in $\CT_\alpha$. We also write $\CQ_{<\alpha}$, $\CQ_{\geq \alpha}$ etc. for the projection onto $\CT_{< \alpha} \eqdef \bigoplus_{\beta < \alpha} \CT_\beta$, etc. 
\end{remark}

In practice, a regularity structure is nothing but a list of symbols constrained by specific algebraic requirements. In order to attribute them a meaning, M. Hairer introduces the notion of \textit{model} that we here recall in the variation given in~\cite[Def. 2.4]{HM15}. 

\begin{definition}\label{d:Model}
Let $\ST = (\CA, \CT, \CG)$ be a regularity structure. An ({\it inhomogeneous}) {\it model} $Z=(\Pi,\Gamma, \Sigma)$ for $\ST$ consists of three collection of maps: 
$\{\Gamma^t\}_{t \in \R}$ is such that $\Gamma^t : \R \times \R \to \CG$ and for any $x, y, z \in \R$ and $t \in \R$,
\begin{subequations}\label{ModelAlgProp}
\begin{equ}[e:PropGamma]
\Gamma^t_{x x}=\Id\,,\qquad\Gamma^t_{x y} \Gamma^t_{y z} = \Gamma^t_{x z}\,,\qquad\Gamma^t_{x y} \one = \one\;,
\end{equ}
where $\Id$ is the identity map; $\{\Sigma_x\}_{x \in \R}$ is such that $\Sigma_x:\R\times\R\to\CG$ and for all $x\in\R$ and $s,\,t,\,r\in\R$,
\begin{equ}[e:PropSigma]
\Sigma^{t t}_x =\Id\,,\qquad\Sigma^{s r}_x \Sigma^{r t}_x = \Sigma^{s t}_x\,,\qquad\Sigma^{s t}_x\Gamma^t_{x y}  = \Gamma_{x y}^s \Sigma_y^{s t}\,,\qquad\Sigma^{st}_{x} \one = \one\;,
\end{equ}
and $\{\Pi^t_x\}_{t, x \in \R}$ is such that $\Pi^t_x: \CT \to \mathcal{S}'(\R)$ is linear and for all $x, y \in \R$ and $t \in \R$, 
\begin{equ}[e:PropPi]
\Pi^t_{y} = \Pi^t_x \Gamma^t_{x y}\,,\qquad\bigl(\Pi_x^t \one\bigr)(y) = 1\,.
\end{equ} 
\end{subequations}
Moreover, for any $\gamma > 0$ and every $T > 0$, there exists a constant $C_\gamma>0$ for which the  bounds
\begin{subequations}\label{ModelAnaProp}
\begin{equ}\label{e:PiGammaBound}
| \langle \Pi^t_{x} \tau, \varphi_{x}^\lambda \rangle| \leq C_\gamma \Vert \tau \Vert \lambda^{|\tau|} \;, \qquad \Vert \Gamma^t_{x y} \tau \Vert_{m} \leq C_\gamma \Vert \tau \Vert | x-y|^{|\tau| - m}\;,
\end{equ}
\begin{equ}[e:SigmaBound]
\Vert \Sigma^{s t}_x \tau \Vert_{m} \leq C_\gamma \Vert \tau \Vert | t-s|^{(|\tau| - m) / 2}
\end{equ}
\end{subequations}
hold uniformly over all $\lambda \in (0,1]$, $\varphi \in \CB^1_0$, $|x-y| \leq 1$, $t\in[-T,T]$,  $\tau\in\CT$ with $|\tau| < \gamma$, and $m \in \CA$ such that $m < |\tau|$. 
\end{definition}
\begin{remark}\label{r:ModelNorm}
For a model $Z = (\Pi, \Gamma, \Sigma)$, we denote by $\Vert \Pi \Vert_{\gamma; T}$, $\Vert \Gamma \Vert_{\gamma;T}$ and $\Vert \Sigma \Vert_{\gamma; T}$ the smallest constant $C_\gamma$ such that the bounds on $\Pi$, $\Gamma$ and $\Sigma$ in \eqref{ModelAnaProp} hold and set $\VERT Z \VERT_{\gamma; T} \eqdef \Vert \Pi \Vert_{\gamma; T} + \Vert \Gamma \Vert_{\gamma; T}+\Vert \Sigma \Vert_{\gamma; T}$. We also define the family of semidistances between two models as
\begin{equ}
\VERT Z; \bar{Z} \VERT_{\gamma; T} \eqdef \Vert \Pi - \bar{\Pi} \Vert_{\gamma; T} + \Vert \Gamma - \bar{\Gamma} \Vert_{\gamma; T}+\Vert \Sigma - \bar{\Sigma} \Vert_{\gamma; T}\;.
\end{equ}
Even though, in general, $Z - \bar{Z}$ is not a model, notice that the expressions on the right hand side still makes sense. 
\end{remark}

 Loosely speaking, for every $x\in\R$ the map $\Pi_x^t$ assigns to every symbol a distribution that vanishes (or explodes) at the right order around $x$ (here, ``right" is understood in the sense of homogeneities) while the families of $\Gamma^t$'s and $\Sigma_x$'s guarantee that these expansions are consistently sewed together in space and time respectively. 

We can now define a space of functions taking values in a regularity structure and representing the model dependent counterpart of the space of H\"older functions (compare~\eqref{e:ModelledDistributionNorm} below with~\eqref{e:ParabolicHolder}). Let $\ST=(\CA,\CT,\CG)$ be a regularity structure and $Z=(\Pi,\Gamma, \Sigma)$ a model on $\ST$. Given $\gamma, \,\eta \in \R$, we say that a function $H : (0,T]\times\R \to \CT_{<\gamma}$ belongs to $\CD^{\gamma, \eta}_T(Z)$ if 
\begin{equ}[e:ModelledDistributionNorm]
\VERT H \VERT_{\gamma, \eta; T} \eqdef \Vert H \Vert_{\gamma, \eta; T} + \sup_{\substack{s \neq t \in (0,T]\\|t-s|\leq\onorm{t,s}^2}} \sup_{x \in \R} \sup_{l < \gamma} \frac{\Vert H_t(x) - \Sigma_x^{t s}H_{s}(x) \Vert_l}{\onorm{t, s}^{\eta - \gamma} |t - s|^{(\gamma - l)/2}} \;,
\end{equ}
where the first summand of the right hand side of~\eqref{e:ModelledDistributionNorm} is given by 
\begin{equs}[e:ModelledDistributionNormAbs]
\sup_{(t, x) \in (0,T]\times\R} \sup_{l < \gamma} \onorm{t}^{-(( \eta-l) \wedge 0)} \Vert H_t(x) \Vert_l+ \sup_{t \in (0,T]} \sup_{y \neq x \in \R} \sup_{l < \gamma} \frac{\Vert H_t(x) - \Gamma^{t}_{x y} H_t(y) \Vert_l}{\onorm{t}^{\eta - \gamma} | x - y |^{\gamma - l}}\;.
\end{equs}
Functions belonging to these spaces are called \textit{modelled distributions}. 
In order to study their continuity properties with respect to the underlying model, we will need to be able to compare modelled distributions belonging to the space $\CD_T^{\gamma,\eta}$, but based on different models. 
Let $Z=(\Pi,\Gamma, \Sigma)$, $\bar{Z}=(\bar{\Pi},\bar{\Gamma},\bar \Sigma)$ be two models on $\ST$, and $H\in\CD_T^{\gamma,\eta}(Z)$, $\bar{H}\in\CD_T^{\gamma,\eta}(\bar{Z})$ two modelled distributions, then a natural notion of distance between them can be obtained via replacing $H_t(x)$ by $H_t(x)-\bar{H}_t(x)$ in the first summand of~\eqref{e:ModelledDistributionNormAbs}, $H_t(x) - \Gamma^{t}_{x y} H_t(y)$ by
\[
H_t(x)-\bar{H}_t(x)-\Gamma^t_{x\bar{x}}H_t(\bar{x})+\bar{\Gamma}_{x\bar{x}}^t\bar{H}(\bar{x})
\]
in the second and $H_t(x) - H_{s}(x)$ by $H_t(x)-\bar{H}_t(x)-\Sigma_x^{t s}H_s(x)-\bar \Sigma_x^{t s}\bar{H}_s(x)$ in the second summand of~\eqref{e:ModelledDistributionNorm}. We denote the result by $\VERT H;\bar{H}\VERT_{\gamma,\eta;T}$, this notation being due to the fact that, as a distance, $\VERT \cdot;\cdot\VERT_{\gamma,\eta;T}$ is not a function of $H-\bar{H}$. 

A modelled distribution can be seen as the generalized abstract Taylor expansion of a given function, in which the elements appearing in its expression are not only abstract polynomials but also other ``jets" belonging to the regularity structure and whose meaning is encoded in the model. 
That said, we want to be able to associate a concrete function/distribution to a modelled distribution and this can be done via the reconstruction operator~\cite[Thm. 3.10]{Hai14}. 
The following is an adaptation of the aforementioned result, better suited to deal with inhomogeneous models, and can be found in~\cite[Thm. 2.11]{HM15}. 

\begin{theorem}\label{t:Reconstruction}
Let $\ST = (\CA, \CT, \CG)$ be a regularity structure with $\alpha \eqdef \min \CA < 0$ and $Z = (\Pi, \Gamma, \Sigma)$ be a model. Then, for every $\eta \in \R$, $\gamma > 0$ and $T > 0$, there exists a unique family of linear 
operators $\CR_t : \CD_T^{\gamma, \eta}(Z) \to \CC^{\alpha}(\R)$, parametrised by $t \in (0,T]$, 
such that the bound
\begin{equ}[e:Reconstruction]
|\langle \CR_t H_t - \Pi^t_x H_t(x), \varphi_x^\lambda \rangle| \lesssim \lambda^\gamma \onorm{t}^{\eta - \gamma} \Vert H \Vert_{\gamma, \eta; T} \Vert \Pi \Vert_{\gamma; T}\;,
\end{equ}
holds uniformly in $H \in \CD^{\gamma, \eta}_T(Z)$, $t \in (0,T]$, $x \in \R$, $\lambda \in (0,1]$ and $\varphi \in \CB^1_0$. If furthermore $\bar Z = (\bar \Pi, \bar \Gamma, \bar \Sigma)$ is another model and $\bar \CR_t : \CD_T^{\gamma, \eta}(\bar Z) \to \CC^{\alpha}(\R)$ is the associated family of operators, then the bound
\begin{equs}
|\langle \CR_t H_t - \Pi^t_x H_t(x) - \bar \CR_t \bar H_t + \bar \Pi^t_x \bar H_t(x), \varphi_x^\lambda \rangle| \lesssim \lambda^\gamma \onorm{t}^{\eta - \gamma} \bigl(\Vert H; \bar H \Vert_{\gamma, \eta; T} \Vert \Pi \Vert_{\gamma; T} + \Vert \bar H \Vert_{\gamma, \eta; T} \Vert \Pi - \bar \Pi \Vert_{\gamma; T} \bigr)\;, 
\end{equs}
holds uniformly over $H \in \CD^{\gamma, \eta}_T(Z)$, $\bar H \in \CD^{\gamma, \eta}_T(\bar Z)$ and the same parameters as above.  
\end{theorem}

The map $\CR$ introduced in Theorem~\ref{t:Reconstruction}, is the so-called {\it reconstruction operator}. We will always postulate in what follows that $\CR_t = 0$, for $t \leq 0$.

Notice that the statement above does not only guarantee that we can uniquely associate a distribution to a modelled distribution in $\CD^{\gamma}$ for $\gamma>0$, but it also tells us what such a distribution ``looks like" in a neighborhood of every point. 
Indeed, the bound~\eqref{e:Reconstruction} gives a good control over its reconstruction since the image through the model of the modelled distribution is fully explicit (after all, $\Pi$ is a \textit{linear} map on the abstract ``jets" in the regularity structure).
Moreover, in case a model $Z$ is composed of smooth functions, then, for any modelled distribution $H\in\CD^{\gamma,\eta}$, $\gamma>0$, $\CR H$ is a continuous function satisfying for all $(t,x)\in(0,T]\times\R$ the identity
\begin{equ}[e:ContRec]
\CR_t H_t(x)=\bigl(\Pi_x^tH_t(x)\bigr)(x)\;.
\end{equ}

As a first concrete application of the previous theorem, we want to recall a proposition stated and proved in \cite[Prop. 4.14]{Hai14}. 
It shows how it is possible to obtain, in this context, a classical result of harmonic analysis concerning the product of two H\"older functions/distributions. We formulate it here for functions/distributions H\"older in space with a blow-up in time as $t$ approaches 0 (see~\eqref{e:SpaceHolder}).  

\begin{proposition}\label{p:ClassicalProduct}
Let $\alpha_1,\,\alpha_2,\,\eta_1,\,\eta_2\in\R\setminus\N$ and $T>0$. Set $\alpha\eqdef\alpha_1\wedge\alpha_2$ and $\eta\eqdef(\eta_1+\alpha_2\wedge 0)\wedge(\eta_2+\alpha_1\wedge 0)\wedge(\eta_1+\eta_2)$. Then, the map $(f,g)\mapsto fg$ extends to a continuous bilinear map from $\CC^{\alpha_1}_{\eta_1,T}\times\CC^{\alpha_2}_{\eta_2,T}$ to $\CC^{\alpha}_\eta$ if and only if $\alpha_1+\alpha_2>0$ and in this case there exists a constant $C>0$ such that 
\begin{equ}
\|fg\|_{\CC^{\alpha}_{\eta,T}}\leq C \|f\|_{\CC^{\alpha_1}_{\eta_1,T}}\|g\|_{\CC^{\alpha_2}_{\eta_2,T}}\;.
\end{equ}
\end{proposition}

As we will see in the upcoming sections, what presented so far is everything we need from the whole theory of regularity structures, in order to be able to solve the equation \eqref{e:SBE}. Before delving into the details and understand how to put this into practice, we want to introduce the discrete counterpart of the objects previously defined. 


\subsection{Discrete models and modelled distributions}

Since we aim at showing the convergence of the discrete equation to the continuous one exploiting the built-in stability of these techniques, on one side we have to introduce a suitable space-time discrete notion of models, modelled distribution and reconstruction operator while on the other we need to verify that the latter satisfy properties similar to the ones enjoyed by their continuous versions. 

We will work on the dyadic grids \eqref{e:Grids} with $\eps = 2^{-N}$. While the notion of regularity structure clearly does not need to be adapted, we begin with the following. 

\begin{definition}\label{d:DModel}
Let $\ST=(\CA,\CT,\CG)$ be a regularity structure. A {\it discrete} ({\it inhomogeneous}) {\it model} $Z^\eps=(\Pi^\eps, \Gamma^\eps, \Sigma^\eps)$ consists of three collections of maps, parametrised by $(t,x) \in \Lambda_{\eps}^\s$,
\begin{equ}
\Pi_x^{\eps, t}: \CT \to \R^{\Lambda_\eps}\;,\qquad \Gamma^{\eps, t} : \Lambda_\eps \times \Lambda_\eps \to \CG\;, \qquad\text{and} \qquad \Sigma^\eps_x :\Lambda_{\eps^2}\times\Lambda_{\eps^2}\to\CG\,,
\end{equ}
satisfying the same algebraic properties as in~\eqref{ModelAlgProp}, but with the spatial and time variables restricted respectively to $\Lambda_\eps$ and $\Lambda_{\eps^2}$. 
Moreover, we require~\eqref{ModelAnaProp} to hold, with the discrete pairing $\langle\cdot,\cdot\rangle_\eps$ replacing the usual one, for $\lambda \in [\eps,1]$. 
Additionally, we impose $\bigl(\Pi^{\eps, t}_x \tau\bigr) (x) = 0$, for all $\tau \in \CT$ with $|\tau| > 0$, and all $(t, x) \in \Lambda_\eps^\s$.
At last, for $\gamma>0$ and $T>0$, we define the quantities $\|\Pi^\eps\|^{(\eps)}_{\gamma, T}$,$\|\Gamma^\eps\|^{(\eps)}_{\gamma, T}$, $\|\Sigma^\eps\|^{(\eps)}_{\gamma, T}$ and $\VERT Z^\eps\VERT^{(\eps)}_{\gamma,T}$ as well as $\VERT Z^\eps;\bar{Z}^\eps\VERT^{(\eps)}_{\gamma,T}$, analogously to Remark~\ref{r:ModelNorm}.
\end{definition}

Let $\gamma,\,\eta\in\R$, $\ST=(\CA,\CT,\CG)$ be a regularity structure and $Z^\eps=(\Pi^\eps,\Gamma^\eps, \Sigma^\eps)$ be a discrete model. 
For a function $H^\eps$ on $ \Lambda_{\eps,T}^\s$ with values in $\CT_{<\gamma}$, we define 
\begin{equs}
 \Vert H^\eps \Vert_{\gamma, \eta; T}^{(\eps)} &\eqdef \sup_{(t, x) \in \Lambda_{\eps,T}^\s}  \sup_{l < \gamma} \enorm{t}^{(l - \eta) \vee 0} \Vert H^\eps_t(x) \Vert_l+ \sup_{t \in \Lambda_{\eps^2, T}} \sup_{\substack{x \neq y \in \Lambda_\eps}} \sup_{l < \gamma} \frac{\Vert H^\eps_t(x) - \Gamma^{\eps, t}_{x y} H^\eps_t(y) \Vert_l}{\enorm{t}^{\eta - \gamma} | x - y |^{\gamma - l}}\;,\label{e:DModelledDistributionNormAbs}\\
\VERT H^\eps \VERT^{(\eps)}_{\gamma, \eta; T} &\eqdef \Vert H^\eps \Vert^{(\eps)}_{\gamma, \eta; T} + \sup_{\substack{s \neq t \in \Lambda_{\eps^2, T} \\ | t - s | \leq \enorm{t, s}^{2}}} \sup_{x \in \Lambda_\eps} \sup_{l < \gamma} \frac{\Vert H^{\eps}_t(x) - \Sigma_x^{\eps, t s} H^{\eps}_{s}(x) \Vert_l}{\enorm{t, s}^{\eta - \gamma} |t - s|^{(\gamma - l)/2}}\;,\label{e:DModelledDistributionNorm}
\end{equs}
where, in both cases $l \in \CA$. To indicate that a sequence of such maps $H^\eps$ parametrised by $\eps$ satisfies $\VERT H^\eps \VERT^{(\eps)}_{\gamma, \eta; T} < \infty$ uniformly in $\eps$, we will write $H^\eps \in \CD^{\gamma, \eta}_{\eps, T}(Z^\eps)$ and call such functions \textit{discrete modelled distributions}. We are now ready to give the following definition. 

\begin{definition}\label{d:DReconstruct}
Let $\gamma,\,\eta\in\R$, $\ST=(\CA,\CT,\CG)$ be a regularity structure and $Z^\eps=(\Pi^\eps,\Gamma^\eps, \Sigma^\eps)$ be a discrete model. For a discrete modelled distribution $H^\eps \in \CD^{\gamma, \eta}_{\eps, T}(Z^\eps)$, we define for all $ (t, x) \in \Lambda_{\eps,T}^\s$ the {\it discrete reconstruction map} $\CR^{\eps}$ by
\begin{equ}[e:DReconstructDef]
\big(\CR^{\eps}_t H^\eps_t\big)(x) \eqdef \big(\Pi_x^{\eps, t} H^\eps_t(x) \big)(x)\;.
\end{equ} 
\end{definition}

Notice that, while in the continuous case, for a smooth model relation~\eqref{e:ContRec} is a consequence of the reconstruction theorem, here we are {\it defining} the discrete reconstruction operator according to~\eqref{e:DReconstructDef} because, as we will see, such definition is well-suited for our purposes. 
Nevertheless, in the discrete setting, we have in principle much more freedom in the choice of the discrete reconstruction operator, because, after all, we only have to specify its value on a finite number of points. 
In any case, no matter how we define it, what one needs is to obtain an analog of~\eqref{e:Reconstruction} which is {\it uniform} in $\eps$ and this is what the next theorem shows.

\begin{theorem}[Thm. 4.5, \cite{HM15}]
\label{t:DReconstruct}
Let $\gamma,\,\eta\in\R$, $\ST=(\CA,\CT,\CG)$ be a regularity structure with $\alpha \eqdef \min \CA < 0$ and $Z^\eps=(\Pi^\eps,\Gamma^\eps, \Sigma^\eps)$ be a discrete model. Then the bound
\begin{equ}\label{e:DReconstruction}
|\langle \CR^\eps_t H^\eps_t - \Pi^{\eps, t}_x H^\eps_t(x), \varphi_x^\lambda \rangle_\eps| \lesssim \lambda^\gamma \enorm{t}^{\eta - \gamma} \Vert H^\eps \Vert^{(\eps)}_{\gamma, \eta; T} \Vert \Pi^\eps \Vert^{(\eps)}_{\gamma; T}\;,
\end{equ}
holds for all $H^\eps \in \CD_{\eps, T}^{\gamma, \eta}(Z^\eps)$, locally uniformly over $(t, x) \in \Lambda_{\eps,T}^\s$, $\varphi \in \CB^1_0$ and uniformly in  $\lambda \in [\eps, 1]$, and the hidden constant does not depend on $\eps$. If furthermore $\bar Z^\eps = (\bar \Pi^\eps, \bar \Gamma^\eps, \bar \Sigma^\eps)$ is another model and $\bar \CR_t^\eps : \CD_{\eps,T}^{\gamma, \eta}(\bar Z) \to \CC^{\alpha}(\R)$ is again defined according to Definition~\ref{d:DReconstruct}, then the bound
\begin{equs}
|\langle \CR_t^\eps H_t^\eps - \Pi^{\eps,t}_x &H_t^\eps(x) - \bar \CR_t^\eps \bar H_t^\eps + \bar \Pi^{\eps,t}_x \bar H_t^\eps(x), \varphi_x^\lambda \rangle|_\eps \\
&\lesssim \lambda^\gamma \onorm{t}^{\eta - \gamma} \bigl(\Vert H^\eps; \bar H^\eps \Vert_{\gamma, \eta; T}^{(\eps)} \Vert \Pi^\eps \Vert_{\gamma; T}^{(\eps)} + \Vert \bar H^\eps \Vert_{\gamma, \eta; T}^{(\eps)} \Vert \Pi^\eps - \bar \Pi^\eps \Vert_{\gamma; T}^{(\eps)} \bigr) \label{e:DContReconstruct}
\end{equs}
holds uniformly over $H^\eps \in \CD^{\gamma, \eta}_{\eps,T}(Z^\eps)$, $\bar H^\eps \in \CD^{\gamma, \eta}_{\eps,T}(\bar Z^\eps)$ and the same parameters as above.
\end{theorem}


\section{Analysis of the continuous and discrete equations}
\label{s:Burgers}

As remarked in the introduction, the difficulty in making sense of the KPZ equation comes from the fact that the space-time white noise, as a random distribution, is extremely singular and consequently the expected regularity of the solution does not allow to classically define the non-linear term appearing on the right hand side of~\eqref{e:KPZ}. 
Instead of the latter, we will focus on the SBE given by 
\begin{equ}
\partial_t u = \Delta u+\partial_x u^2 + \partial_x \xi\;, \qquad u(0) = u_0\;,
\end{equ}
where $\xi$ is the space-time white noise on the one-dimensional torus and $u_0$ the initial condition. 
We will mainly work with the mild formulation of the previous, so let $P:\R^2\to\R$ be 
the Green's function of $\partial_t-\Delta$, which, for $t > 0$, is given by $P_t(x) = \frac{1}{\sqrt{4\pi t}} e^{-\frac{x^2}{4 t}}$, and is $0$ for $t\leq0$. We mollify the noise $\xi$, i.e. we consider $\xi_\eps\eqdef \xi\ast\varrho_\eps$,  where $\varrho$ is a symmetric compactly supported smooth function integrating to 1 and $\varrho_\eps(t,x)\eqdef \eps^{-3}\varrho(\eps^{-2} t, \eps^{-1} x)$ its rescaled version, and write 
\begin{equation}\label{e:SBEMild}
u_\eps = P_t u_0 + P'\ast(u_\eps^2)+ P'\ast \xi_\eps\;,
\end{equation}
where $P' \eqdef \partial_x P$, the first summand on the right hand side is a purely spatial convolution and $\ast$ denotes the convolution in space-time.  

The idea developed in~\cite{Hai13, Reloaded} (and~\cite{Hai14}) is then to split the analysis of the equation \eqref{e:SBE} in two distinct modules. 
At first one infers from~\eqref{e:SBEMild} a minimal set of processes, the \textit{controlling processes} $\X$,  built from $\xi$, and  \textit{postulates} that there is a way, not only to properly define them but also to prove that they satisfy certain regularity requirements. 
Then, one identifies a suitable subspace of the space of distributions, depending on such processes, for which it is possible to make sense of the ill-posed operations and formulates a fixed point map that is continuous in these data. 
At last, one exploits stochastic calculus techniques to show that it is indeed possible to construct $\X$ starting from a space-time white noise. 

In this section we want to accomplish two goals. First, we will carefully carry out the program outlined above in the context of~\eqref{e:SBE}, introducing all the quantities of interest.  
Then we will present the family of discrete models we want to deal with, explaining the procedure to follow in order to prove their convergence to the solution of \eqref{e:SBE}, and precisely state all the main results, postponing their proofs to the subsequent sections.


\subsection{The controlling processes}\label{sec:ControllingProcesses}

For reasons that will soon be clarified, we need to split the heat kernel $P$ in a ``singular" and a ``regular" part. To do so, we consider a smooth compactly supported function $\chi:\R^2\to[0,1]$ such that $\chi(z)=1$ for $\|z\|_\s\leq \frac{1}{2}$ and $\chi(z)=0$ for $\|z\|_\s> 1$ (of course the choice of $1$ is completely arbitrary). Then we set, for $z\in\R^2$
\begin{equation}\label{def:Kernel}
K (z)=\chi(z)P(z)\qquad\text{and}\qquad \hat K(z)= (1-\chi(z))P(z)
\end{equation}
so that $P=K+\hat K$. 
Clearly $K$ is compactly supported, smooth except at $0$ and coincides with $P$ in a neighborhood of the origin, while $\hat K$ is even, smooth and, when dealing with convolutions of $\hat K$ with spatially periodic distributions on the time interval $[-1,1]$, we can assume $\hat K$ to be compactly supported (see~\cite[Lem. 7.7]{Hai14} for the precise statement and the proof of this claim). 
As for $P$, we will indicate by $K'$ (resp. $\hat K'$) the spatial derivative of $K$ (resp. $\hat K$).
\newline

Let us try to heuristically understand how to identify what the above mentioned controlling processes in our case are. Recall that we aim at solving \eqref{e:SBE} in a suitable function space of H\"older type and, looking at~\eqref{e:SBEMild}, it is reasonable to expect that the main contribution to the regularity of its solution, $u$, will come from the singular part of the stochastic convolution $P'\ast\xi$, i.e. $K'\ast \xi$, since what is left, no matter how (ir-)regular $\xi$ is, is smooth. 
Hence, if in terms of regularity we can presume that $u$ ``looks like" $K'\ast \xi$ at small scales, briefly, $u\sim K'\ast\xi$, then the ill-posed product at the right hand side should satisfy $u^2\sim(K'\ast \xi)^2$.
At this point, we can further detail the expansion of $u$ beyond the first term. Indeed, referring again to~\eqref{e:SBEMild}, we see that if $u^2\sim(K'\ast \xi)^2$ then $u\sim K'\ast\xi + K'\ast(K'\ast \xi)^2$ and so on.  

The point here is that it suffices to pursue such an iteration only finitely many times. Let us be more precise. Replace the noise $\xi$ by a smooth function $\eta$ and, given $X^{\<tree1>}(\eta) = K' \ast \eta$ and two constants $a,\,b\in\R$, we define the following 
\begin{subequations}\label{Control}
\begin{align}
&X^{\<tree11>}(\eta) = K' \ast X^{\<tree1>}(\eta)\;,&  
&X^{\<tree21>}(\eta) = X^{\<tree11>}(\eta) X^{\<tree1>}(\eta)-b\;,& \notag\\
&X^{\<tree2>}(\eta) = \bigl( X^{\<tree1>}(\eta) \bigr)^2-a\;, &
&X^{\<tree12>}(\eta) = K' \ast X^{\<tree2>}(\eta)\;,&\label{e:Control1}\\
&&\text{as well as              }\qquad\qquad&&\notag\\
&X^{\<tree22>}(\eta) = X^{\<tree12>}(\eta) X^{\<tree1>}(\eta) - 2 b\, X^{\<tree1>}(\eta)\;, &
&X^{\<tree122>}(\eta) = K' \ast X^{\<tree22>}(\eta)\;,&\label{e:Control2}\\
&X^{\<tree124>}(\eta) = P'\ast\bigl(X^{\<tree12>}(\eta)\bigr)^2\;,& 
&X^{\<tree1222>}(\eta) = P'\ast\bigl(X^{\<tree122>}(\eta) X^{\<tree1>}(\eta) - b\,X^{\<tree12>}(\eta)\bigr)\;.&\notag
\end{align}
At last, we introduce, for $t, x, y \in \R$ and $z, \bar z \in \R^2$, the remainders
\begin{align}\label{e:Remainders}
R^{\<tree21>}(t, x; y) &\eqdef X^{\<tree21>}(t, y) - X^{\<tree11>}(t,x)\, X^{\<tree1>} (t,y)\;,\\
R^{\<tree1222>}(z; \bar z) &\eqdef X^{\<tree1222>} (\bar z) - X^{\<tree122>}(z)\, P'\ast X^{\<tree1>} (\bar z)\;,\notag
\end{align}
\end{subequations}
where we omitted the dependence on $\eta$ not to clutter the presentation. 

\begin{remark}\label{rem:Ren}
The reason why we introduced certain constants in the expression of some of the processes above is that, when $\eta$ is the space-time white noise, the products there appearing are ill-posed. 
To be more precise, if $\eta = \xi^\eps$ is a smooth approximation of $\xi$, then in the limit as $\eps$ goes to 0, these processes would diverge to $\infty$ and the only way to prevent them from blowing up is through a suitable renormalization procedure, consisting in surgically remove the divergences. 
As we will see in the following sections, in our context, this simply amounts to subtract the 0-chaos component from the Wiener-chaos expansion of the stochastic process at hand. 
\end{remark}

\begin{remark}\label{rem:Ren2}
In principle, one would expect the presence of a renormalisation constant in the definition of $X^{\<tree124>}$ and  $X^{\<tree1222>}$ as well, but these constants are killed by the convolution with the spatial derivative of the heat kernel and therefore we can forget about them. 
This is the advantage of working with SBE instead of KPZ.
\end{remark}

As pointed out in the introduction, we want to show not only that we can construct these objects in the case in which $\eta$ is the space-time white noise but we also need to prove that they have certain space-time regularities. 
Recall that, as a random distribution, the realizations of the space-time white noise belong almost surely to $\CC^{\alpha}_\s$ for any $\alpha<-\frac{3}{2}$ (see for example the proof of~\cite[Prop. 9.5]{Hai14}). Starting from this, one can guess what is the expected regularity of the terms in~\eqref{Control} by applying the following rules:
\begin{enumerate}[noitemsep, leftmargin=*]
\item by usual Schauder estimates, the convolution with the heat kernel $P$ (and, of course, with its singular part $K$) increases the regularity by 2, and consequently the one with $P'$ (resp. $K'$) by 1;
\item for $\beta,\,\gamma\in\R$, the product of two distributions in $\CC^\beta$ and $\CC^\gamma$ respectively belongs to $\CC^\delta$ where $\delta=\min\{\beta,\,\gamma,\,\beta+\gamma\}$ (see \cite[Thms. 2.82, 2.85, Prop. 2.71]{BCD11} or simply Proposition~\ref{p:ClassicalProduct});
\item since we are subtracting the ill-posed part of the product (i.e. the ``diagonal"), the regularity of the terms in~\eqref{e:Remainders} will be given by the sum of the regularities of the factors.
\end{enumerate}
Even if the second property \textit{analytically} holds provided that $\beta+\gamma>0$, we will simply \textit{postulate} that this is still the case (and this is the point in which the renormalization will enter the game, see Remark~\ref{rem:Ren}). 

In order to lighten the notation, and following the properties 1.-3. stated above, we assign numbers $\alpha_\tau$ and $\beta_{\bar \tau}$, representing the regularities of $X^\tau$ and $R^{\bar \tau}$ respectively, to each label $\tau \in \SL \eqdef \{\<tree1>, \<tree2>, \<tree11>, \<tree21>, \<tree12>, \<tree22>, \<tree122>, \<tree124>, \<tree1222>\}$ and $\bar \tau \in \SL_R \eqdef \{\<tree21>, \<tree1222>\}$. We do it in a recurrent way: 
\begin{itemize}[noitemsep, leftmargin=*]
\item we fix a parameter $\alpha_{\<tree1>}$;
\item we set $\alpha_{\tau} = \sum_{i} \alpha_{\tau_i} \wedge 0 + 1$, in case the forest $\{\tau_i\}$ is obtained by removing the root and the adjacent edges from a tree~$\tau$;
\item we set $\alpha_{\tau} = \alpha_{\tau_1} \wedge 0 + \alpha_{\tau_2} \wedge 0$ and $\beta_{\tau} = \alpha_{\tau_1} + \alpha_{\tau_2}$, for a label $\tau = \tau_1 \tau_2$, with $\tau_i$ being a tree.
\end{itemize}
With these notations at hand we are ready to give the following definition.
\begin{definition}[Controlling Processes]\label{def:ControlProc}
Let $\alpha_{\<tree1>}\in \bigl(-\frac{3}{5}, -\frac{1}{2}\bigr)$ and $\alpha_\tau$, $\beta_{\bar \tau}$, for $\tau\in\SL$ and $\bar \tau\in\SL_R$, be given as above. 
For $(\eta,\,a,\,b)\in \CC^{\infty}(\R^2)\times\R^2$, we set $\X(\eta)\eqdef\X(\eta,a,b)$ to be 
\[
\X(\eta)=\left(\eta, X^{\<tree2>},  X^{\<tree22>}, X^{\<tree11>}, X^{\<tree1>}, X^{\<tree21>}, X^{\<tree12>}, X^{\<tree122>}, X^{\<tree124>}, X^{\<tree1222>}, R^{\<tree21>}, R^{\<tree1222>}\right)(\eta)\;,
\]
where the elements appearing in the previous are given by~\eqref{Control} (and we hid the dependence on the constants $a,\,b$ for the processes $X^{\<tree2>},\,X^{\<tree21>}$ to lighten the notation). Then we define the space $\CX$ of \textit{controlling processes} as 
\[
\CX\eqdef \text{cl}_\CW\left\{\X(\eta, a,b)\,:\, (\eta,\,a,\,b)\in \CC^{\infty}(\R^2)\times\R^2\right\}\;,
\]
where $\text{cl}_\CW\{\cdot\}$ denotes the closure of the set in brackets with respect to the topology of $\CW$ and 
\[
\CW \eqdef \CC_1^{\alpha_{\<tree1>} - 1,\s} \oplus \Bigl( \bigoplus_{\tau \in \SL_1} \CC^{\alpha_\tau,\s}_1 \Bigr)\oplus \Bigl( \bigoplus_{\tau \in \SL_2} \CC^{\alpha_\tau}_1 \Bigr)\oplus\Bigl( \bigoplus_{\tau \in \SL_3} \CC^{\alpha_\tau,\s}_1 \Bigr)\oplus \CL_1^{\beta_{\<tree21>}} \oplus \CL_1^{\beta_{\<tree1222>}, \s}
\] 
endowed with the usual norm, where $\SL_1\eqdef\{\<tree2>,\,\<tree22>,\,\<tree11>\}$, $\SL_2\eqdef\{\<tree1>,\,\<tree21>,\,\<tree12>,\,\<tree122>\}$ and $\SL_3\{\<tree124>,\,\<tree1222>\}$. 
We denote by $\X$ a generic element of this space and, if $\eta\in\CC^{\alpha_{\<tree1>} - 1,\s}$ coincides with the first component of $\X\in\CX$, we say that $\X$ is an {\it enhancement} (or {\it lift}) of $\eta$. 
\end{definition}
\begin{center}
\begin{equ}
\begin{array}{cc}
\toprule
\mathbf{\alpha_\tau} & \textbf{reg.}\\
\midrule
\alpha_{\<tree1>} & -\frac{1}{2}^{-}\\
\alpha_{\<tree2>} & -1^{-}\\
\alpha_{\<tree11>} & \frac{1}{2}^{-}\\
\bottomrule
\end{array}\qquad\quad
\begin{array}{cc}
\toprule
\mathbf{\alpha_\tau} & \textbf{reg.}\\
\midrule
\alpha_{\<tree21>} & 0^{-}\\
\alpha_{\<tree12>} & 0^{-}\\
\alpha_{\<tree22>} & -\frac{1}{2}^{-}\\
\bottomrule
\end{array}\qquad\quad
\begin{array}{cc}
\toprule
\mathbf{\alpha_\tau} & \textbf{reg.}\\
\midrule
\alpha_{\<tree122>} & \frac{1}{2}^{-}\\
\alpha_{\<tree124>} & 1^{-}\\
\alpha_{\<tree1222>} & \frac{1}{2}^{-}\\
\bottomrule
\end{array}\qquad\quad
\begin{array}{cc}
\toprule
\mathbf{\beta_{\bar\tau}} & \textbf{reg.}\\
\midrule
\beta_{\<tree21>} & 0^{-}\\
\beta_{\<tree1222>} & 1^{-}\\
\bottomrule
\end{array}
\end{equ}
\captionof{table}{Summary of the values of the $\alpha_\tau$ and $\beta_{\bar \tau}$ for $\tau\in\SL$ and $\bar\tau\in\SL_R$ respectively. Given $\alpha\in\R$, $\alpha^-$ stands for $\alpha-\eps$ for any $\eps>0$ small enough. }\label{table:reg}
\end{center}

\begin{remark}
Some of the objects in $\X(\eta)$ have low regularity and are not functions in the time variable, e.g. $\eta$, $\<tree2>$ and $\<tree22>$, that's why they belong to the respective spaces of space-time distributions. On the other hand, we prefer to consider some objects with positive regularities, e.g. $\<tree11>$, $\<tree124>$ and $\,\<tree1222>$, as H\"{o}lder functions in space-time, which is more convenient when working with regularity structures in Section~\ref{sec:HeuCont}.
\end{remark}

The following proposition, whose proof is provided in Section~\ref{sec:control}, states that it is indeed possible to enhance the space-time white noise. 

\begin{proposition}\label{p:Control}
Let $\xi$ be a space-time white noise on a probability space $(\Omega,\SF,\Prob)$, $\varrho$ a symmetric compactly supported smooth function integrating to 1 and $\varrho_\eps(t,x)\eqdef \eps^{-3}\varrho(\eps^{-2} t, \eps^{-1} x)$ its rescaled version. Set $\xi_\eps\eqdef \xi\ast\varrho_\eps$. Then, there exists a sequence of constants $C_\eps^{\<tree2>} \sim \eps^{-1}$ and a controlling process $\X(\xi)\in\CX$, such that $\X(\xi_\eps, C_\eps^{\<tree2>}, 0)$ converges to $\X(\xi)$ in $L^p(\Omega,\CX)$ for all $p>1$, and moreover $\X(\xi)$ is independent of the choice of the mollifier $\varrho$.
\end{proposition}


\subsection{The ill-posed product, regularity structure and a notion of solution}\label{sec:HeuCont}

Thanks to Proposition~\ref{p:Control}, we know how to construct a number of stochastic processes, closely related to the equation~\eqref{e:SBE}, starting from a space-time white noise $\xi$. 
Building on the ideas sketched at the beginning of the previous section (see also~\cite{Hai13, Reloaded}), let $\xi_\eps$ be given as before and $u_\eps$ be the solution of \eqref{e:SBE} with $\xi_\eps$ replacing $\xi$. We want to proceed to a Wild expansion of $u_\eps$ around the solution of the linearized equation, or more precisely, around the singular part of it. 
Indeed, we expect that, upon subtracting sufficiently many irregular terms, what remains has better regularity properties allowing, on one side, to make sense of the ill-posed product (via the theory of regularity structures) and, on the other, to define a map that can be proved to have a unique fixed point in a suitable space.

To be more concrete, let $\X(\xi_\eps, 0, 0)\in\CX$ be an enhancement of $\xi_\eps$ according to Definition~\ref{def:ControlProc} without renormalization and we define $v_\eps$ as
\begin{equ}[e:uExp]
v_\eps \eqdef u_\eps - X^{\<tree1>}_\eps - X^{\<tree12>}_\eps - 2 X^{\<tree122>}_\eps\;,
\end{equ}
where $u_\eps$ is as above (i.e. the solution to \eqref{e:SBE} driven by the mollified noise $\xi_\eps$). Then, $v_\eps$ solves the following stochastic PDE (written in its mild formulation), with $u_0$ as initial condition,
\begin{equ}\label{e:TruncBurgers}
v_\eps = P u_0 + 4 X^{\<tree1222>}_\eps+P' \ast \Bigl(  2 v_\eps X^{\<tree1>}_\eps +  F^\eps_{v_\eps}\Bigr) +  Q^\eps\;,
\end{equ}
with $P u_0 \eqdef \bigl(P_t \ast u_0\bigr)(x)$, the convolution involving only the space-variable, and 
\begin{equ}[e:F]
 F^\eps_{v_\eps} \eqdef 2 X^{\<tree12>}_\eps \bigl(2 X^{\<tree122>}_\eps + v_\eps \bigr) + \bigl( 2 X^{\<tree122>}_\eps + v_\eps\bigr)^2\;, \qquad  Q^\eps \eqdef X^{\<tree124>}_\eps  +\hat K' \ast \Bigl( \xi_\eps + X^{\<tree2>}_\eps + 2 X^{\<tree22>}_\eps \Bigr)\;,
\end{equ}
where $\hat K$ is the smooth part of the heat kernel introduced in~\eqref{def:Kernel}. 

We now want to pass to the limit as $\eps$ goes to $0$. Since Proposition~\ref{p:Control} will take care of the terms belonging to $\X(\xi_\eps)$, we can focus on the others.  
In the $\eps$-limit, the regularity of $v_\eps$, $\alpha_\star$ cannot be better than the one of $X^{\<tree1222>}_\eps$, which is $\alpha_{\<tree1222>}<\frac{1}{2}$. 
Assuming also $\alpha_\star > -\alpha_{\<tree12>}$ and thanks to the bounds on $\alpha_{\<tree1>}$, one easily sees that all the summands in~\eqref{e:TruncBurgers} and~\eqref{e:F} are analytically well-defined (all the products fall into the scope of Proposition~\ref{p:ClassicalProduct}) apart from $v_\eps X_\eps^{\<tree1>}$, for which we will need the theory of regularity structures.  

To take a glimpse at the idea behind our approach (and, more generally, the rough path/regularity structures approach), we notice that if we look at the small increments of the mild solution to the equation \eqref{e:TruncBurgers}, then, at least formally, they should behave as the ones of the process $X^{\<tree11>}_\eps$. More precisely, it is reasonable to expect the relation
\begin{equ}[e:VExpansionEps]
 \delta_{z, \bar z} v_\eps = v'_\eps(z)\, \delta_{z, \bar z} X_\eps^{\<tree11>} + R_\eps(z, \bar z)
\end{equ}
to hold, where $v'_\eps = 4 X_\eps^{\<tree122>} + 2 v_\eps$, and $R_\eps$ has ``better'' regularity than $X_\eps^{\<tree11>}$.

This formal relation hints at the way in which the regularity structure should be defined in order to be able to encapsulate a suitable description of the process $v_\eps$ and consequently to make sense of the ill-posed term in~\eqref{e:TruncBurgers} in the limit.

To this purpose we now define a regularity structure $(\CA, \CT, \CG)$ such that each controlling process $\X \in \CX$ gives raise to a model on it. 
Let $\alpha_\tau$, $\beta_{\bar \tau}$ be the parameters introduced in the previous section. Then we set $\CA\eqdef \{\alpha_{\<tree1>},\,\beta_{\<tree21>},\,0,\,\alpha_{\<tree11>}  \}$, the model space $\CT$ will be $\CT \eqdef \CT_{\alpha_{\<tree1>}} \oplus \CT_{\beta_{\<tree21>}} \oplus \CT_{0} \oplus \CT_{\alpha_{\<tree11>}}$, where the Banach spaces $\CT_{\alpha_{\<tree1>}}$, $\CT_{\beta_{\<tree21>}}$, $\CT_{0}$ and $\CT_{\alpha_{\<tree11>}}$ are copies of $\R$ with the unit vectors $\<bigtree1>$, $\<bigtree21>$, $\one$ and $\<bigtree11>$ respectively. 

Let $\X\in\CX$, then we introduce the model $Z=(\Pi, \Gamma,\Sigma)$ given by $\Pi^{t}_x \one = 1$ as well as
\begin{subequations}\label{e:Model}
\begin{equ}\label{e:ModelPi}
\bigl( \Pi^{t}_x \<bigtree1>\bigr)(\cdot) = X^{\<tree1>}(t,\cdot)\;, \qquad \bigl( \Pi^{t}_x \<bigtree21>\bigr)(\cdot) = R^{\<tree21>}(t,x; \cdot)\;, \qquad \bigl( \Pi^{t}_x \<bigtree11>\bigr)(\cdot) = \delta^{(t)}_{x, \cdot} X^{\<tree11>}\;.
\end{equ}
For any points $s, t, x, y \in \R$, the maps $\Gamma^t$ and $\Sigma_x$ are defined as $\Gamma^{t}_{xy} \one = \one$, $\Sigma^{s t}_{x} \one = \one$ and
\begin{align}
&\Gamma^{t}_{xy} \<bigtree1> = \<bigtree1>\;,&  &\Gamma^{t}_{xy} \<bigtree21> = \<bigtree21> + \delta^{(t)}_{x, y} X^{\<tree11>}\, \<bigtree1>\;,& & \Gamma^{t}_{xy} \<bigtree11> = \<bigtree11> + \delta^{(t)}_{x, y} X^{\<tree11>}\, \one\;,&\label{e:ModelGamma}\\
&\Sigma^{s t}_{x} \<bigtree1> = \<bigtree1>\;,&  &\Sigma^{s t}_{x} \<bigtree21> = \<bigtree21> + \delta^{(s, t)}_{x} X^{\<tree11>}\, \<bigtree1>\;,& & \Sigma^{s t}_{x} \<bigtree11> = \<bigtree11> + \delta^{(s, t)}_{x} X^{\<tree11>}\, \one\;.&\label{e:ModelSigma}
\end{align}
\end{subequations}
It follows immediately from the definition of the space $\CX$ that the maps $(\Pi, \Gamma,\Sigma)$ have all the algebraic and analytic properties required in Definition~\ref{d:Model}.
Furthermore, one can define a structure group $\CG$ in such a way that the operators $\Gamma^{t}_{xy}$ and $\Sigma_x^{st}$ belong to $\CG$, see \cite{Hai14, HM15}.

Motivated by the expansion \eqref{e:VExpansionEps}, we can define the modelled distributions, which describe $v$ and $v X^{\<tree1>}$ at the abstract level. 
More precisely, given two functions $v, v' : (0,T]\times\T \rightarrow \R$, we set 
\begin{equ}[e:V]
\BV_t(x)\eqdef v(t,x)\, \one + v'(t,x)\,\<bigtree11> 
\end{equ}
and consequently
\begin{equ}[e:VDot]
\big(\BV \<bigtree1>\big)_t(x) \eqdef v(t,x) \,\<bigtree1> + v'(t,x)\, \<bigtree21>\;.
\end{equ}
which corresponds to the product between the modelled distributions $\BV$ and $\<bigtree1>_t(x)\eqdef 1\,\<bigtree1>$. 
The idea is now to define the product $v X^{\<tree1>}$ in \eqref{e:TruncBurgers} as $\CR \big(\BV \<bigtree1>\big)$, where $\CR$ is the reconstruction operator associated to the model $Z=(\Pi, \Gamma, \Sigma)$. 
Indeed, provided that $\BV \<bigtree1>$ belongs to a space $\CD^{\gamma,\eta}_T(Z)$ for some $\gamma>0$, Theorem~\ref{t:Reconstruction} guarantees that this quantity is indeed well-defined.
\newline

In order to verify this latter statement and define a fixed point map, we need to introduce a suitable function space.
To this end, we fix another global parameter $\alpha_\star > 0$ which, as before, represents the regularity of the solution $v$. 
Now, given a triplet $V = \bigl(v, v', R\bigr)$ of smooth functions, $v, v'$ on $(0,T]\times\R$ and $R$ on $((0,T]\times\R)^2$, we define, for $\gamma > \alpha_{\<tree11>}$ and $\eta \in \R$ the seminorm
\begin{equ}[e:Norms]
\Vert V \Vert_{\eta, \gamma; T} \eqdef \Vert v\Vert _{\CC^{\alpha_\star,\s}_{\eta, T}} + \Vert v' \Vert _{\CC^{\gamma - \alpha_{\<tree11>},\s}_{\eta - \alpha_{\<tree11>}, T}} + \sup_{z \in (0,T]\times\R} \pnorm{R(z, \cdot)}_{\eta - \gamma, \gamma; T, z}\;,
\end{equ}
and we denote by $\CH^{\eta, \gamma}_{T}$ the completion of such smooth triplets under this seminorm. Then we define the space to which the solution will belong. 
\begin{definition}\label{d:CP} 
Let $\alpha_{\star}>0$, $\gamma>-\alpha_{\<tree1>}$, $\eta\in(-1,0)$ and let $\CX$ be the space of controlling processes as in Definition~\ref{def:ControlProc}. Then we define the space of {\it controlled processes} $\CX^{\eta, \gamma}_{T} \subset \mathcal{X} \oplus \CH^{\eta, \gamma}_{T}$ as the algebraic variety of $\X \in \CX$ and $V=\bigl(v, v', R\bigr) \in \CH^{\eta, \gamma}_{T}$ satisfying the identity
\begin{equ}[e:VExpansion]
\delta_{z, \bar z} v = v'(z)\, \delta_{z, \bar z} X^{\<tree11>} + R(z, \bar z)\;.
\end{equ}
Moreover, we define the set of all processes controlled by $\X \in \CX$ as
\begin{equ}
\CH^{\eta, \gamma}_{\X,T} \eqdef \Bigl\{ V \in \CH^{\eta, \gamma}_{T} : \bigl(\X, V\bigr) \in \CX^{\eta, \gamma}_{T} \Bigr\}\;,
\end{equ}
endowed with the norm \eqref{e:Norms}. 
\end{definition}
It is now immediate to verify that, given any controlling process in $\X\in\CX$ and any process $V$ controlled by $\X$, there is a unique way to define the ill-posed product as the reconstruction of the modelled distribution given in~\eqref{e:VDot}.

\begin{proposition}\label{p:ModelledDistribution}
Let $\alpha_{\star}>0$, $\gamma>-\alpha_{\<tree1>}$ and $\eta\in(-1,0)$. Let $\X\in\CX$ and $Z=(\Pi,\Gamma,\Sigma)$ be the model given in~\eqref{e:Model}. Let $V=(v,v',R)\in\CH^{\gamma,\eta}_{\X, T}$ and $\BV$, $\BV\<bigtree1>$ be defined as in~\eqref{e:V} and~\eqref{e:VDot} respectively. 
Then, $\BV\in\CD^{\gamma,\eta}_T(Z)$, $\BV\<bigtree1>\in\CD^{\gamma+\alpha_{\<tree1>},\eta+\alpha_{\<tree1>}}_T(Z)$ and, for all $t\in(0,T]$, $\CR_t(\BV\<bigtree1>)\in\CC^{\alpha_{\<tree1>}}$ is uniquely defined. 
Moreover, the map assigning to every couple  $(\X,V)$ in $\CX_T^{\eta,\gamma}$ the reconstruction of $\BV\<bigtree1>$ is jointly locally Lipschitz continuous. 

At last, if $\X(\eta)\eqdef\X(\eta,a,b)$, with $a,\,b\in\R$, is an enhancement of a smooth function $\eta$, then for all $(t,x)\in(0,T]\times\R$ the following equality holds 
\begin{equ}[e:SmoothReconstruction]
\CR(\BV\<bigtree1>)_t(x)=v(t,x)X^{\<tree1>}(t,x)-bv'(t,x)\;.
\end{equ} 
\end{proposition}

\begin{remark}
The reason why in the actual equation the constant $b$ will not appear is that, thanks to Proposition~\ref{p:Control}, we know that $b=0$ in case $\eta$ is a smooth and symmetric approximation of the space-time white noise. 
\end{remark}
\begin{proof}
There is very little to prove. Indeed, the fact that $\BV\in\CD^{\gamma,\eta}_T(Z)$ follows by 
\[
\VERT\BV\VERT_{\gamma,\eta;T}=\|v\|_{\CC^0_{\eta;T}}+\|v'\|_{\CC^{\gamma-\alpha_{\<tree11>},\s}_{\eta-\alpha_{\<tree11>};T}}+\sup_{z \in (0,T]\times\R} \pnorm{R(z, \cdot)}_{\eta - \gamma, \gamma; T, z}\leq \| V\|_{\gamma,\eta;T}\;.
\]
Notice that the function $(0,T]\times\R\ni (t,x)\mapsto \<bigtree1>$ is a modelled distribution in $\CD^{\bar\gamma,\bar\gamma}$ for every $\bar \gamma>0$ whose element of minimum homogeneity is $\<bigtree1>$. Therefore,~\cite[Prop. 6.12]{Hai14} implies that $\BV\<bigtree1>\in\CD^{\gamma+\alpha_{\<tree1>},\eta+\alpha_{\<tree1>}}_T(Z)$ (alternatively one could prove this by direct computation). 
By assumption, $\gamma+\alpha_{\<tree1>}>0$ hence, Theorem~\ref{t:Reconstruction} guarantees that, for every $t\in(0,T]$, $\CR_t(\BV\<bigtree1>)$ is unique and that the map assigning to every couple  $(\X,V)$ in $\CX_T^{\eta,\gamma}$ the reconstruction of $\BV\<bigtree1>$ is locally Lipschitz continuous. 
At last, the equality~\eqref{e:SmoothReconstruction} is another consequence of the aforementioned theorem. Indeed, for smooth $\eta$ thanks to~\eqref{e:ContRec} we have
\begin{equ}
\CR_t(\BV\<bigtree1>)(x)=\Pi^t_x(\BV\<bigtree1>)_t(x)(x)=v(t,x) \bigl(\Pi^t_x\<bigtree1>\bigr)(x)+v'(t,x) \bigl(\Pi^t_x\<bigtree21>\bigr)(x)\;,
\end{equ}
and the last term coincides with the right hand side of~\eqref{e:SmoothReconstruction} because of the definition of the model since the second summand is equal to $v'(t,x)b$. 
\end{proof}

At this point we are ready to formulate the map that we will prove to admit a unique fixed point in the space of controlled processes. Let us define the map $\CM : \CC^{\eta} \times \CX^{\eta, \gamma}_{T} \to \CH^{\eta, \gamma}_{T}$ by $\CM \bigl(v^0, \X, V\bigr) = \bigl(\tilde{v}, \tilde{v}', \tilde{R}\bigr)$ such that
\minilab{FPMap}
\begin{equs}
\tilde{v} &= P u_0 + 4 X^{\<tree1222>} +P' \ast \bigl( 2 \CR \bigl(\BV\<bigtree1>\bigr) + F_v\bigr) + Q\;,\label{e:FPMapV}\\
\tilde{v}' &= 4 X^{\<tree122>} + 2 v\;, \qquad\qquad \tilde{R}(z, \bar z) = \delta_{z, \bar z} \tilde{v} - \tilde{v}'(z)\, \delta_{z, \bar z} X^{\<tree11>}\;,\label{e:FPMapR}
\end{equs}
where $\bigl(P u_0\bigr)_t(x) \eqdef \bigl(P_t \ast u_0\bigr)(x)$, the modelled distribution $\BV\<bigtree1>$ is defined in~\eqref{e:VDot}, the reconstruction operator $\CR$ corresponds to the model introduced above for the controlling process $\X \in \CX$, and $F_v$ and $Q$ are as in \eqref{e:F}. 

\begin{remark}\label{rem:RenEqu}
Notice that if the controlling process $\X_\eps\eqdef\X(\eta_\eps, 0, 0)$ is not renormalized, then Proposition~\ref{p:ModelledDistribution} and \eqref{e:uExp} yield equation \eqref{e:SBEMild} for the fixed point of \eqref{FPMap}. In general, the controlling process $\X_\eps\eqdef\X(\eta_\eps,a_\eps,b_\eps)$ depends on a non-trivial $b_\eps$ (that, in principle, might even diverge) and thanks to Proposition~\ref{p:ModelledDistribution} and our definitions \eqref{e:uExp} and \eqref{Control}, the equation for $u_\eps$ would read
\begin{equ}[e:RenSBE]
\partial_t u_\eps=\Delta u_\eps+\partial_x u_\eps^2 - 4 b_\eps\partial_x u_\eps +\eta_\eps\,,\qquad u_\eps(0,\cdot)=u_0(\cdot)\;,
\end{equ}
which corresponds to the {\it renormalized} equation. 
\end{remark}

The following theorem, which is proved in Section~\ref{sec:FP}, states the existence and uniqueness result for the limit of equation~\eqref{e:SBE}. 

\begin{theorem}\label{t:FixedPoint}
Let $\alpha_\star \in \bigl(-\alpha_{\<tree12>}, \alpha_{\<tree11>}\bigr]$, $\gamma \in \bigl(\alpha_{\<tree11>}, \alpha_{\<tree11>} + \alpha_\star \wedge \alpha_{\<tree122>}\bigr)$ and $\eta \in (-1,0)$.
Then for every $u_0 \in \CC^{\eta}$ and every $\X \in \mathcal{X}$, there exists $T_\infty\eqdef T_\infty(u_0,\X) \in (0, +\infty]$ such that the map $V \mapsto \CM \bigl(u_0, \X, V\bigr)$ admits a unique fixed point in $\CH^{\eta, \gamma}_{\X,T}$ with $T < T_\infty$. Furthermore, for all $T< T_\infty$, equation \eqref{e:SBE} admits a unique solution $u$ on $[0, T]$ and, if $T_\infty<\infty$ then $\lim_{t \to T_\infty} \|u(t,\cdot)\|_{\CC^\eta}=\infty$. At last, for every $T<T_\infty$ the map $\CS_T$ that assigns to $(u_0,\X)\in\CC^\eta\times\CX$ the unique solution $u$ to~\eqref{e:SBE} on $[0, T]$ is jointly locally Lipschitz continuous. 
\end{theorem}

\begin{remark}
It was shown in~\cite{Reloaded}, that the solution of SBE exists at all times, i.e. $T_\infty=\infty$ and an analogous argument would work also in this context, but since we will not prove it explicitly here we refrain from stating the previous theorem so to include this observation. 
\end{remark}


Before proving the previous theorem, we want to introduce the family of discrete systems we have in mind, state the main assumptions on the quantities mentioned in~\eqref{e:DiscreteSBE} and see how to translate in the discrete setting what we have done so far. 

\subsection{The discrete setting}\label{sec:DControllingProcesses}

As already mentioned in the introduction, we now want to deal with a family of space-time discrete systems and prove that each of its members converges to the solution of SBE given in Theorem~\ref{t:FixedPoint}. 
Such systems are defined on a space-time grid that mimics the parabolic nature of the equation we are working with. 
To be more precise, we fix $N\in\N$ and $\eps\eqdef 2^{-N}$, define the grids as in~\eqref{e:Grids} and consider the discrete equations \eqref{e:DiscreteSBE}. In order to define the discrete operators introduced above, we will need three signed measures, $\nu$, $\pi$ and $\mu$, the first two on $\R$ and the latter on $\R^2$, satisfying the following assumptions.    

\begin{assumption}\label{a:nu}
$\nu$ is a purely atomic signed symmetric measure on $\R$, supported on $\Z \cap B(0, R_\nu)$, for some radius $R_\nu>0$, such that 
\begin{equ}[e:DLaplacianProps]
\int_{\R} \nu(dx) = \int_{\R} x\, \nu(dx) = 0 \;,\qquad \int_{\R} x^2\, \nu(dx) = 2\;
\end{equ}
and its Fourier transform, given by $\hat{\nu}(k)\eqdef\int_\R e^{-2\pi i k y}\nu(\dd y)$, vanishes only on $\Z$. 
\end{assumption}

\begin{assumption}\label{a:pi}
$\pi$ is a purely atomic signed measure on $\R$, supported on $\Z\cap B(0,R_\pi)$, for $R_\pi>0$ fixed, such that
\begin{equ}[e:DDerivativeProps]
\int_{\R} \pi(dx) = 0\;,\qquad \int_{\R} x\, \pi(\dd x) = 1\;.
\end{equ}
\end{assumption}

\begin{assumption}\label{a:MeasureMu}
$\mu$ is a purely atomic signed measure on $\R^2$, supported on $(\Z\cap B(0,R_\mu))^2$, for $R_\mu>0$ fixed, such that for any $A,\,B\subseteq \Z^2$, $\mu(A\times B)=\mu(B\times A)$.
We will denote its Fourier transform by $\widehat{\mu}(k_1, k_2)\eqdef \int_{\R^2} e^{-2\pi i  k_1 y_1}e^{-2\pi i k_2 y_2}\mu(\dd y_1,\dd y_2)$.
\end{assumption}

Then, to describe the action of $\Delta_\eps$, $D_{x,\eps}$ and $B_\eps$ on functions on the grid $\Lambda_\eps$, we write $\nu_\eps(\dd y)= \nu(\eps^{-1}\dd y)$, $\pi_\eps(\dd y)=\pi(\eps^{-1}\dd y)$ and $\mu_\eps(\dd y_1,\dd y_2)= \mu(\eps^{-1}\dd y_1,\eps^{-1}\dd y_2)$ for the rescaled versions of $\nu$, $\pi$ and $\mu$ respectively, so that, for any $\varphi\,,\psi\in\ell^\infty(\Lambda_\eps)$ and $x\in\Lambda_\eps$, we can define
\begin{subequations}\label{e:DOperators}
\begin{equs}
\Delta_\eps \varphi(x) &\eqdef  \frac{1}{2\bar{\nu}\eps^2}\int_{\R} \varphi(x + y)\, \nu_\eps(dy)\;, \qquad D_{x,\eps} \varphi(x)\eqdef \frac{1}{\eps}\int_\R \varphi (x+ y)\pi_\eps(\dd y)\;, \label{e:DLaplacian}\\
&B_\eps(\varphi,\psi)(x) \eqdef \int_{\R^2} \varphi(x+y_1)\psi(x+y_2) \mu_\eps (\dd y_1,\dd y_2)\;, \label{e:Dproduct}
\end{equs}
\end{subequations}
where $\bar{\nu}\eqdef \int_{\R} |\nu|(\dd x)$ is the total variation of $\nu$. 

In order to write the mild formulation of~\eqref{e:DiscreteSBE} and build the discrete counterpart of the controlling processes given in Section~\ref{sec:ControllingProcesses}, we introduce the operator $P^\eps$, defined as the Green's function of the space-time discrete heat operator. 
More explicitly, $P^\eps$ is the unique solution of 
\begin{equ}[eq:DHeat]
\bar D_{t,\eps^2} P^\eps(z)=\Delta_\eps P^\eps(z)\;, \qquad P^\eps(0, \cdot) = \eps^{-1}\delta_{0, \cdot}\;,
\end{equ}
where $z \in\Lambda^\s_\eps$, $\delta_{0,\cdot}$ is the Kronecker delta function and we recall that $\bar D_{t,\eps^2}f(t)=\eps^{-2}(f(t+\eps^2)-f(t))$. Furthermore, we impose $P_t^\eps=0$ for all $t\leq0$. 
The mild formulation of~\eqref{e:DiscreteSBE} with $C^\eps=0$ then becomes
\begin{equ}[e:DiscreteSBEMild]
u^\eps=P^\eps u_0^\eps +D_{x,\eps}P^\eps\ast_\eps B^\eps (u^\eps,u^\eps) + D_{x,\eps}P^\eps\ast_\eps\xi^\eps\;,
\end{equ}
where $\ast_\eps$ is the discrete space-time convolution, i.e. for two functions $\varphi,\,\psi$ on $\Lambda_{\eps^2,T}\times\Lambda_\eps$ we have $f\ast_\eps g(z)\eqdef \eps^{3}\sum_{w \in \Lambda_\eps^\s} f(z-w) g(w)$, and $P^\eps u_0^\eps$ is a discrete spatial convolution. 

Moreover, proceeding as in~\eqref{def:Kernel}, we define $K^\eps$ and $\hat K^\eps$ in such a way that $P^\eps=K^\eps+\hat K^\eps$ and $K^\eps$ is compactly supported in a ball centered at 0 and coincides with $P^\eps$ in a neighborhood of the origin, while $\hat K^\eps$ is even and can be assumed to be compactly supported (see Lemma~\ref{l:DKernelDec} for the details). 

We are now ready to build the discrete controlling processes, whose definition is analogous to~\eqref{Control} but accommodates the structure of~\eqref{e:DiscreteSBEMild}.
Let $\eta^\eps$ be a function on $\Lambda_\eps^\s$, then, upon setting  $X^{\<tree1>,\,\eps}(\eta^\eps) = D_{x,\eps}K^\eps \ast \eta^\eps$ and given constants $a,\,b \in \R$, we have
\begin{subequations}\label{DControl}
\begin{align}
&X^{\<tree11>,\,\eps} = B_\eps(1,D_{x,\eps}K^\eps\ast_\eps X^{\<tree1>,\,\eps})\;,&  
&X^{\<tree21>,\,\eps} = B_\eps \big(X^{\<tree11>,\,\eps},\, X^{\<tree1>,\,\eps}\big)-b\;,& \notag\\
&X^{\<tree2>,\,\eps} = B_\eps\bigl( X^{\<tree1>,\,\eps}, X^{\<tree1>,\,\eps} \bigr)-a\;, &
&X^{\<tree12>,\,\eps}= D_{x,\eps}K^\eps\ast_\eps X^{\<tree2>,\,\eps}\;,&\label{e:DControl1}
\end{align}
as well as
\begin{align}
&X^{\<tree22>,\,\eps} = B_\eps\big(X^{\<tree12>,\,\eps},\, X^{\<tree1>,\,\eps}\big) - 2  b\, X^{\<tree1>,\,\eps}\;, && X^{\<tree122>,\,\eps} = D_{x,\eps}K^\eps\ast_\eps X^{\<tree22>,\,\eps}\;,\notag\\
&X^{\<tree124>,\,\eps} = D_{x,\eps}P^\eps \ast_\eps B_\eps\bigl(X^{\<tree12>,\,\eps},\,X^{\<tree12>,\,\eps}\bigr)\;,&&X^{\<tree1222>,\,\eps} = D_{x,\eps}P^\eps \ast_\eps \Bigl(B_\eps\big( X^{\<tree122>,\,\eps},\, X^{\<tree1>,\,\eps}\big) -  b X^{\<tree12>,\,\eps}\Bigr)\;,\label{e:DControl2}
\end{align}
Furthermore, for $t\in\Lambda_{\eps^2}$ and $x,\,y\in\Lambda_\eps$, we set
\begin{align}
R^{\<tree21>,\,\eps}(t, x; y) &\eqdef X^{\<tree21>,\,\eps}(t,y)-\int_{\R^2} X^{\<tree11>,\,\eps}(t, x+ y_1)X^{\<tree1>,\,\eps}(t, y+ y_2)\mu_\eps(\dd y_1,\,\dd y_2)\;,\label{e:DRemainders}\\
R^{\<tree1222>,\,\eps}(z; \bar z) &\eqdef X^{\<tree1222>,\,\eps} (\bar z) - \int_{\R^2} X^{\<tree122>,\,\eps}(z+ y_1)\, D_{x,\eps}P^\eps \ast_\eps X^{\<tree1>,\,\eps} (\bar z + y_2)\mu_\eps(\dd y_1,\dd y_2)\notag\;,
\end{align}
\end{subequations}
where we omitted the dependence of the previous terms on $\eta^\eps$ and where we wrote $z+ y_1 = z+ (0,y_1)$. 

We are now ready for the following definition and the subsequent proposition, representing the discrete version of Definition~\ref{def:ControlProc} and Proposition~\ref{p:Control}. 

\begin{definition}[Discrete Controlling Processes]\label{def:DControlProc}
Let $\alpha_{\<tree1>}\in \bigl(-\frac{3}{5}, -\frac{1}{2}\bigr)$ and $\alpha_\tau$, $\beta_{\bar \tau}$, for $\tau\in\SL$ and $\bar \tau\in\SL_R$, be the same as in Definition~\ref{def:ControlProc} (see also Table~\ref{table:reg}). 
For $\eta^\eps$, a function on $\Lambda_\eps^\s$, and two non-zero constants $a,b\in\R$, we set 
\[
\X^\eps(\eta^\eps)=\left(\eta^\eps, X^{\<tree2>,\,\eps} ,  X^{\<tree22>,\,\eps}, X^{\<tree11>,\,\eps}, X^{\<tree1>,\,\eps}, X^{\<tree21>,\,\eps}, X^{\<tree12>,\,\eps}, X^{\<tree122>,\,\eps}, X^{\<tree124>,\,\eps}, X^{\<tree1222>,\,\eps}, R^{\<tree21>,\,\eps}, R^{\<tree1222>,\,\eps}\right)\;,
\]
where the elements appearing in the previous are given by~\eqref{DControl} (and we hid the dependence on the constants $a,\,b$ to lighten the notation). Then we define the space $\CX^\eps$ of \textit{discrete controlling processes} as the set of families $\X^\eps(\eta^\eps,a,b)$, parametrized by $\eps$, such that $\X^\eps(\eta^\eps,a,b)\in\CW^\eps$, where 
\[
\CW^\eps \eqdef \CC_1^{\alpha_{\<tree1>} - 1,\s,\eps} \oplus \Bigl( \bigoplus_{\tau \in \SL_1} \CC^{\alpha_\tau,\s,\eps}_1 \Bigr)\oplus \Bigl( \bigoplus_{\tau \in \SL_2} \CC^{\alpha_\tau,\eps}_1 \Bigr)\oplus\Bigl( \bigoplus_{\tau \in \SL_3} \CC^{\alpha_\tau,\s,\eps}_1 \Bigr)\oplus \CL_1^{\beta_{\<tree21>},\eps} \oplus \CL_1^{\beta_{\<tree1222>}, \s, \eps}
\] 
endowed with the usual norm\footnote{As pointed out in the introduction, we here recall that we say that a family of functions $f^\eps$ on the grid, parametrized by $\eps$, belongs to (for example) $\CC^{\alpha, \eps}$ if $\|f^\eps\|_{\CC^{\alpha}}^{(\eps)}<\infty$ {\it uniformly} in $\eps$.}, where $\SL_1\eqdef\{\<tree2>,\,\<tree22>,\,\<tree11>\}$, $\SL_2\eqdef\{\<tree1>,\,\<tree21>,\,\<tree12>,\,\<tree122>\}$ and $\SL_3\eqdef\{\<tree124>,\,\<tree1222>\}$. 

We denote by $\X^\eps$ a generic element of this space and, if $\eta^\eps$ coincides with the first component of $\X^\eps\in\CX^\eps$, we say that $\X^\eps$ is an {\it enhancement} (or {\it lift}) of $\eta^\eps$. 
\end{definition}

\begin{proposition}\label{p:DControl}
Let $\{\xi^\eps(z)\}_{z\in\Lambda_\eps^\s}$ be a family of i.i.d. normal random variables with mean $0$ and variance $\eps^{-3}$ on a probability space $(\Omega,\SF,\Prob)$. Let $\nu,\,\pi$ and $\mu$ be three signed measures satisfying Assumptions~\ref{a:nu},~\ref{a:pi},~\ref{a:MeasureMu} and $\hat\nu,\,\hat\pi,\,\hat\mu$ be their Fourier transforms. Then there exists a sequence of constants $\BC^\eps\eqdef(C^{\<tree2>,\,\eps},\,C^{\<tree21>,\,\eps})$ such that for all $p\geq 1$
\begin{equ}
\E \left[\left(\left\|\X^\eps(\xi^\eps,\BC^\eps)\right\|_{\CX}^{(\eps)}\right)^p\right]\lesssim 1
\end{equ}
uniformly in $\eps$, where $C^{\<tree2>,\,\eps}$ behaves asymptotically as $\eps^{-1}$ and $C^{\<tree21>,\,\eps}$ is independent of $\eps$. Precisely, they are given by
\begin{equs}
C^{\<tree2>,\,\eps}&=\eps^{-1}\int_{-\frac{1}{2}}^{\frac{1}{2}}g(k)g(-k)\frac{4\bar \nu^2}{f(k)\big(4\bar\nu+\hat\nu(k)\big)}\hat\mu(- k , k)\dd k\,,\\
C^{\<tree21>,\,\eps}&=-\int_{-\frac{1}{2}}^{\frac{1}{2}}\frac{\Imma(g(-k)\hat\mu(- k , 0))}{k} |g(k)|^2 \frac{4\bar\nu^2(2\bar\nu+\hat\nu(k))^2}{f(k)^2(4\bar\nu+\hat\nu(k))^2} \hat\mu(- k , k)\dd k\;,
\end{equs}
where $f(k)\eqdef -\hat{\nu}(k)/k^2$, $g(k)\eqdef\hat{\pi}(k)/ (i k)$, $\Imma$ denotes the imaginary part and both the integrals  are finite. Furthermore, let $\varrho$ be a symmetric compactly supported smooth function integrating to 1 and, for $\bar \eps\geq \eps$, $\varrho_{\bar \eps}(t,x)\eqdef \bar \eps^{-3}\varrho(\bar \eps^{-2} t, \bar \eps^{-1} x)$ be its rescaled version. Set  $\xi^{\eps}_{\bar \eps}\eqdef \xi^\eps\ast_\eps\varrho_{\bar\eps}$ and let $\BC^{\eps}_{\bar\eps}$ be the sequence of constants introduced above associated to $\xi^{\eps}_{\bar \eps}$. Then, there exists $\theta>0$ such that, for all $p\geq 1$ one has uniformly in $\eps$:
\[
\E\left[\left(\|\X^\eps(\xi^\eps, \BC^\eps);\X^\eps(\xi^{\eps}_{\bar\eps}, \BC^{\eps}_{\bar\eps})\|_{\CX}^{(\eps)}\right)^p\right]\lesssim \bar \eps^{\theta p}\;.
\] 
\end{proposition}

We prove this theorem in Section~\ref{sec:dnif_bounds_disc}, and at this point we can focus on the analytical aspects of the discrete equation~\eqref{e:DiscreteSBEMild} for which we will follow once again the same procedure described in the previous paragraph. 
Let $\X^\eps\in\CX^\eps$ be an enhancement of $\xi^\eps$ according to Definition~\ref{def:DControlProc} and set 
\begin{equ}
v^\eps\eqdef u^\eps-X^{\<tree1>,\,\eps}-X^{\<tree12>,\,\eps}-2X^{\<tree122>,\,\eps}\;,
\end{equ}
where $u^\eps$ solves~\eqref{e:DiscreteSBEMild}, so that $v^\eps$ satisfies
\begin{equ}\label{e:DTruncBurgers}
v^\eps = P^\eps u^\eps_0 +4 X^{\<tree1222>,\,\eps} +  D_{x,\eps}P^\eps\ast_\eps \Bigl( 2 B_\eps\big(v^\eps,\, X^{\<tree1>,\,\eps}\big) +  F^{(\eps)}_{v^\eps}\Bigr) +  Q^{(\eps)}
\end{equ}
with $P^\eps u^\eps_0 \eqdef \bigl(P_t \ast_\eps u^\eps_0\bigr)(x)$, the discrete convolution involving only the space-variable, and 
\begin{equs}[e:DF]
 F^{(\eps)}_{v^\eps} &\eqdef 2 B_\eps\bigr(X^{\<tree12>,\,\eps}, 2 X^{\<tree122>,\,\eps} + v^\eps \bigr) + B_\eps\bigl( 2 X^{\<tree122>,\,\eps} + v^\eps\bigr)\;, \\  
Q^{(\eps)} &\eqdef  X^{\<tree124>,\,\eps}  +D_{x,\eps} \hat K^\eps\ast_\eps \Bigl( \xi_\eps + X^{\<tree2>,\,\eps} + 2 X^{\<tree22>,\,\eps} \Bigr)\;,
\end{equs}
and where we write $B_\eps (f)\eqdef B_\eps(f,f)$. The term that we have to describe through our regularity structure is $B_\eps\big(v^\eps,\, X^{\<tree1>,\,\eps}\big)$, i.e. a twisted product in whose expression $v^\eps$ and $X^{\<tree1>,\,\eps}$ are evaluated at two different points. 
If we look at the expression~\eqref{e:DTruncBurgers}, it seems reasonable to expect that $v^\eps$ admits an expansion of the form
\begin{equ}[e:DVExpansion]
 \delta_{z, \bar z} v^\eps = v'^{,\,\eps}(z)\, \delta_{z, \bar z} X^{\<tree11>,\,\eps} + R^\eps(z, \bar z)\;,
\end{equ}
for $z,\bar z\in\Lambda_{\eps,T}^\s$, $v'^{,\,\eps} = 4 X^{\<tree122>,\,\eps} + 2 v^\eps$ and a suitable remainder $R^\eps$. It is important to keep in mind this expansion in the definition of the abstract setting, and in the next section we will see how to adapt, in this discrete situation, the construction carried out in the continuous case. 

\subsection{The regularity structure, discrete models and modelled distributions}

As we pointed out at the end of the previous section, we want to keep track, at the abstract level, of the fact that the product we are considering is {\it twisted}. 
To do so, we introduce a slightly different regularity structure. Let $\alpha_\tau$, $\beta_{\bar \tau}$, for $\tau\in\SL$ and $\bar \tau\in\SL_R$, be given as in Definition~\ref{def:DControlProc} (see also the Table~\ref{table:reg}), and let $\FI = \Z\cap B(0,R_\mu)$, then we define $(\CA^D,\CT^D,\CG^D)$ (where $D$ stands for ``discrete"), as
\begin{equ}\label{e:DRegularityStructure}
\CA^D=\{\alpha_{\<tree1>},\,\beta_{\<tree21>},\,0,\,\alpha_{\<tree11>}  \}\,,\qquad \CT^D=\langle \<bigtree1>_{k}\,:\,k\in \FI \rangle\oplus\langle\<bigtree21>\rangle\oplus\langle \one\rangle\oplus\langle\<bigtree11>\rangle\;.
\end{equ}
Moreover we set the map $Z^\eps=(\Pi^\eps,\Gamma^\eps, \Sigma^\eps)$, for all $x,y\in\Lambda_\eps$, $s,t\in\Lambda_{\eps^2}$ and $k\in\FI$, to be given by
\begin{subequations}\label{e:DTModel}
\begin{equ}
\Pi^{\eps,t}_x\<bigtree1>_{k}(\cdot)\eqdef X^{\<tree1>,\,\eps}(t,\cdot+\eps k)\,,\qquad\Pi^{\eps,t}_x\<bigtree21>(\cdot)\eqdef R^{\<tree21>,\,\eps}(t, x;\cdot)\,,\qquad\Pi^{\eps,t}_x\<bigtree11>(\cdot)\eqdef\delta^{(t)}_{x,\cdot}X^{\<tree11>,\,\eps}\label{e:DTModelPi}\;,
\end{equ}
and consequently the maps $\Gamma^\eps$ are given by
\begin{equ}
\Gamma^{\eps,t}_{x,\,y }\<bigtree1>_{k}\eqdef \<bigtree1>_{k}\,,\qquad \Gamma^{\eps,t}_{x,\,y }\<bigtree11>\eqdef\<bigtree11> -\delta^{(t)}_{x,y}X^{\<tree11>,\,\eps}\one\,,\qquad\Gamma^{\eps,t}_{x,\,y }\<bigtree21>\eqdef\<bigtree21>\; -\int_{\R^2}  \delta^{(t)}_{x+y_1,y+y_1}X^{\<tree11>,\,\eps}\,\<bigtree1>_{y_2/\eps}\mu_\eps(\dd y_1,\,\dd y_2)\,,
\label{e:DTModelGamma}
\end{equ}
and $\Sigma^\eps$ are equal to
\begin{equ}
\Sigma^{\eps,st}_x\<bigtree1>_{k}\eqdef \<bigtree1>_{k}\,,\qquad \Sigma^{\eps,st}_x\<bigtree11>\eqdef\<bigtree11> -\delta^{(s,t)}_{x}X^{\<tree11>,\,\eps}\one\,, \qquad \Sigma^{\eps,st}_x\<bigtree21>\eqdef\<bigtree21>\; -\int_{\R^2}  \delta^{(s,t)}_{x+y_1}X^{\<tree11>,\,\eps}\,\<bigtree1>_{y_2/\eps}\mu_\eps(\dd y_1,\,\dd y_2)\,.
\label{e:DTModelSigma}
\end{equ}
\end{subequations}
In the following lemma, whose proof is rather immediate, we show that, for every $\eps>0$, $Z^\eps$ defined above is indeed a discrete model. 

\begin{lemma}\label{l:DModel}
Let $\X^\eps\in\CX^\eps$, $\ST^D=(\CA^D,\CT^D,\CG^D)$ be the regularity structure given in~\eqref{e:DRegularityStructure} and $Z^\eps=(\Pi^\eps,\Gamma^\eps, \Sigma^\eps)$ be defined according to~\eqref{e:DTModel}. Then, $Z^\eps$ is a discrete model for $\ST^D$ in the sense of  Definition~\ref{d:DModel}. 
\end{lemma}
\begin{proof}
The validity of the analytic properties come from the fact that $Z^\eps$ is built on $\X^\eps$, so that for $\Gamma^\eps$ and $\Sigma^\eps$ there is nothing to verify. 
Concerning $\Pi^{\eps,t}_x\<bigtree11>(\cdot)$ and $\Pi^{\eps,t}_x\<bigtree21>(\cdot)$ the correct bounds hold by assumption (indeed by definition the latter is given by $ R^{\<tree21>,\,\eps}(t, x;\cdot)$). For the other notice that, given $\lambda \in [\eps,1]$, $\varphi \in \CB^1_0(\R)$, $x\in\Lambda_\eps$, $t\in\Lambda_{\eps^2}$ and $k\in\FI$, we have
\begin{equ}
\big| \langle \Pi^{\eps,t}_x\<bigtree1>_k,\,\varphi^\lambda_x\rangle_\eps\big| = \big|\langle X^{\<tree1>,\,\eps}(t,\cdot+\eps k),\,\varphi^\lambda_x\rangle_\eps\big| =\big|\langle X^{\<tree1>,\,\eps}(t,\cdot),\,\varphi^\lambda_{x+\eps k}\rangle_\eps \big| \leq C \lambda^{\alpha_{\<tree1>}}\,.
\end{equ}
At last, the algebraic properties can be verified directly applying the very definitions of the maps $\Pi^\eps$, $\Gamma^\eps$ and $\Sigma^\eps$. 
\end{proof}

At this point, we are ready to define the discrete modelled distributions we will use in order to describe $v^\eps$ and $B_\eps(v^\eps, X^{\<tree1>,\,\eps})$. 
While for the first we want to recover the controlled structure given in~\eqref{e:DVExpansion}, for the second we would like an expansion better adapted to our twisted product. 
Let $\X^\eps\in\CX^\eps$ and $Z^\eps$ be defined as in~\eqref{e:DTModel}. Given two functions, $v^\eps$ and $v'^{,\,\eps}$, on $\Lambda_{\eps,T}^\s$, we define
\begin{equs}
\BV^\eps_t(x) &\eqdef v^\eps(t,x)\, \one + v'^{,\,\eps}(t,x)\,\<bigtree11>\;, \label{e:DTV}\\
\big(\BV \<bigtree1>\big)^\eps_t(x) &\eqdef \int_{\R^2}v^\eps(t,x+ y_1) \,\<bigtree1>_{y_2/\eps}\mu_\eps(\dd y_1,\dd y_2) + v'^{,\,\eps}(t,x)\, \<bigtree21>\;.\label{e:DTVDot}
\end{equs}
As before, we now would like to replace the twisted product between $v^\eps$ and $X^{\<tree1>,\,\eps}$ appearing in~\eqref{e:DTruncBurgers} by the discrete recontruction of the modelled distribution $\big(\BV \<bigtree1>\big)^\eps$, so to obtain a quantity which is {\it uniformly} well-defined as $\eps$ goes to $0$. 
\newline

Let us fix $\alpha_{\star}>0$, $\gamma > \alpha_{\<tree11>}$ and $\eta \in \R$. Given a triplet of functions $V^\eps=(v^\eps,v'^{,\,\eps},R^\eps)$, where $v^\eps$, $v'^{,\,\eps}$ are on $\Lambda_{\eps,T}^\s$ and $R^\eps$ is on $(\Lambda_{\eps,T}^\s)^2$, we define the norm
\begin{equ}[e:DNorms]
\Vert V^\eps \Vert_{\eta, \gamma; T}^{(\eps)} \eqdef \Vert v^\eps\Vert _{\CC^{\alpha_\star,\s}_{\eta, T}}^{(\eps)} + \Vert v'^{,\,\eps} \Vert _{\CC^{\gamma - \alpha_{\<tree11>},\s}_{\eta - \alpha_{\<tree11>}, T}}^{(\eps)} + \sup_{z \in \Lambda_{\eps,T}^\s} \pnorm{R^\eps(z, \cdot)}_{\eta - \gamma, \gamma; T, z}^{(\eps)}\;,
\end{equ}
and we say that a family of triplets $V^\eps$, parametrized by $\eps$ belongs to $\CH^{\eta, \gamma}_{\eps,T}$ if $\Vert V^\eps \Vert_{\eta, \gamma; T}^{(\eps)}<\infty$ uniformly in $\eps$. 
The discrete controlling processes can then be defined as follows.
\begin{definition}\label{d:DCP} 
Let $\alpha_{\star}>0$, $\gamma>-\alpha_{\<tree1>}$ and $\eta\in(-1,0)$. Let $\CX^\eps$ be the space of discrete controlling processes as in Definition~\ref{def:DControlProc}. Then, the space of {\it discrete controlled processes} $\CX^{\eta, \gamma}_{\eps,T} \subset \CX^\eps \oplus \CH^{\eta, \gamma}_{\eps, T}$ is the algebraic variety of $\X^\eps \in \CX^\eps$ and $V^\eps=\bigl(v^\eps, v'^{,\,\eps}, R^\eps\bigr) \in \CH^{\eta, \gamma}_{\eps, T}$ such that for $z,\bar z\in\Lambda_{\eps,T}^\s$ the identity \eqref{e:DVExpansion} holds. Moreover, we define the set of all processes controlled by $\X^\eps \in \CX^\eps$ as
\begin{equ}
\CH^{\eta, \gamma}_{\eps, \X^\eps,T} \eqdef \Bigl\{ V^\eps \in \CH^{\eta, \gamma}_{\eps,T} : \bigl(\X^\eps, V^\eps\bigr) \in \CX^{\eta, \gamma}_{\eps,T} \Bigr\}\;,
\end{equ}
endowed with the norms introduced in \eqref{e:DNorms}. 
\end{definition}

The next proposition shows that the twisted product defined above, behaves sufficiently similarly to the usual one. 

\begin{proposition}\label{p:DProduct}
Let $\alpha_{\star}>0$, $\gamma>-\alpha_{\<tree1>}$ and $\eta\in(-1,0)$. Let $\X^\eps\in\CX^\eps$, $\ST^D=(\CA^D,\CT^D,\CG^D)$ be the regularity structure given in~\eqref{e:DRegularityStructure} and $Z^\eps$ be defined according to~\eqref{e:DTModel}. 
Let $V^\eps=(v^\eps,v'^{,\,\eps},R^\eps)\in\CH^{\gamma,\eta}_{\eps,\X^\eps,T}$ and $\BV^\eps$, $(\BV\<bigtree1>)^\eps$ be given by~\eqref{e:DTV} and~\eqref{e:DTVDot} respectively. 
Then $\BV^\eps\in\CD^{\gamma,\eta}_{\eps,T}(Z^\eps)$, $(\BV\<bigtree1>)^\eps\in\CD^{\gamma+\alpha_{\<tree1>},\eta+\alpha_{\<tree1>}}_{\eps,T}(Z^\eps)$ and $\CR^\eps_t(\BV\<bigtree1>)^\eps\in\CC^{\alpha_{\<tree1>},\eps}$ for all $t\in\Lambda_{\eps^2,T}$, where the latter is the discrete reconstruction operator. 
Moreover, if $\X^\eps(\eta^\eps)\eqdef \X^\eps(\eta^\eps,\,a,\,b)$, with $a, b\in\R$, then for all $(t,x)\in\Lambda_{\eps,T}^\s$ the following equality holds 
\begin{equ}[e:DSmoothReconstruction]
\CR^\eps_t(\BV\<bigtree1>)^\eps(x)=B^\eps(v_t^\eps, X_t^{\<tree1>,\,\eps})(x)-b\, v'^{,\,\eps}(t,x).
\end{equ} 
\end{proposition}

\begin{proof}
It is immediate to see, by the definition of the model $Z^\eps$, that $\BV^\eps\in\CD^{\gamma,\eta}_{\eps,T}(Z^\eps)$, so we focus on $(\BV\<bigtree1>)^\eps$. Each of the summands in~\eqref{e:ModelledDistributionNorm} can be easily shown to be bounded uniformly in $\eps$, apart from two terms, whose argument is given by 
\[
\CQ_{\alpha_{\<tree1>}}\Big((\BV\<bigtree1>)_t^\eps(x)-\Gamma^{\eps,t}_{xy}(\BV\<bigtree1>)_t^\eps(y)\Big)\quad \text{and}\quad \CQ_{\alpha_{\<tree1>}}\Big((\BV\<bigtree1>)_t^\eps(x)- \Sigma^{\eps,ts}_x(\BV\<bigtree1>)_s^\eps(x)\Big)\,. 
\]
They can be treated exploiting the same techniques, so we will consider only the first. Notice that for $x \neq y\in\Lambda_\eps$ and $t\in\Lambda_{\eps^2,T}$ the identity \eqref{e:DVExpansion} yields
\begin{equs}
\CQ_{\alpha_{\<tree1>}}\Big((\BV\<bigtree1>)_t^\eps(x)&-\Gamma^{\eps,t}_{xy}(\BV\<bigtree1>)_t^\eps(y)\Big)=\int_{\R^2} \Big(\delta^{(t)}_{y+y_1,x+y_1} v^\eps -v'^{,\,\eps}(t,y)\delta^{(t)}_{y+y_1,x+y_1} X^{\<tree11>,\,\eps}\Big)\<bigtree1>_{y_2/\eps}\mu_\eps(\dd y_1,\dd y_2)\\
&=\int_{\R^2}\Big( R^\eps((t,y+y_1), (t,x+y_1)) + \delta_{y, y + y_1}^{(t)} v'^{,\,\eps}\, \delta^{(t)}_{y+y_1,x+y_1} X^{\<tree11>,\,\eps}\Big)\<bigtree1>_{y_2/\eps}\mu_\eps(\dd y_1,\dd y_2)\;.
\end{equs}
Using our assumptions on $v'^{,\,\eps}, R^\eps$, $X^{\<tree11>,\,\eps}$ and $\mu$, it follows
\begin{equs}
\|(\BV\<bigtree1>)_t^\eps(x)&-\Gamma^{\eps,t}_{xy}(\BV\<bigtree1>)_t^\eps(y)\|_{\alpha_{\<tree1>}} \leq \int_{\R^2}\Big(\big|R^\eps((t,y+y_1), (t,x+y_1))\big|+\big|\delta_{y, y + y_1}^{(t)} v'^{,\,\eps}\big| \big|\delta^{(t)}_{y+y_1,x+y_1} X^{\<tree11>,\,\eps}\big|\Big) |\mu_\eps|(\dd y_1,\dd y_2)\\
&\lesssim \int_{\R^2}\Big(|t|_0^{\eta-\gamma}|x-y|^\gamma + |t|_0^{\eta-\gamma} (\eps |y_1|)^{\gamma-\alpha_{\<tree11>}} |x-y|^{\alpha_{\<tree11>}}\Big) |\mu|(\dd y_1,\dd y_2) \lesssim |t|_0^{\eta-\gamma}|x-y|^\gamma\;,
\end{equs}
where we used the facts that $\gamma>\alpha_{\<tree11>}$ and $\eps\leq|x-y|$. Moreover, the constants hidden in the previous chain of inequalities are clearly uniform in $\eps$. 
Hence, we can conclude that $(\BV\<bigtree1>)^\eps\in\CD^{\gamma+\alpha_{\<tree1>},\eta+\alpha_{\<tree1>}}_{\eps,T}(Z^\eps)$ and, by Definition~\ref{d:DReconstruct} and Lemma~\ref{l:DModel}, we also have
\begin{equs}
\CR^\eps_t(\BV\<bigtree1>)^\eps(x)=\Pi^{\eps,t}_x(\BV\<bigtree1>)^\eps_t(x)(x)&=\int_{\R^2}v^\eps(t,x+y_1) \bigl(\Pi^{\eps,t}_x\<bigtree1>_{y_2/\eps}\bigr)(x)\mu_\eps(\dd y_1,\dd y_2)+v'^{,\,\eps}(t,x) \bigl(\Pi^{\eps,t}_x\<bigtree21>\bigr)(x)\\
&=B^\eps(v^\eps_t, X^{\<tree1>,\,\eps}_t)(x)-b\, v'^{,\,\eps}(t,x)\;,
\end{equs}
which in turn completes the proof.
\end{proof}

Thanks to the previous proposition, we are ready to define the map $\CM^\eps$, which represents the discrete counterpart of~\eqref{FPMap}.  
More precisely, let $\CM^\eps$ be the map on $\CC^{\eta,\eps} \times \CX^{\eta, \gamma,\eps}_{T}$ given by $\CM^\eps \bigl(u_0^\eps, \X^\eps, V^\eps\bigr) = \bigl(\tilde{v}^\eps, \tilde{v}'^{,\,\eps}, \tilde{R}^\eps\bigr)$, such that for $z,\,\bar z\in\Lambda_{\eps,T}^\s$ one has
\minilab{DFPMap}
\begin{equs}
\tilde{v}^\eps &= P^\eps u_0^\eps + 4 X^{\<tree1222>,\,\eps} + D_{x,\eps}P^\eps \ast_\eps \bigl(2 \CR^\eps \bigl(\BV\<bigtree1>\bigr)^\eps
+ F_{v^\eps}^{(\eps)}\bigr) + Q^{(\eps)}\;,\label{e:DFPMapV}\\
\tilde{v}'^{,\,\eps} &= 4 X^{\<tree122>,\,\eps} + 2 v^\eps\;, \qquad \tilde{R}^\eps(z, \bar z) = \delta_{z, \bar z} \tilde{v}^\eps - \tilde{v}'^{,\,\eps}(z)\, \delta_{z, \bar z} X^{\<tree11>,\,\eps}\;,\label{e:DFPMapR}
\end{equs}
where the modelled distribution $(\BV\<bigtree1>)^\eps$ is defined in~\eqref{e:DTVDot}, the reconstruction operator $\CR^\eps$ corresponds to the model introduced above for the controlling process $\X^\eps \in \CX^\eps$, and $F_{v^\eps}^{(\eps)}$ and $Q^{(\eps)}$ are as in \eqref{e:DF}. The following theorem, which we prove in Section~\ref{sec:unif_bounds}, represents the discrete version of Theorem~\ref{t:FixedPoint}.

\begin{theorem}\label{t:DFixedPoint}
Under the assumptions of Theorem~\ref{t:FixedPoint}, let $\X^\eps\in\CX^\eps$ be a sequence of controlling processes and let $u_0^\eps\in\CC^{\eta,\eps}$ be a sequence of periodic functions on $\Lambda_\eps$. Then there exists $T^d_\infty\in(0,+\infty]$, independent of $\eps$, such that for all $T<T^d_\infty$ the sequence of solution maps $\CS_T^\eps$ that assigns to $(u_0^\eps, \X^\eps)$ the solution $u^\eps$ of~\eqref{e:DiscreteSBE} is locally Lipschitz continuous uniformly in $\eps$ with respect to the discrete norms. 
\end{theorem}

\begin{remark}\label{rem:DRenEqu}
As in Remark~\ref{rem:RenEqu}, if the discrete controlling process $\X^\eps\eqdef\X(\eta^\eps,a^\eps,b^\eps)$ depends on a non-trivial $b^\eps$ (which is going to be the case in the setting of Proposition~\ref{p:DControl}), then, by the equality~\eqref{e:DSmoothReconstruction} the equation for $u^\eps$ would read
\begin{equ}[e:DRenSBE]
\bar D_{t,\eps^2}u^{\eps}=\Delta_\eps u^{\eps} +D_{x,\eps} B_\eps(u^{\eps},u^{\eps}) - 4 b^\eps D_{x,\eps} u^\eps+D_{x,\eps}\eta^\eps\,,\qquad u^\eps(0,\cdot)=u^\eps_0(\cdot)\;,
\end{equ}
which is the {\it discrete renormalized} equation. 
\end{remark}

\section{Continuous and discrete Schauder estimates and fixed points}\label{sec:Analytic}

The aim of this section is twofold. 
We will begin by proving the Schauder estimates needed to treat the convolution with the heat kernel and see how they can be put into practice so to build a solution for \eqref{e:SBE}. Then we will translate this construction in the discrete setting and obtain uniform bounds on the discrete solution maps. 

\subsection{Continuous convolutions}\label{sec:ContConv}

As already mentioned in the introduction, our approach is based on the ability of treating the convolution with the heat kernel as a testing against a recentered and rescaled test function, where the time variable plays the role of the scaling parameter. 
After all, the convolution is nothing but a ``centering" around a specific point, as $t$ tends to 0 the heat kernel converges weakly to the Dirac delta function and its $L^1$ norm is constant in $t$. 
The missing ingredient is the compactness of the support which is however not essential. Indeed, for $t>0$, the heat kernel is a Schwarz function and the following direct generalization of~\cite[Lem.~6.3]{GIP15} shows that this is enough to guarantee the validity of what claimed above (actually we prove that it suffices that the function decays sufficiently fast at $\infty$).  

\begin{lemma}\label{l:Schwartz}
Let $\alpha \in \R$, $\gamma \geq 0$, $\lambda \in (0,1]$, and let the map $\T \ni x \mapsto \zeta_x \in \CD'(\T)$ satisfy
\begin{equ}[e:SchwartzCond]
|\langle \zeta_x, \varphi_x^\lambda\rangle | \leq C \lambda^{\alpha}\;, \qquad |\langle \zeta_x - \zeta_y, \varphi_x^\lambda\rangle | \leq C \lambda^{\alpha - \gamma} |x-y|^{\gamma}\;,
\end{equ}
for some $C > 0$, all $\varphi \in \CB^r_0(\R)$ with integer $r > |\alpha|$, and locally uniformly over $x, y \in \R$ such that $|x - y|\geq \lambda $, where we have identified $\zeta$ with its periodic extension to $\R$. Let $\delta>0$ be such that $\gamma-\delta<-1$. Then, for any $\psi \in \CC^r(\R)$ such that $\sup_{j\leq r}\sup_x |x|^\delta|\partial^j\psi(x)|<\infty$, one has
\begin{equ}\label{ineq:Schwartz}
 |\langle \zeta_x, \psi_x^\lambda\rangle | \lesssim C \lambda^{\alpha}\;,
\end{equ}
(locally) uniformly over $x\in\R$, $\lambda \in (0,1]$ and where the proportionality constant depends only on $\psi$.
\end{lemma}

\begin{proof}
Fix $\lambda\in (0,1]$. We first notice that, since $\zeta_x$ is periodic, the second bound in \eqref{e:SchwartzCond} holds uniformly over $x, y \in \R$ such that $|x-y| \geq \lambda$. 
Let $Y\eqdef \frac{\lambda}{2}\Z\setminus B$, where $B=\{y\in\frac{\lambda}{2}\Z\,: |y|< \lambda\}$. 
Now, for $\psi \in \CS(\R)$, we can write $\psi_x^\lambda = \sum_{y \in Y \cup \{0\}} \varphi^\lambda_{x+y}$, where $\varphi^\lambda_{x+y}$ is supported in a ball of radius $\lambda$ centered around $x+y$ and scales like $\lambda$ (see \cite[Rem. 2.21]{Hai14}). 
Moreover, since $\psi$ is such that $\sup_{j\leq r}\sup_x |x|^\delta|\partial^j\psi(x)|<\infty$,  for each point $y \in Y$ one has the bound 
\begin{equ}[e:Crucial]
\Vert \varphi_{x+y}^\lambda \Vert_{\CC^r} \lesssim \lambda^{-1-r+\delta} |y|^{-\delta}\,.
\end{equ}
Thus, using \eqref{e:SchwartzCond} we get
\begin{equs}
|\langle \zeta_x, \psi_x^\lambda\rangle| &\leq |\langle \zeta_x, \varphi_{x}^\lambda\rangle|  + \sum_{y \in Y} |\langle \zeta_{x + y}, \varphi_{x+y}^\lambda\rangle| + \sum_{y \in Y} |\langle \zeta_{x} - \zeta_{x + y}, \varphi_{x+y}^\lambda\rangle| \\
&\lesssim C \lambda^\alpha + C \lambda^\alpha \sum_{y \in Y} \bigl(|y| \lambda^{-1}\bigr)^{-\delta}+ C \lambda^{\alpha - \gamma} \sum_{y \in Y} |y|^{\gamma} \bigl(|y| \lambda^{-1}\bigr)^{-\delta} \lesssim C \lambda^\alpha \sum_{y \in \Z} |y|^{\gamma -\delta} \lesssim C \lambda^\alpha\;,
\end{equs}
as soon as $\gamma - \delta < - 1$, which is the required bound.
\end{proof}

\begin{remark}
The reason why we take into account maps $x\mapsto\zeta_x$, where $\zeta_x$ \textit{depends} on a base point $x$, is that we want to be able to replace $\zeta$ with the image of a model for a certain regularity structure or the term appearing on the left hand side of the reconstruction bound~\eqref{e:Reconstruction}. 
In other words, $\zeta_x$ will represent a local generalized Taylor expansion around the point $x$. 
\end{remark}

\begin{remark}\label{rem:Schwartz}
The proof of Lemma~\ref{l:Schwartz}, as well as~\cite[Rem. 2.21]{Hai14}, show that we could have formulated the statement in a slightly different fashion. Namely, instead of taking a function $\psi\in \CC^r(\R)$ such that $\sup_{j\leq r}\sup_x |x|^\delta|\partial^j\psi(x)|<\infty$, and prove~\eqref{ineq:Schwartz} where $\psi^\lambda$ is the rescaled version of $\psi$, we could have considered a family of maps $\psi^\lambda\in\CC^r(\R)$, where $\lambda\in(0,1]$, such that $\sup_{j\leq r}\sup_x |x|^\delta|\partial^j\psi^\lambda(x)|\lesssim\lambda^{-1-r+\delta}$ and~\eqref{ineq:Schwartz} would still hold.   
\end{remark}

The next proposition, joint with the previous lemma represent, as we will shortly see, the Schauder estimates we need.

\begin{proposition}\label{p:HeatBounds}
Let $\beta > -2$, $\rho\in(-2,0]$, $\alpha \leq \beta$ and $n\in\N$, $n \in [\beta + 1, \alpha + 2)$. 
Let $T\in(0,1]$ and consider a map $(0,T]\times\T \ni (t,x)= z \mapsto \zeta_z \in L^\infty\bigl([0, T], \CD'(\T)\bigr)$ for which there exists a constant $C>0$ such that
\begin{subequations}\label{e:HeatBound}
\begin{equs}
\big|\langle \zeta_{z}(t, \cdot), \varphi_x^\lambda\rangle\big| &\leq C \lambda^{\beta} \onorm{t}^{\rho}\;, \label{e:HeatBound1}\\
\big|\langle \zeta_{\bar z}(t, \cdot)-\zeta_z(t,\cdot), \varphi_x^\lambda\rangle\big| &\leq C \lambda^{\alpha}|t,\bar t|_0^{\rho}|\bar z-z|_\s^{\beta-\alpha}\;,\label{e:HeatBound2}
\end{equs}
\end{subequations}
uniformly over $z=(t,x),\,\bar z=(\bar t, \bar x) \in (0,T]\times\T$ such that $|\bar t-t|\leq\onorm{\bar t, t}^2$, $\lambda \in (0,1]$ and $\varphi \in \CB^r_0(\R)$ with integer $r > |\alpha|$. 
Then, for any $\theta_1\in(0,\alpha-n+2)$ and $\theta_2\in\R$ such that $\theta\eqdef\theta_1-\theta_2>0$, the following bound holds
\begin{equ}[e:ConvBound]
\sup_{z} \onorm{t}^{-(\bar\rho+\bar \beta)} \bigl|\bigl(P^{(n)} \ast \zeta_{z}\bigr)(z)\bigr| +  \sup_{z }\pnorm{P^{(n)} \ast \zeta_{z}}_{ \bar\rho, \bar \beta; T, z} \lesssim C T^{\frac{\theta}{2}}\;,
\end{equ}
where, the suprema run over $z=(t,x)\in (0,T]\times\T$, $P$ is the heat kernel, $P^{(n)}$ is the $n$-th spatial derivative of $P$, $\bar \beta = \beta + 2 - n-\theta_1$ and $\bar \rho = \rho+\theta_2$. 
\end{proposition}

\begin{remark}
The statement would still hold in case one took $n\in(\beta+1,\alpha+2)$ and $\theta=\theta_1=\theta_2=0$. The reason why we introduce these parameters is that we want to use the last result in order to close the fixed point argument and we need to be able to ``play" with the H\"older regularity (given by the second summand of~\eqref{e:ConvBound}) and the explosion rate as time approaches $0$ (see the proof of Proposition~\ref{FixedPointMap}). 
\end{remark}

\begin{remark}
In the previous proposition, the assumptions~\eqref{e:HeatBound}, in a sense, require the possibility to trade the regularity of the distribution $\zeta_z$, for fixed $z\in(0,T]\times\T$ with the one of the map $(0,T]\times\T\ni z\mapsto\zeta_z$. 
As we will see in the following lemma, this is indeed the case for both the model and the reconstruction operator. 
\end{remark}

\begin{proof}
Let $z=(t,x),\,w=(s,y)\in(0,T]\times\T$, then, for a function $\varphi \in \CB^r_0(\R)$ and $\lambda\in(0,1]$, the triangular inequality and~\eqref{e:HeatBound} immediately give
\begin{equ}[e:HeatFour]
\big|\langle \zeta_z(s,\cdot),\varphi_y^\lambda\rangle\big|\leq \big|\langle \zeta_w(s,\cdot),\varphi_y^\lambda\rangle\big| +\big|\langle \zeta_z(s,\cdot)-\zeta_w(s\cdot),\varphi_y^\lambda\rangle\big|\lesssim s^{\frac{\rho}{2}} \big(\lambda^\beta + \lambda^{\alpha}  |z-w|_\s^{\beta-\alpha}\big)
\end{equ}
under the assumption that $s\leq t$. 
We apply the previous to bound the first summand at the left hand side of~\eqref{e:ConvBound} and obtain
\begin{equ}
|P^{(n)}\ast \zeta_z(z)|\leq \int_0^t \big|\langle\zeta_{z}(s,\cdot), P^{(n)}_{t-s}(x-\cdot)\rangle\big|\dd s\lesssim \int_0^t (t-s)^{\frac{\beta-n}{2}}s^{\frac{\rho}{2}}\dd s \lesssim |t|_0^{\bar\rho+\bar \beta} T^{\frac{\theta}{2}}\;,
\end{equ}
where the last passage holds because $\rho>-2$, $n<\beta+2$ and $\theta>0$. 

For the second term in~\eqref{e:ConvBound}, we can separately deal with space and time increments.
So, take $\bar z=(t,\bar x)\neq (t,x)=z$ and the quantity of interest is 
\begin{equ}[e:SpaceIncrements]
P^{(n)}\ast \zeta_z(\bar z)-P^{(n)}\ast \zeta_z(z)=\int_0^t \langle\zeta_{z}(s,\cdot), P^{(n)}_{t-s}(\bar x-\cdot)\rangle-\langle\zeta_{z}(s,\cdot), P^{(n)}_{t-s}(x-\cdot)\rangle \dd s\;.
\end{equ}
Now, in case $t-s<|x-\bar x|^2$ we treat  each of the two summands separately. 
For $t-s<\frac{t}{2}$, the second summand at the integrand can be bounded as before (since it equals $P^{(n)}\ast \zeta_z(z)$), thus giving a contribution of
\begin{equ}\label{e:TimeIntegralA}
\int_{t-s<|x-\bar x|^2\wedge\frac{t}{2}} (t-s)^{\frac{\beta-n}{2}}s^{\frac{\rho}{2}} \dd s\lesssim t^{\frac{\rho+\theta_1}{2}}\int_{t-s<|x-\bar x|^2} (t-s)^{\frac{\beta-n-\theta_1}{2}}\dd s\lesssim T^{\frac{\theta}{2}}|t|_0^{\bar\rho} |z-\bar z|_\s^{\bar \beta}\;,
\end{equ}
where we restricted the integral to the set of $s$ for which $t-s<|x-\bar x|^2$ and used the facts $s\in(\frac{t}{2},t]$ and $0<\theta_1<\beta+2-n$. 
For the other we apply~\eqref{e:HeatFour} to the integrand, and obtain two summands, the first of which corresponds to~\eqref{e:TimeIntegralA} and can therefore be analogously bounded while the second is
\begin{equ}\label{e:TimeIntegralB}
t^{\frac{\theta_1}{2}}|x-\bar x|^{\beta-\alpha}\int_{t-s<|x-\bar x|^2\wedge\frac{t}{2}} (t-s)^{\frac{\alpha-n-\theta_1}{2}}s^{\frac{\rho}{2}}\lesssim T^{\frac{\theta}{2}}|t|_0^{\bar\rho} |z-\bar z|_\s^{\bar \beta}
\end{equ}
where the last bound holds since $s\in(\frac{t}{2},t]$ and $0<\theta_1<\alpha-n+2$. 
If instead $|x-\bar x|^2>t-s\geq\frac{t}{2}$, we follow the same procedure outlined above to bound the integrands, but now the integrals in~\eqref{e:TimeIntegralA} and~\eqref{e:TimeIntegralB} are over $s<\frac{t}{2}$, thus respectively giving a contribution of order
\begin{equ}[e:TimeIntegral2]
t^{\frac{\rho+\beta-n+2}{2}}\lesssim T^{\frac{\theta}{2}} |t|_0^{\bar\rho}|z-\bar z|_\s^{\bar\beta}\qquad\text{and}\qquad|x-\bar x|^{\beta-\alpha}t^{\frac{\rho+\alpha-n+2}{2}}\lesssim T^{\frac{\theta}{2}}|t|_0^{\bar\rho}|z-\bar z|_\s^{\bar\beta}
\end{equ}
since $\rho>-2$, $\alpha+2>n$ and $\beta\geq\alpha$. 
We can now turn to the case $t-s\geq|x-\bar x|^2$. 
In this situation, we apply Taylor's theorem to the spatial increment of the heat kernel, so that
\begin{equ}
\int \langle\zeta_{z}(s,\cdot), P^{(n)}_{t-s}(\bar x-\cdot)- P^{(n)}_{t-s}(x-\cdot)\rangle \dd s=(x-\bar x)\int_0^1\int \langle\zeta_{z}(s,\cdot), P^{(n+1)}_{t-s}(\tilde x-\cdot)\rangle\dd s\dd \nu
\end{equ}
where we hid the dependence of the integrand on $\nu$ in $\tilde x=\tilde x(\nu)= x+\nu (\bar x-x)$ and the integral in $s$ is taken over $t-s\geq|x-\bar x|^2$. For $t-s\leq \frac{t}{2}$, we apply~\eqref{e:HeatFour} and bound the previous by
\begin{equ}
|x-\bar x|\int_{\frac{t}{2}\geq t-s\geq|x-\bar x|^2} (t-s)^{\frac{\beta-n-1}{2}} s^{\frac{\rho}{2}} \dd s
\lesssim  t^{\frac{\rho+\theta_1}{2}}|x-\bar x|\int_{t-s\geq|x-\bar x|^2} (t-s)^{\frac{\beta-n-1-\theta_1}{2}}\dd s\lesssim T^{\frac{\theta}{2}}|t|_0^{ \bar\rho} |z-\bar z|_\s^{\bar \beta}
\end{equ}
where we implicitly exploited that, for any $\nu\in[0,1]$, $|x-\tilde x|^2\leq|x-\bar x|^2\leq t-s$ and $\beta\geq \alpha$, and the last inequalities are a consequence of $s\in(\frac{t}{2},t]$ and $n\geq\beta+1>\beta+1-\theta_1$. 
If instead, $\frac{t}{2}\leq t-s \,(\leq t)$, then the quantity to bound is the same as the first term in the last chain, but this time the integral is over $s\in[0,\frac{t}{2}]$, so that 
\begin{equs}
|x-\bar x|\int_{t-s\geq|x-\bar x|^2\vee \frac{t}{2}} (t-s)^{\frac{\beta-n-1}{2}} s^{\frac{\rho}{2}} \dd s\lesssim t^{\frac{\beta-n-1}{2}}|x-\bar x|\int_0^{\frac{t}{2}} s^{\frac{\rho}{2}}\dd s\lesssim T^{\frac{\theta}{2}}|t|_0^{ \bar\rho} |z-\bar z|_\s^{\bar \beta}
\end{equs}
since $n\geq\beta+1$ and $\rho>-2$. 
At last we need to investigate the time regularity of the convolution of $\zeta_\cdot$ with the heat kernel. 
To do so, we take $t,\,\bar t\in(0,T]$ such that $|t-\bar t|\leq\onorm{t,\bar t}^2$, we assume, without loss of generality, $t<\bar t$ and set $\bar z=(\bar t,x)$ and $z=(t,x)$. 
Then, we notice that the time increment of the quantity at study can be written in the following way
\begin{equ}[e:HeatTimeBound]
\int_{t}^{\bar t} \langle \zeta_{z}(s, \cdot), P^{(n)}_{\bar t-s}(x - \cdot)\rangle\, ds + \int_{0}^t \langle \zeta_{z}(s, \cdot), \big(P^{(n)}_{\bar t-s} - P^{(n)}_{t-s}\big)(x - \cdot)\rangle\, ds \;.
\end{equ}
For the first summand, we proceed as usual, i.e. we apply~\eqref{e:HeatFour} to the integrand and bound the resulting integral by
\begin{multline*}
\int_t^{\bar t} (\bar t-s)^{\frac{\beta-n}{2}} s^{\frac{\rho}{2}} \dd s+(\bar t-t)^{\frac{\beta-\alpha}{2}}\int_t^{\bar t}(\bar t-s)^{\frac{\alpha-n}{2}} s^{\frac{\rho}{2}} \dd s\\
\lesssim t^{\frac{\rho+\theta_1}{2}}\Big(\int_t^{\bar t}(\bar t-s)^{\frac{\beta-n-\theta_1}{2}}\dd s
+(\bar t-t)^{\frac{\beta-\alpha}{2}}\int_t^{\bar t}(\bar t-s)^{\frac{\alpha-n-\theta_1}{2}}\dd s\Big)\lesssim T^{\frac{\theta}{2}} \onorm{t}^{\bar \rho}|z-\bar z|_\s^{\bar \beta}
\end{multline*}
which holds since for $s\in[t,\bar t]$, $s-t\leq\bar t-t$ and $\theta_1\in(0,\alpha-n+2)$.  
The analysis of the second summand in \eqref{e:HeatTimeBound} is similar to the one we carried out for the spatial increments, with some special arrangements that we will now point out. 
For $t-s<\bar t-t$, we treat each term separately. Thanks to~\eqref{e:HeatFour}, we have
\begin{equs}
\langle \zeta_{z}(s, \cdot), P^{(n)}_{\bar t-s}(x - \cdot)\rangle&\lesssim (\bar t-s)^{\frac{\beta-n}{2}}s^{\frac{\rho}{2}}+(\bar t-s)^{\frac{\alpha-n}{2}}s^{\frac{\rho}{2}}(t-s)^{\frac{\beta-\alpha}{2}}\lesssim ( t-s)^{\frac{\beta-n}{2}}s^{\frac{\rho}{2}} \\
\langle \zeta_{z}(s, \cdot), P^{(n)}_{t-s}(x - \cdot)\rangle&\lesssim ( t-s)^{\frac{\beta-n}{2}}s^{\frac{\rho}{2}}
\end{equs}
where in the first line we exploited the facts $\bar t-s>t-s$ and $\beta\geq \alpha$. 
If $t-s<\frac{t}{2}$, the integral to bound is the same as in~\eqref{e:TimeIntegralA} but with a different domain of integration, namely those $s$ such that $t-s<\bar t-t$, and the bound is analogous. If instead $\frac{t}{2}\leq t-s<\bar t-t$, one proceeds as in the first term of~\eqref{e:TimeIntegral2}. 
In case $t-s\geq\bar t-t$, we apply Taylor's formula and the equality $\partial_t P_t=\Delta P_t$, so that the integrand of the second summand in~\eqref{e:HeatTimeBound} becomes
\begin{equ}
\int_t^{\bar t} \langle \zeta_{z}(s, \cdot), P^{(n+2)}_{q -s}(x - \cdot)\rangle\, \dd q \lesssim s^{\frac{\rho}{2}} \int_t^{\bar t}  (q-s)^{\frac{\beta-n-2}{2}}\dd q\lesssim (\bar t-t) \,s^{\frac{\rho}{2}}  (t-s)^{\frac{\beta-n-2}{2}}
\end{equ}
where we made use of~\eqref{e:HeatFour} and the fact that, trivially, $q-s\geq t-s$ for all $q\in[t,\bar t]$.  
If $t-s<\frac{t}{2}$, we integrate the previous over $s$ such that $t-s\geq\bar t-t$ and get 
\begin{equ}
(\bar t-t)\int_{t-s\geq \bar t-t}s^{\frac{\rho}{2}}  (t-s)^{\frac{\beta-n-2}{2}}\dd s
\lesssim  t^{\frac{\rho+\theta_1}{2}}(\bar t-t)\int_{t-s\geq \bar t-t} (t-s)^{\frac{\beta-n-2-\theta_1}{2}}\dd s\lesssim T^{\frac{\theta}{2}}|t|_0^{ \bar\rho}|z-\bar z|_\s^{\bar\beta}
\end{equ}
since $\beta<n$. At last, for $(t\geq )\,t-s\geq \frac{t}{2}\vee (\bar t-t)$, we have
\begin{equ}
(\bar t-t)\int_0^{\frac{t}{2}}s^{\frac{\rho}{2}}  (t-s)^{\frac{\beta-n-2}{2}}\dd s\lesssim t^{\frac{\beta-n-2}{2}}(\bar t-t)\int_0^{\frac{t}{2}} s^{\frac{\rho}{2}}\dd s\lesssim T^{\frac{\theta}{2}} |t|_0^{\bar \rho}|z-\bar z|_\s^{\bar\beta}
\end{equ}
which concludes the proof. 
\end{proof}

 As a straightforward corollary of the previous proposition combined with Lemma~\ref{l:Schwartz}, we have the Schauder estimate for H\"older distributions. 

\begin{corollary}\label{c:Schauder}
Let $T>0$, $\beta > -2$, $\rho\in(-2,0]$ and $n$ an integer in $[\beta+1,\beta+2)$. If  $\zeta\in\CC^\beta_{\rho,T}$, then for any $\theta_1\in(0,\beta-n+2)$ and $\theta_2\in\R$ such that $\theta\eqdef\theta_1-\theta_2>0$, we have
\begin{equ}[e:ConvBound1]
\sup_{z} \onorm{t}^{-(\bar \rho+\bar \beta)} \bigl|\bigl(P^{(n)} \ast \zeta\bigr)(z)\bigr| +  \sup_{z }\pnorm{P^{(n)} \ast \zeta}_{ \bar\rho, \bar \beta; T, z} \lesssim C T^{\frac{\theta}{2}}
\end{equ}
where, the suprema run over $z=(t,x)\in (0,T]\times\T$, $P$ is the heat kernel, $P^{(n)}$ is the $n$-th spatial derivative of $P$, $\bar \beta = \beta + 2 - n-\theta_1$ and $\bar \rho=\rho+\theta_2$. 
\end{corollary}

\begin{proof}
The proof is immediate. Indeed, since $\zeta$ does not depend on the base point,~\eqref{e:HeatBound2} becomes trivially true (the left hand side is simply $0$) while the validity of~\eqref{e:HeatBound1} is guaranteed by Lemma~\ref{l:Schwartz} and the assumption $\zeta\in\CC^\beta_{\rho,T}$. 
\end{proof}

The following is instead a significant application of Proposition~\ref{p:HeatBounds}, in which we show how to apply it in the context of the Reconstruction Theorem~\ref{t:Reconstruction}. 

\begin{lemma}\label{l:ReconstructBound}
Let $\ST=(\CA,\CT, \CG)$ be a regularity structure according to Definition~\ref{def:RS}, $\alpha \eqdef \min \CA$ and let $Z = (\Pi, \Gamma, \Sigma)$ be a periodic model on $\ST$ with the corresponding reconstruction operator $\CR$. Moreover, fix $\gamma > 0$, $\eta>-2$ such that $\eta+2>\gamma$, $\beta \in \bigl[\alpha, \min\{\CA \setminus \{\alpha\}\}\bigr]$ and let $n$ be an integer in $(\beta + 1, \alpha + 2)$. For $z = (t,x) \in (0,T]\times\T$ and $H \in \CD^{\eta, \gamma}_T(Z)$, define the distribution
\begin{equ}[e:ReconstructCut]
\bigl(\CJ H\bigr)_{\beta, z}(s, \cdot) \eqdef \bigl(\CR_s H_s - \Pi^s_x \CQ_{< \beta} H_t(x)\bigr)(\cdot)
\end{equ}
where $\CQ_{<\beta}$ is the projection onto $\CT_{<\beta}$ (see Remark~\ref{r:QOperators}). Then, for any $\theta_1\in(0,\alpha-n+2)$ and $\theta_2\in\R$ such that $\theta\eqdef\theta_1-\theta_2>0$, the following bound holds
\begin{equ}[b:ReconstructionCut]
\sup_z\onorm{t}^{-(\eta-\gamma+\theta_2+\bar \beta)} \bigl|\bigl(P^{(n)} \ast \bigl(\CJ H\bigr)_{\beta, z}\bigr)(z)\bigr| + \sup_z\pnorm{P^{(n)} \ast \bigl(\CJ H\bigr)_{\beta, z}}_{ \eta -  \gamma+\theta_2, \bar \beta; T, z} \lesssim C T^{\frac{\theta}{2}}\;,
\end{equ}
where, the suprema run over $z=(t,x)\in (0,T]\times\T$, $P$ is the heat kernel, $P^{(n)}$ is the $n$-th spatial derivative of $P$ and $\bar \beta = \beta + 2 - n-\theta_1$. 

If furthermore $\bar Z = (\bar \Pi, \bar \Gamma)$ is another model, $\bar \CR_t$ the associated family of reconstruction operators and $\bar H \CD_T^{\gamma, \eta}(\bar Z)$, then~\eqref{e:ReconstructCut} still holds upon replacing $\CJ H$ by $\CJ H-\bar \CJ \bar H$ and with a constant $C$ proportional to 
\[
\bigl(\Vert H; \bar H \Vert_{\gamma, \eta; T} \Vert \Pi \Vert_{\gamma; T} + \Vert \bar H \Vert_{\gamma, \eta; T} \Vert \Pi - \bar \Pi \Vert_{\gamma; T} \bigr)\,.
\]
\end{lemma}

\begin{remark}
The reason why we focus on~\eqref{e:ReconstructCut} and not, as might seem more natural, on the whole left hand side of~\eqref{e:Reconstruction} will be clarified in the upcoming section. 
The point is that we only care about the ill-posed part of the product that, in our (and, also, more general) context, will be the one coming from the element in the regularity structure with lowest homogeneity. 
\end{remark}

\begin{proof}
We begin with~\eqref{b:ReconstructionCut}. In order to establish the connection with Proposition~\eqref{p:HeatBounds}, we define, for $z=(t,x)\in(0,T]\times\T$ and $s\in(0,T]$, $\zeta_z(s,\cdot)$ as the right hand side of~\eqref{e:ReconstructCut}. 
We want to verify that~\eqref{e:HeatBound} hold for $z\mapsto\zeta_z$. 
Let us begin by showing~\eqref{e:SchwartzCond} since it will imply~\eqref{e:HeatBound1}. Let $\varphi$ be a function in $\CB_0^r(\R)$ and $\lambda\in(0,1]$. 
Notice that
\begin{equation}
\bigl|\langle \zeta_{z}(t, \cdot), \varphi_x^\lambda\rangle\bigr| \leq \bigl| \langle \CR_t H_t - \Pi^t_x H_t(x), \varphi_x^\lambda \rangle\bigr| + \bigl| \langle \Pi^t_x \CQ_{\geq \beta} H_t(x), \varphi_x^\lambda \rangle\bigr|
\lesssim \lambda^\gamma \onorm{t}^{\eta - \gamma} + \sum_{\beta \leq l < \gamma} \lambda^l \onorm{t}^{(\eta - l) \wedge 0} \lesssim \lambda^\beta \onorm{t}^{\eta-\gamma}\label{b:BoundRecCut}
\end{equation}
where $\CQ_{\geq \beta}$ is the projection onto $\bigoplus_{l\geq\beta} \CT_l$ (recall that the components of $H$ in $\CT_l$ for $l>\gamma$ are $0$ by definition) and the bound on the first term follows by~\eqref{e:Reconstruction} while the second by the properties of the model and the definition of modelled distribution. 
Let $x\neq\bar x$, $\bar z=(t,\bar x)$ and $z=(t,x)$. Then
\begin{equ}[b:ModelEquality1]
\langle \zeta_{\bar z}(t, \cdot)-\zeta_z(t,\cdot), \varphi_x^\lambda\rangle=\langle \Pi_x^t\CQ_{<\beta}\big( H_t(x) -\Gamma^t_{x \bar x} H_t(\bar x)\big)+\Pi_x^t\CQ_{<\beta}\Gamma^t_{x \bar x} \CQ_{\geq \beta} H_t(\bar x), \varphi_x^\lambda\rangle
\end{equ}
where we made use of the algebraic properties of the model and of the easily verifiable equality $\Gamma^t_{x \bar x}\CQ_{<\beta} H_t(\bar x)=\CQ_{<\beta}\Gamma^t_{x \bar x} H_t(\bar x)-\CQ_{<\beta}\Gamma^t_{x \bar x} \CQ_{\geq \beta} H_t(\bar x)$.
At this point, we exploit the analytical properties of the model and the definition of modelled distributions to bound the first by
\begin{equ}
\lambda^\alpha\|H_t(x) -\Gamma^t_{x \bar x} H_t(\bar x)\|_\alpha\lesssim \lambda^\alpha |t|_0^{\eta-\gamma}|x-\bar x|^{\gamma-\alpha}\lesssim \lambda^{\alpha}|t|_0^{\eta-\gamma}|\bar x-x|^{\beta-\alpha}
\end{equ}
since $z$ and $\bar z$ live in a compact and $\gamma\geq \beta$, and the second by
\begin{equ}
\lambda^\alpha \|\Gamma^t_{x \bar x} \CQ_{\geq \beta} H_t(\bar x)\|_\alpha\lesssim \lambda^\alpha\sum_{ \beta\leq l<\gamma} |t|^{(\eta-l)\wedge 0}|x-\bar x|^{l-\alpha}\lesssim \lambda^{\alpha}|t|_0^{\eta-\gamma}|\bar x-x|^{\beta-\alpha}
\end{equ}
Then,~\eqref{b:BoundRecCut} and the previous allow to apply Lemma~\ref{l:Schwartz}, which in turn implies~\eqref{e:HeatBound1}. 

Consider now $t\neq\bar t$ such that $|\bar t-t|\leq \onorm{t,\bar t}^2$ and let this time $\bar z=(\bar t,x)\neq(t,x)=z$. Then, we have
\begin{equ}[b:ModelEquality2]
\langle \zeta_{\bar z}(t, \cdot)-\zeta_z(t,\cdot), \varphi_x^\lambda\rangle = \langle \Pi_x^t\CQ_{<\beta}\big( H_t(x) -\Sigma^{t \bar t}_x  H_{\bar t}( x)\big)+\Pi_x^t\CQ_{<\beta}\Sigma^{t \bar t}_x \CQ_{\geq \beta} H_t(\bar x), \varphi_x^\lambda\rangle
\end{equ}
which presents essentially the same terms as~\eqref{b:ModelEquality1}. Thus, following the same procedure as above but exploiting the properties of the map $\Sigma$ (instead of the ones of $\Gamma$), one immediately obtains
\begin{equ}
\bigl| \langle \zeta_{\bar z}(t, \cdot)-\zeta_z(t,\cdot), \varphi_x^\lambda\rangle\bigr| \lesssim\lambda^\alpha |t,\bar t|_0^{\eta-\gamma} |t-\bar t|^{\frac{\beta-\alpha}{2}}
\end{equ}
Concerning~\eqref{e:HeatBound2}, notice at first that the analytical and algebraic properties of a model and Lemma~\ref{l:Schwartz} guarantee that for any $\tau\in\CT_\alpha$ and $\psi\in\CS(\R)$ we have
\[
\langle\Pi_x^t\tau,\psi_x^\lambda\rangle\lesssim \lambda^\alpha\|\tau\|
\]
where the proportionality constant hidden in the previous bound depends only on the norm of $\psi$. Now, since in~\eqref{b:ModelEquality1} and~\eqref{b:ModelEquality2}, we showed that the increment $\zeta_{\bar z}(t, \cdot)-\zeta_z(t,\cdot)$ can be written as $\Pi_x^t\tau$ for some $\tau\in\CT_\alpha$, the validity of~\eqref{e:HeatBound2} follows at once.

The last part of the statement can be obtained applying the same scheme as before but exploiting the local Lipschitz continuity of the reconstruction operator, Theorem~\ref{t:Reconstruction}. 
\end{proof}

Before proceeding, we state the following lemma, whose proof is provided in~\cite[Lem.~7.5]{Hai14} or can be easily obtained using Lemma~\ref{l:Schwartz} and some of the tools exploited in Proposition~\ref{p:HeatBounds}.

\begin{lemma}\label{l:InitialCondition}
Let $\eta\in\R\setminus\Z$ and $u_0 \in \CC^{\eta}(\T)$ be periodic. Then $P u_0 \in \CC^{\gamma,\s}_{\eta, T}$ for any $\gamma \in (0,1)$, and one has the bound
\begin{equ}
 \Vert P u_0 \Vert_{\CC^{\gamma}_{\eta, T}} \lesssim \Vert u_0 \Vert_{\CC^{\eta}}\;.
\end{equ}
\end{lemma}

\subsection{The fixed point argument: proof of Theorem~\ref{t:FixedPoint}}\label{sec:FP}

Thanks to the results in the previous section, we now have all the necessary ingredients to suitably bound the map $\CM$ introduced in~\eqref{FPMap} and that we here recall. 
For $u_0\in\CC^\eta$ and $\X\in\CX$ we set, with a slight abuse of notation, $\CM(\cdot) \eqdef \CM (v^0, \X, \cdot)$. Now, let $V=(v,v',R)\in\CH^{\eta,\gamma}_{\X,T}$, then $\CM(V)=\tilde V=(\tilde v,\tilde v',\tilde R)$ is given by
\begin{subequations}\label{FixedPointMap}
\begin{align}
\tilde{v} &= P u_0 + 4 X^{\<tree1222>} + P' \ast \bigl( 2 \CR \bigl(\BV\<bigtree1>\bigr) + F_v\bigr) + Q\;,\label{e:FPMapv} \qquad \tilde{v}' = 4 X^{\<tree122>} + 2 v\;,\\
\tilde{R}(z, \bar z) &= \delta_{z, \bar z} \big(P u_0\big) + 4 R^{\<tree1222>}( z,\bar z)+\delta_{z, \bar z}\Big( P' \ast \big(2 \bigl(\CJ \BV\<bigtree1>\bigr)_{\beta_{\<tree21>}, z} + F_v  \big)+ Q+ \tilde v'(z) \hat X^{\<tree11>}\Big)\label{e:FPMapr}
\end{align}
\end{subequations}
where $\hat X^{\<tree11>}\eqdef \hat K'\ast X^{\<tree1>}$, $F_v$ and $Q$, as in~\eqref{e:F}, equal
\begin{equ}[e:Q]
 F_{v} \eqdef 2 X^{\<tree12>} \bigl(2 X^{\<tree122>} + v \bigr) + \bigl( 2 X^{\<tree122>} + v\bigr)^2\;, \qquad  Q \eqdef  X^{\<tree124>} +\hat K' \ast \Bigl( \xi + X^{\<tree2>} + 2 X^{\<tree22>} \Bigr)\;,
\end{equ}
and in~\eqref{e:FPMapr} we unwrapped relation~\eqref{e:VExpansion}, making use of the definition of $\tilde v'$, $X^{\<tree11>}$ and~\eqref{e:ReconstructCut}. 
We are now ready to state and prove the next proposition which represents the core of the fixed point argument. 

\begin{proposition}\label{prop:FixedPoint}
Let $\alpha_\star \in \bigl(-\alpha_{\<tree12>}, \alpha_{\<tree11>}\bigr)$, $\gamma \in \bigl(\alpha_{\<tree11>}, \alpha_{\<tree11>} + \alpha_\star \wedge \alpha_{\<tree122>}\bigr)$ and $\eta \in (-1,0)$. Let $u_0 \in \CC^{\eta}$ be periodic and $\X \in \mathcal{X}$. Then there exists $\theta>0$ such that the map $\CM$, defined in~\eqref{FixedPointMap}, satisfies, for all $V$, $\bar V \in \CH^{\eta, \gamma}_{\X,T}$ the following bounds
\begin{subequations}\label{FixedPointBounds}
\begin{gather}
\Vert \CM (V) \Vert _{\eta, \gamma; T} \lesssim \Vert u_0 \Vert _{\CC^\eta} + T^\theta \bigl(1+\Vert \X \Vert _{\CX}\bigr)^2\bigl(1 + \Vert V \Vert _{\eta, \gamma; T} \bigr)^2\;,\label{e:FixedPointBoundsFirst}\\
\Vert \CM (V) - \CM (\bar V) \Vert _{\eta, \gamma; T} \lesssim T^\theta \Vert V - \bar V \Vert _{\eta, \gamma; T} (1+\|\X\|_\CX)^2\;.\label{e:FixedPointBoundsSecond}
\end{gather}
\end{subequations}
where the second bound holds provided that $\Vert V\Vert _{\eta, \gamma; T},\,\Vert \bar V\Vert _{\eta, \gamma; T}\leq M<\infty$ and the hidden constant depends only on $M$. 
\end{proposition}

In the proof of the previous we will need the following lemma. 

\begin{lemma}\label{l:RemV}
Let $\alpha_{\star}\in(0,\alpha_{\<tree1>}+1)$, $\gamma>-\alpha_{\<tree1>}$ and $\eta\in(-1,0)$. Let $\X\in\CX$ and $Z=(\Pi,\Gamma,\Sigma)$ be the model given in~\eqref{e:Model}. Let $V=(v,v',R)\in\CH^{\gamma,\eta}_{\X, T}$ and $\BV\<bigtree1>$ be defined as in~\eqref{e:VDot}. 
For $z=(t,x)\in(0,T]\times\T$, let
\[
\zeta_z(s,\cdot)\eqdef \Pi_x^s\CQ_{<\beta}\big(\BV\<bigtree1>\big)_t(x)(s,\cdot)
\]
where $\CQ_{<\beta}$ is the projection onto $\CT_{<\beta}$. Then the following bound holds
\begin{equ}[b:PiVBeta]
\sup_z\onorm{t}^{-\eta} \bigl|\bigl(P^{(n)} \ast \bigl(\zeta_z\bigr)\bigr)(z)\bigr| + \sup_z\pnorm{P^{(n)} \ast \bigl(\zeta_z\bigr)}_{ \eta-\alpha_\star, \alpha_\star; T, z} \lesssim C T^{\frac{\theta}{2}} \;,
\end{equ}
where, the suprema run over $z=(t,x)\in (0,T]\times\T$, $P$ is the heat kernel, and $\theta$ can be chosen to be $\alpha_{\<tree1>}+1-\alpha_\star$. 

Moreover, if $\bar \X\in\CX$, $\bar Z = (\bar \Pi, \bar \Gamma)$, $\bar V\in\CH^{\eta,\gamma}_{\bar \X, T}$ and $\bar{\BV\<bigtree1>}$ are another controlling process, model, controlled process and modelled distribution respectively,  then, for $\bar \zeta$ defined as above,~\eqref{b:PiVBeta} still holds upon replacing $\zeta$ with $\zeta-\bar \zeta$. 
\end{lemma}
\begin{proof}
The proof proceeds along the same lines of Lemma~\ref{l:ReconstructBound}. We will focus at first on~\eqref{e:SchwartzCond}. 
Let $\varphi\in\CB_0^r(\R)$, $\lambda\in(0,1]$ and $z=(t,x)\in(0,T]\times\T$, then
\begin{equ}[b:PiInfty]
\big|\langle\zeta_z(t,\cdot),\varphi_x^\lambda\rangle\big|=\big|\langle \Pi_x^t\CQ_{<\beta_{\<tree21>}}\big(\BV\<bigtree1>\big)_t(x)(t,\cdot),\varphi_x^\lambda\rangle\big|=|v(t,x)|\big|\langle \Pi_x^t \<bigtree1>,\varphi_x^\lambda\rangle\big|\lesssim \|V\|_{\eta,\gamma;T}\onorm{t}^{\eta}\lambda^{\alpha_{\<tree1>}}\;.
\end{equ}
On the other hand, for $\bar z=(t,\bar x)$, with $\bar x\neq x$ we have
\begin{equ}
\langle\delta_{\bar z, z}\zeta(t,\cdot),\varphi_x^\lambda\rangle=\langle \Pi_x^t\CQ_{<\beta_{\<tree21>}}\big(\BV\<bigtree1>\big)_t(x)(t,\cdot)-\Pi_{\bar x}^t\CQ_{<\beta_{\<tree21>}}\big(\BV\<bigtree1>\big)_t(\bar x)(t,\cdot),\varphi_x^\lambda\rangle
= \delta^{(t)}_{\bar x, x}v\,\langle \Pi_x^t \<bigtree1>,\varphi_x^\lambda\rangle
\end{equ}
where the last equality is a consequence of the fact that the realization of $\<bigtree1>$ through the model does not depend on the base point. Now, if $|x-\bar x|>\onorm{t}$, we bound each summand of the previous separately using~\eqref{b:PiInfty}.
Otherwise, we get
\begin{equ}
\big|\delta^{(t)}_{\bar x, x}v\big|\big|\langle \Pi_x^t \<bigtree1>,\varphi_x^\lambda\rangle\big|\lesssim \|V\|_{\eta,\gamma;T} |t|^{\eta-\alpha_\star}\lambda^{\alpha_{\<tree1>}}|x-\bar x|^{\alpha_\star}\leq \|V\|_{\eta,\gamma;T}\onorm{t}^{\eta}\lambda^{\alpha_{\<tree1>}}\;.
\end{equ}
Hence, the assumption~\eqref{e:SchwartzCond} is matched and consequently~\eqref{e:HeatBound1} with $\rho=\eta-\alpha_\star$. 
Moreover notice that, taking $t\neq\bar t$ such that $|\bar t-t|\leq \onorm{t,\bar t}^2$ and $\bar z=(\bar t,x)\neq(t,x)=z$, we have
\begin{equ}
\langle\delta_{\bar z, z}\zeta(t,\cdot),\varphi_x^\lambda\rangle= \langle \Pi_x^t\CQ_{<\beta}\big(\BV\<bigtree1>\big)_{\bar t}(x)(t,\cdot)-\Pi_{ x}^t\CQ_{<\beta}\big(\BV\<bigtree1>\big)_t(x)(t,\cdot),\varphi_x^\lambda
= \delta^{(\bar t, t)}_{x}v\,\langle \Pi_x^t \<bigtree1>,\varphi_x^\lambda\rangle\;.
\end{equ}
Proceeding as above, separately taking into account the cases $|t-\bar t|\leq |t,\bar t|_0^2$ and $|t-\bar t|> |t,\bar t|_0^2$,one obtains the same bound. 

But now, also the validity of~\eqref{e:HeatBound2} with $\alpha=\alpha_{\<tree1>}$ and $\beta = \beta_{\<tree21>}$ can be obtained exploiting the same argument presented at the end of the proof of Proposition~\ref{p:HeatBounds}, so that, upon choosing $\theta_1=\alpha_{\<tree1>}+1-\alpha_\star$ and $\theta_2=0$ we are done. 
\end{proof}

\begin{proof}[Proof of Proposition~\ref{prop:FixedPoint}]
In order to prove the statement, we need to bound each of the norms appearing in definition~\eqref{e:Norms}. 
As we will see, this will essentially amount to verify that the assumptions of Proposition~\ref{p:HeatBounds} are satisfied.

We begin with the $\CC^{\alpha_\star,\s}_{\eta,T}$ of $\tilde v$, the latter being defined in \eqref{e:FPMapv}. 
The bound on the initial condition follows immediately from Lemma~\ref{l:InitialCondition}. 
Concerning the terms of the form $X^\tau$, for $\tau\in\{\<tree1222>,\,\<tree124>\}$, the bound follows by the fact that they belong to $\X$. Let us write, $F_v^{(i)}$, $i=1,2$, for the first and second summand in the expression of $F_v$, as given in~\eqref{e:Q}. Then by Proposition~\ref{p:ClassicalProduct}, they satisfy
\begin{equ}[b:Fv]
 \big| \langle F^{(1)}_v(t, \cdot), \varphi_x^\lambda \rangle \big| \lesssim \lambda^{\alpha_{\<tree12>}} \onorm{t}^{\eta - \alpha_\star} \;,\quad
 \big| \langle F^{(2)}_v(t, \cdot), \varphi_x^\lambda \rangle \big| \lesssim \onorm{t}^{2\eta} \big( \Vert \X \Vert_{\CX} + \Vert v \Vert_{\CC_{\eta, T}^{0}} \big)^2\lesssim \lambda^{-\sigma}\onorm{t}^{2\eta}
\end{equ}
for every $t \in (0, T]$ and $\sigma\in(0,\eta+1)$, so that thanks to Proposition~\ref{p:HeatBounds}, upon choosing for the first $\theta_1^1=\alpha_{\<tree12>}+1-\alpha_\star$, $\theta_2^1=0$ and for the second $\theta_1^2=1-\sigma-\alpha_\star$, $\theta_2^2=-\eta-\alpha_\star$, we deduce
\begin{equ}
\|P'\ast F^{(1)}_v\|_{\CC^{\alpha_\star,\s}_{\eta,T}}+\|P'\ast F^{(2)}_v\|_{\CC^{\alpha_\star,\s}_{\eta,T}}\lesssim T^{\frac{\alpha_{\<tree12>}+1-\alpha_\star}{2}}+ T^{\frac{\eta+1-\sigma}{2}}
\end{equ}
where, by assumption, the exponents of $T$ are strictly positive. 
Since the function $\hat K$ is smooth and compactly supported, the convolution with $\hat K'$ is a continuous operation and the required bound on the term $Q$ in \eqref{e:FPMapv} follows straightforwardly. 
We can now deal with the term containing the reconstruction operator. We see that for $z=(t,x)\in(0,T]\times\T$ we have
\[
P'\ast(\CR(\BV\<bigtree1>))=P'\ast\bigl(\CJ \BV\<bigtree1>\bigr)_{\beta_{\<tree21>}, z}+P'\ast\Big(\Pi_x^\cdot\CQ_{<\beta_{\<tree21>}}\big(\BV\<bigtree1>\big)_t(x)\Big)\;.
\]
Now, while the second is taken care by Lemma~\ref{l:RemV}, since by Proposition~\ref{p:ModelledDistribution} we know that $\BV\<bigtree1>\in\CD^{\gamma+\alpha_{\<tree1>},\eta+\alpha_{\<tree1>}}_T$ and $\gamma+\alpha_{\<tree1>}>0$, $\eta>\gamma+2$, $\eta+\alpha_{\<tree1>}>-2$, we can directly apply Lemma~\ref{l:ReconstructBound} to the first summand with $\theta_1=\beta_{\<tree21>}+1-\alpha_\star$ and $\theta_2=\gamma-\alpha_\star$ so that its $\CC^{\alpha_\star,\s}_{\eta,T}$-norm is bounded by $T^{\frac{\beta_{\<tree21>}+1-\gamma}{2}}$. 

For the Gubinelli's derivative $\tilde{v}'$, defined in \eqref{e:FPMapv}, we immediately get
\begin{equs}
 \Vert \tilde v' \Vert _{\CC^{\gamma - \alpha_{\<tree11>},\s}_{\eta - \alpha_{\<tree11>}, T}} \lesssim \Vert X^{\<tree122>} \Vert _{\CC^{\gamma - \alpha_{\<tree11>},\s}_{\eta - \alpha_{\<tree11>}, T}} + \Vert v \Vert _{\CC^{\gamma - \alpha_{\<tree11>},\s}_{\eta - \alpha_{\<tree11>}, T}} \lesssim T^\theta \big(\Vert \X\Vert _{\CX} + \Vert v \Vert _{\CC^{\alpha_{\<tree12>},\s}_{\eta, T}}\big)\;,
\end{equs}
for some $\theta > 0$, as soon as $\alpha_{\<tree11>} < \gamma < \alpha_{\<tree11>} + \alpha_\star$.

At last we look at the seminorm involving $R$, whose expression is given in~\eqref{e:FPMapr}. The increment of the terms containing initial condition can be dealt with as before. 
The bound on $R^{\<tree1222>}$ follows by definition, while for $\delta_{z, \bar z} P' \ast \bigl(\CJ \BV\<bigtree1>\bigr)_{\beta_{\<tree21>}, z}$, Lemma~\ref{l:ReconstructBound} and in particular~\eqref{b:ReconstructionCut}, along with the observations made before, gives
\begin{equ}
\sup_z\pnorm{P' \star \bigl(\CJ \BV\<bigtree1>\bigr)_{\beta_{\<tree21>}, z}}_{ \eta -  \gamma, \gamma; T, z}\lesssim T^{\frac{\beta_{\<tree21>}+1-\gamma}{2}}\sup_z\pnorm{P' \star \bigl(\CJ \BV\<bigtree1>\bigr)_{\beta_{\<tree21>}, z}}_{ \eta -  \gamma, \beta_{\<tree21>}+1; T, z}\;.
\end{equ}
For $F_v^{(i)}$, $i=1,2$, we make use once more of~\eqref{b:Fv} but with $\theta_1^1=\alpha_{\<tree12>}+1-\gamma$, $\theta_2^1=\alpha_\star-\gamma$, $\theta_1^2=1-\sigma-\gamma$, $\theta_2^2=-\eta-\gamma$, while the increments of $Q$ are again bounded straightforwardly. 
The last summand is such that $\delta_{z,\bar z} \hat X^{\<tree11>}$ is smooth and $|\tilde v'(z)|\lesssim \onorm{z}^{\eta-\alpha_{\<tree11>}}\leq T^{\frac{\gamma-\alpha_{\<tree11>}}{2}}\onorm{t}^{\eta-\gamma}$. 

The inequality~\eqref{e:FixedPointBoundsSecond} can be shown following essentially the same steps exploited so far, so that the proof is complete. 
\end{proof}

Thanks to the previous proposition, we have now all the elements to prove Theorem~\ref{t:FixedPoint}.

\begin{proof}[Proof of Theorem~\ref{t:FixedPoint}]
It follows from \eqref{e:FixedPointBoundsFirst} that for fixed $u_0 \in \CC^{\eta}$, $\X \in \CX$ and a sufficiently small $T > 0$ there exists a ball in $\CH^{\eta, \gamma}_{\X,T}$ which is left invariant by $\CM$. 
Furthermore, \eqref{e:FixedPointBoundsSecond} and the Banach fixed point theorem imply that $\CM$ admits a unique fixed point in this ball. The uniqueness of this fixed point on whole $\CH^{\eta, \gamma}_{\X,T}$ follows from a simple argument as in \cite[Thm.~4.8]{Hai14}. We can now restart the procedure at time $T$ and so and so on, constructing in this way the maximal solution to the fixed point problem. 
The solution to~\eqref{e:SBE} is then built from~\eqref{e:uExp}. 

Concerning the local Lipschitzianity of the map, using the same procedure exploited in the proof of the previous proposition and the local Lipschitz-continuity of all the operations we performed (along with a procedure analog to the proof of~\cite[Prop. 8]{Gub04}), it is not difficult to see that, given $\X$, $\bar{\X}\in\CX$ and $u_0$, $\bar{v}_0\in\CC^{\eta}$ such that
\begin{equ}
\max\left\{\Vert u_0\Vert _{\CC^\eta},\Vert \bar{v}_0\Vert _{\CC^\eta},\Vert \X\Vert _\CX,\Vert \bar{\X}\Vert _\CX\right\}\leq R
\end{equ}
for a fixed $R>0$, and letting $V \in \CH^{\eta, \gamma}_{\X,T}$ and $\bar{V} \in \CH^{\eta, \gamma}_{\bar \X,T}$ be the respective fixed point of the of equation starting at $u_0$ (resp. $\bar{v}_0$) and controlled by $\X$ (resp. $\bar{\X}$) determined above, then the following inequality holds
\begin{equ}
\Vert V - \bar V \Vert_{\eta, \gamma; T} \lesssim \Vert u_0-\bar{v}_0\Vert _{\CC^\eta}+\Vert X-\bar{X}\Vert _\CX
\end{equ}
and, invoking again~\eqref{e:uExp}, the proof of Theorem~\ref{t:FixedPoint} is concluded. 
\end{proof}

\subsection{Discrete heat kernel and Schauder estimates}\label{sec:DConv}

In this section we want to show that it is possible to perform the same operations and obtain the same bounds as in the continuous case but in the discrete setting. As before, we fix $N\in\N$, set $\eps\eqdef 2^{-N}$ and define the grids as in~\eqref{e:Grids}, and $\T_\eps\eqdef \T\cap\Lambda_\eps$. Let us begin with the discrete counterpart of Lemma~\ref{l:Schwartz}. 

\begin{lemma}\label{l:DSchwartz}
Let $\alpha \in \R$, $\gamma \geq 0$, and let the maps $\T_\eps \ni x \mapsto \zeta^\eps_x \in \R^{\T_\eps}$ satisfy
\begin{equ}
|\langle \zeta^\eps_x, \varphi_x^\lambda\rangle_\eps | \leq C \lambda^{\alpha}\;, \qquad |\langle \zeta^\eps_x - \zeta^\eps_y, \varphi_x^\lambda\rangle_\eps | \leq C \lambda^{\alpha - \gamma} |x-y|^{\gamma}\;,
\end{equ}
uniformly in $\lambda \in [\eps,1]$, for some $C > 0$, all $\varphi \in \CB^r_0(\R)$ with integer $r > |\alpha|$, and locally uniformly over $x, y \in \Lambda_\eps$ such that $|x - y|\geq \lambda$, where we have identified $\zeta^\eps$ with its periodic extension to $\Lambda_\eps$. Let $\delta>0$ be such that $\gamma-\delta<-1$. Then, for any family of maps $\psi^\lambda\in\CC^r(\R)$, parametrized by $\lambda\in I_\eps\eqdef \Lambda_{\eps}\cap [\eps,1]$ such that $\sup_{j\leq r}\sup_x |x|^\delta|\partial^j\psi^\lambda(x)|\lesssim \lambda^{-1-r+\delta}$, one has
\begin{equ}
 |\langle \zeta^\eps_x, \psi_x^\lambda\rangle_\eps | \lesssim C \lambda^{\alpha}\;,
\end{equ}
(locally) uniformly over $x\in\Lambda_\eps$, $\lambda \in I_\eps$ and where the proportionality constant depends only on $\psi$.
\end{lemma}

\begin{remark}\label{rem:DSchwartz}
The condition imposed on the family $\psi^\lambda$ is the same mentioned in Remark~\ref{rem:Schwartz}, but  we are allowing $\lambda$ to take value in a discrete set. 
\end{remark}

\begin{proof}
The proof is identical to that of Lemma~\ref{l:Schwartz}, with the only difference that we need to use the points from the grid. 
\end{proof}

In order to be able to follow the same steps of Section~\ref{sec:ContConv}, we need to show that the convolution with the discrete heat kernel can be regarded as the testing against the recentered family $\psi^\lambda$ introduced above, in which time will play the role of $\lambda$. 
Recall that the space-time discrete heat kernel, $P^\eps$, is the unique solution of \eqref{eq:DHeat}. From the previous equation, we immediately deduce that for all $(t,x)\in\Lambda_\eps^\s$ one has
\begin{equ}
P^\eps_t(x) = \eps^{-1} \1_{t\geq 0} \bigl((1 + \eps^2 \Delta_\eps)^{t\eps^{-2}} \delta_{0,\cdot}\bigr)(x)\;.
\end{equ}
In order to derive suitable bounds on the kernel just defined we need at first to build a regular extension of $P^\eps$ off the grid $\Lambda_\eps^\s$. 
While the time component can simply be dealt with by taking $t\in\R$ and the integer part of $t\eps^{-2}$ at the exponent, for the space one we briefly recall the construction done in~\cite[Sec. 5.1]{HM15}. 
Let $\tilde{\varphi}(x)\eqdef \frac{\sin(\pi x)}{\pi x}$ and $\bar{\varphi}$ be a smooth function compactly supported in a ball of radius $\frac{1}{4}$ around the origin and integrating to 1. Then, let $\psi$ be the function whose Fourier transform is $\CF\tilde{\varphi}\ast \bar{\varphi}$. Notice that $\psi$ belongs to $\CS(\R)$, coincides with the Kronecker's function on $\Z$ and $\CF \psi$ is supported on $\{\zeta:|\zeta|\leq \frac{3}{4}\}$. Indicating as usual its rescaled version as $\psi^\eps$, i.e. $\psi^\eps(x)\eqdef \eps^{-1}\psi(x/\eps)$, we obviously have the identity
\begin{equation}\label{e:DSTHeatKernel}
P^\eps_t(x) = \1_{t\geq 0}\bigl((1 + \eps^2 \Delta_\eps)^{\lfloor t\eps^{-2}\rfloor}\psi^\eps \bigr)(x)\;,\qquad(t,x) \in \Lambda_\eps^\s\;.
\end{equation}
But now we can define $\tilde P^\eps$ according to the right hand side of the previous but for $(t,x)$ varying in $\R^2$. Notice that, by construction $\tilde P^\eps_t(x)=P^\eps_t(x)$ for $(t,x)\in\Lambda_\eps^\s$ so that for a function $f$ defined on $\Lambda_\eps^\s$, we have $\tilde P^\eps\ast_\eps f(t,x)=P^\eps\ast_\eps f(t,x)$, where $\ast_\eps$ denotes the space-time discrete convolution. 
From now on, with a slight abuse of notation, we will indicate the extension $\tilde P^\eps$ by $P^\eps$, because no confusion can arise. 

\begin{remark}
Let $D_{x,\eps}$ be the discrete derivative as in~\eqref{e:DLaplacian}, $j\in\N$ and assume that Assumption~\ref{a:pi} is satisfied. The same extension procedure that lead to the definition of $P^\eps$ can be performed on $D_{x,\eps}^jP^\eps$ without any modification and, from now on, we will denote by $D_{x,\eps}^jP^\eps$ such extensions.
\end{remark}

The following lemma provides some bounds on this kernel and its discrete derivative \textit{uniform} in $\eps$, which will allow to apply Lemma~\ref{l:DSchwartz} and use it in the same way we exploited its continuous counterpart. 

\begin{lemma}\label{lemma:DHeatKernels}
Let $N\in\N$, $\eps=2^{-N}$ and Assumptions~\ref{a:nu} and \ref{a:pi} be in place. Let $\Delta_\eps$, $D_{x,\eps}$, be respectively the discrete Laplacian and derivative introduced in~\eqref{e:DLaplacian}, and $P^\eps$ the extension to $\R^2$ of the kernel defined in~\eqref{e:DSTHeatKernel}. Then, for any $c>0$ fixed, and $j,\,k$ and $m\in\N$ such that $m< \lfloor c\rfloor\vee \lfloor c^{-1}\rfloor$, there exists a constant $C>0$, depending only on $j, \,k$ and $m$, such that the following bound 
\begin{equ}[b:DHeatKernel]
\bigl|\partial_x^k D_{x,\eps}^j P^\eps_t (x)\bigr| \leq C   |t|_{\eps}^{-1 -j- k+m}|x|^{-m}
\end{equ}
holds uniformly over $z=(t,x)\in\R^2$ such that $\|z\|_{\s}\geq  c\eps$, where $|t|_\eps= |t|^{1 / 2} \vee \eps$. In case $\|z\|_\s<c\eps$, upon taking $m=0$,~\eqref{b:DHeatKernel} is still valid.
\end{lemma}

\begin{proof}
The proof of the statement follows the same lines as the one of~\cite[Lem. 5.3]{HM15} but we need to adapt it to the present situation. Let $j\in\N$ and, for $(t,x)\in\R^2$ and $\eps>0$, set 
\begin{equ}
F^\eps_t(x)\eqdef |t|_\eps^{1+j} \left(D_{x,\eps}^jP^\eps_t\right)(|t|_\eps x)\;.
\end{equ}
At first, we aim at obtaining bounds for $F^\eps$ uniform in $t$, $x$ and $\eps$. To this purpose, consider its spatial Fourier transform, which is given by
\begin{equ}[def:FTDHeatKernel]
\widehat{F^\eps_t}(\xi)=\widehat{\psi}(\eps |t|_\eps^{-1} \xi)\left(\frac{\widehat{\pi}(-\eps|t|_\eps^{-1}\xi)}{\eps|t|_\eps^{-1} }\right)^j \left(1+\frac{\widehat{\nu}(\eps|t|_\eps^{-1}\xi)}{2\bar{\nu}}\right)^{\lfloor t\eps^{-2}\rfloor}\;.
\end{equ}
Take for now $c>1$ and $m\in\N$ such that $m< \lfloor c\rfloor$.  Let us begin with the case $t\geq c\eps^2$ and we first study the situation in which $\xi\in\R$ is such that $\eps|t|_\eps^{-1} |\xi|\leq \frac{3}{4}$. We want to control the $m$-th derivative of~\eqref{def:FTDHeatKernel} and, thanks to the generalized Leibniz rule, we can separately treat each of the factors there appearing. Now, by construction, $\hat{\psi}$ is a $\CC^\infty$-function with compact support hence for any $m\in\N$,  
\[
\left|\partial_\xi^m\left(\widehat{\psi}(\eps |t|_\eps^{-1} \xi)\right)\right|= (\eps|t|_\eps^{-1})^m\Big|\big(\partial_\xi^m\widehat{\psi}\big)(\eps |t|_\eps^{-1} \xi)\Big|\lesssim \|\partial_\xi^m\widehat{\psi}\|_\infty
\]
where the last passage comes from the fact that, since $|t|\geq c\eps^2$, $\eps|t|_\eps^{-1}\lesssim 1$. For the second factor, we point out that, by Assumption~\ref{a:pi}, $\pi$ has compact support, hence, for $l\in\N$, $l\geq 1$, we get
\begin{equation}\label{b:DDer1}
\left|\partial_\xi^l\left(\frac{\widehat{\pi}(-\eps|t|_\eps^{-1}\xi)}{\eps|t|_\eps^{-1} }\right)\right|\lesssim (\eps|t|_\eps^{-1})^{l-1}\int|x|^l|\pi|(\dd x)\lesssim 1
\end{equation}
where once again we are using $\eps|t|_\eps^{-1}\lesssim 1$. On the other hand for $l=0$, thanks to Assumption~\ref{a:pi}, in particular \eqref{e:DDerivativeProps}, and Taylor's formula, we have
\begin{equation}\label{b:DDer2}
\left|\frac{\widehat{\pi}(-\eps|t|_\eps^{-1}\xi)}{\eps|t|_\eps^{-1} }\right|=\left|\int x \int_0^1 e^{i\eps|t|_\eps^{-1}\lambda\xi x}\dd\lambda \pi(\dd x)\right|\lesssim \int|x||\pi|(\dd x)\lesssim 1\;.
\end{equation}
In conclusion this implies that $\widehat{\pi}(-\eps|t|_\eps^{-1}\cdot)/(\eps|t|_\eps^{-1})$ converges to the identity in $\CC^\infty$, therefore its $j$-th power is uniformly bounded by $|\xi|^j$. We now come to the third and last factor, which is the one that guarantees the polynomial decay. Let $m\in\N$ be such that $m< \lfloor c\rfloor$. By Fa\`a di Bruno formula, we can write the $m$-th derivative of the abovementioned term as
\begin{equation}\label{b:ThirdFactor}
\partial_\xi^m \left(\left(1+\frac{\widehat{\nu}(\eps|t|_\eps^{-1}\xi)}{2\bar{\nu}}\right)^{\lfloor t\eps^{-2}\rfloor}\right) =\left(1+\frac{\widehat{\nu}(\eps|t|_\eps^{-1}\xi)}{2\bar{\nu}}\right)^{\lfloor t\eps^{-2}\rfloor}\tilde{T}^{\eps,m}_t(\xi)
\end{equation}
with $\tilde{T}^{\eps,m}$ defined by
\[
\sum \tilde{c}\left(1+\frac{\widehat{\nu}(\eps|t|_\eps^{-1}\xi)}{2\bar{\nu}}\right)^{-\tilde{n}}\prod_{l=1}^{\tilde{n}}\frac{\lfloor t\eps^{-2}\rfloor-l}{t\eps^{-2}}\prod_{l=1}^m \left[(\eps|t|_\eps^{-1})^{l-2}\widehat{\nu}^{(l)}(\eps|t|_\eps^{-1}\xi)\right]^{n_l}
\]
where the sum runs over $(n_1,\cdots,n_m)\in\N^m$ such that $\sum_l l n_l=m$, $\tilde{c}$ is a constant depending on the $m$-tuple and $\bar{\nu}$, $\tilde{n}\eqdef\sum_l n_l\leq m$ and we exploited the fact that, since $t\geq c\eps^2$, $t\eps^{-2}=(\eps|t|_\eps^{-1})^2$. It is not difficult to see that $\tilde{T}$ is uniformly bounded in $\xi$, $t$ and $\eps$. Indeed, by the very definition of $\widehat{\nu}$ and the fact that $\nu$ is symmetric by Assumption~\ref{a:nu}, $1+\widehat{\nu}(\cdot)/(2\bar{\nu})$ is bounded away from $0$ (actually it is greater or equal to $1/2$) hence the first factor does not create any problem. While the second is obvious, for the third we proceed as for the second factor of~\eqref{def:FTDHeatKernel} but exploiting assumption Assumption~\ref{a:nu} instead of Assumption~\ref{a:pi}. To be more precise, for $l\in\N$, $l\geq 2$, the bound is identical to~\eqref{b:DDer1}, and we have
\[
\left|\partial_\xi^l\left(\frac{\widehat{\nu}(\eps|t|_\eps^{-1}\xi)}{\eps^2|t|_\eps^{-2} }\right)\right|=(\eps|t|_\eps^{-1})^{l-2}\widehat{\nu}^{(l)}(\eps|t|_\eps^{-1}\xi)\lesssim (\eps|t|_\eps^{-1})^{l-2}\int|x|^l|\nu|(\dd x)\lesssim 1
\]
for $l=0,\,1$, we follow~\eqref{b:DDer2} and use Assumption~\ref{a:nu}, in particular~\eqref{e:DLaplacianProps}, and Taylor's formula, so that
\[
\left|\frac{\widehat{\nu}^{(l)}(-\eps|t|_\eps^{-1}\xi)}{\eps^{2-l}|t|_\eps^{-(2-l)} }\right|=\left|\int x^2 \int_0^1 e^{i\eps|t|_\eps^{-1}\lambda\xi x}\dd\lambda \nu(\dd x)\right|\lesssim \int|x|^2|\nu|(\dd x)\lesssim 1\;.
\] 
We can now focus on the main part of the proof, i.e. the one concerning the first term on the right hand side of~\eqref{b:ThirdFactor}. Let us rewrite it as
\begin{equs}
\left(1+\frac{\widehat{\nu}(\eps|t|_\eps^{-1}\xi)}{2\bar{\nu}}\right)^{\lfloor t\eps^{-2}\rfloor}&=e^{\frac{\lfloor t\eps^{-2}\rfloor}{t\eps^{-2}} t\eps^{-2} \log\big(1+\widehat{\nu}(\eps|t|_\eps^{-1}\xi)/(2\bar{\nu})\big)}\\
&=e^{-\frac{\lfloor t\eps^{-2}\rfloor}{t\eps^{-2}}t|t|_\eps^{-2}\xi^2f(\eps|t|_\eps^{-1}\xi)\int_0^1 \frac{1}{1+\lambda\widehat{\nu}(\eps|t|_\eps^{-1}\xi)/(2\bar{\nu})}\dd \lambda}\lesssim e^{-\xi^2f(\eps|t|_\eps^{-1}\xi)}\lesssim e^{-2c_f \xi^2}\;,
\end{equs}
where in the passage from the first to the second line we applied Taylor's formula and introduced the function $f(x)\eqdef -x^{-2}\hat{\nu}(x)$,  while  to understand in the inequality right after, it suffices to recall that $t|t|_\eps^{-2}=1$ since $|t|>c\eps^2$ and that, as pointed out above, $\widehat{\nu}(\cdot)/(2\bar{\nu})\geq-\frac{1}{2}$. 
To understand the reason why also the last bound holds, notice that, by Assumption~\ref{a:nu}, there exists a constant $c_f>0$ such that $f(x)$ is bounded from below by $c_f$ for all $|x|\leq\frac{3}{4}$. Summarizing what obtained so far, we have
\begin{equation}\label{bound:FTDHeatKernel}
\left|\partial_\xi^m\widehat{F^\eps_t}(\xi)\right|\lesssim |\xi|^{m+j}e^{-2c_f \xi^2}\lesssim \bigl(1 + |\xi| \bigr)^{-k}\;,
\end{equation}
for any $k\in\N$, $\xi$ such that $\eps|t|_\eps^{-1}|\xi|\leq \frac{3}{4}$ and $|t|\geq c\eps^2$. The very same bound holds trivially also for $\eps |t|_\eps^{-1} |\xi|>\frac{3}{4}$ since, by construction, $\widehat{\psi}\equiv0$. 
At last, in case $|t|<c\eps^2$,~\eqref{def:FTDHeatKernel} becomes
\[
\widehat{F^\eps_t}(\xi)=\hat{\psi}(\xi)\left(\widehat{\pi}(-\xi)\right)^j \left(1+\widehat{\nu}(\xi)/(2\bar{\nu}\right)^{\lfloor t\eps^{-2}\rfloor}
\]
and the second factor can be bounded by the total variation of the measure $\pi$ while the last is trivially bounded by a constant. Ultimately, the polynomial decay of the $m$-th derivative of $\widehat{F^\eps}$ comes from $\widehat{\psi}$, and~\eqref{bound:FTDHeatKernel} holds also in case $|t|< c\eps^2$. 

At this point we exploit the continuity properties of the Fourier transform, that guarantee that $F^\eps$ is $\CC^\infty$ (since its Fourier transform decays faster than any polynomial) 
and that its derivatives decay faster than any polynomial of degree less or equal to $\lfloor c\rfloor$ (since, vice versa, its Fourier transform is $m$ times continuously differentiable). 
In other words we have, for any $k,\,m\in\N$ such that $m< \lfloor c\rfloor$, there exists a constant $C$, depending only on $k$ and $m$, for which
\[
\left|\partial_x^k F^\eps_t(x)\right|\leq C(1+|x|)^{-m}\quad\Longrightarrow\quad \left|\partial_x^k D_{x,\eps}^jP^\eps_t(|t|_\eps x)\right|\leq C|t|_\eps^{-1-j-k}(1+|x|)^{-m}
\] 
uniformly in $z=(t,x)\in\R^2$ and $\eps>0$. The bound~\eqref{b:DHeatKernel} is now a direct consequence of the previous since it suffices to change variables ($x\mapsto |t|_\eps^{-1}|x|$). For $\|z\|_\s< c\eps$, notice that the second factor in the last inequality is bounded by a constant.

At last we remark that, for $c<1$, we need to distinguish another case, namely $|t|\in[c\eps^2,c^{-1}\eps^2)$. While for $|t|\geq c^{-1}\eps^2$ and $|t|< c\eps^2$ we proceed as before, the just mentioned situation can be dealt with as in the latter case.
\end{proof}

At this point we can formulate and prove the following proposition, which represents the discrete analogue of Proposition~\ref{p:HeatBounds}.

\begin{proposition}\label{p:DHeatBounds}
In the same setting as Proposition~\ref{p:HeatBounds}, let $T\in(0,1]$ and consider a map $\Lambda_{\eps^2,T}\times\T_\eps \ni (t,x)= z \mapsto \zeta^\eps_z \in \ell^\infty\bigl(\Lambda_{\eps^2,T}, \R^{\T_\eps}\bigr)$ for which there exists a constant $C>0$ (independent of $\eps$) such that
\begin{subequations}
\begin{equs}
\big|\langle \zeta^\eps_{z}(t, \cdot), \varphi_x^\lambda\rangle_\eps\big| \leq C \lambda^{\beta} \enorm{t}^{\rho}\;, \qquad \big|\langle \zeta^\eps_{\bar z}(t, \cdot)-\zeta^\eps_z(t,\cdot), \varphi_x^\lambda\rangle_\eps\big| \leq C \lambda^{\alpha} |t,\bar t|_\eps^{\rho}|\bar z-z|_\s^{\beta-\alpha}\;,
\end{equs}
\end{subequations}
uniformly over $z=(t,x),\,\bar z=(\bar t, \bar x) \in \Lambda_{\eps^2,T}\times\T_\eps$ such that $|\bar t-t|_\eps\leq\enorm{\bar t, t}^2$, $\lambda \in [\eps,1]$ and $\varphi^\lambda\in\CC^r(\R)$ a family of maps parametrized by $\lambda\in I\eqdef\Lambda_{\eps}\cap [\eps,1]$ for which $\sup_{j\leq r}\sup_x |x|^\delta|\partial^j\varphi^\lambda(x)|\lesssim \lambda^{-1-r+\delta}$, where $\beta-\alpha-\delta<-1$. 
Then, for the same values of $\theta_1$, $\theta_2$, $\theta$, $\bar \beta$ and $\bar \rho$ as in Proposition~\ref{p:HeatBounds} the following bound holds
\begin{equ}[e:DConvBound]
\sup_{z} \enorm{t}^{-(\bar\rho+\bar \beta)} \bigl|\bigl(P^{(\eps, n)} \ast_\eps \zeta^\eps_{z}\bigr)(z)\bigr| +  \sup_{z }\pnorm{P^{(\eps, n)} \ast_\eps \zeta^\eps_{z}}^{(\eps)}_{ \bar\rho, \bar \beta; T, z} \lesssim C \enorm{T}^{\theta}
\end{equ}
uniformly in $\eps$, where the suprema run over $z=(t,x)\in \Lambda_{\eps^2,T}\times\T_\eps$, $P^\eps$ is the discrete heat kernel and $P^{(\eps, n)} = D^n_{x,\eps} P^\eps$ and $\ast_\eps$ is the discrete space-time convolution. 
\end{proposition}

\begin{proof}
The proof of this proposition is extremely similar to the one of Proposition~\ref{p:HeatBounds} so we limit ourselves to point out the differences and describe the adjustments to be made. 

Let us fix $c<1$ such that $\delta<\lfloor c^{-1}\rfloor$. At first, notice that we can always write the discrete convolution of the discrete heat kernel with a general map $f$ on $\Lambda_\eps^\s$ as 
\begin{equation}\label{e:splitting}
P^{(\eps, n)} \ast_\eps f(z)=\eps^2\sum_{s\in\Lambda_{\eps^2,t}}\langle f(s,\cdot),P^{(\eps,n)}_{t-s}(x-\cdot)\rangle_\eps
=\Big(\eps^2\sum_{\substack{s\in\Lambda_{\eps^2,t}\\t-s\geq c\eps^2}}+\eps^2\sum_{\substack{s\in\Lambda_{\eps^2,t}\\t-s< c\eps^2}}\Big)\langle f(s,\cdot),P^{(\eps,n)}_{t-s}(x-\cdot)\rangle_\eps
\end{equation}
where $z=(t,x)\in\Lambda_\eps^s$. In all the estimates we need to obtain, we will always split the discrete convolution as above and separately bound each of the sums appearing at the last member of the previous equality. 

For the first, we can apply exactly the same strategy as in the proof of Proposition~\ref{p:HeatBounds}, i.e. we exploit the discrete analog of~\eqref{e:HeatFour} (which also holds in this context by assumption) for the discrete pairing in space and then bound the Riemann-sum approximation by the actual integral in time over $[0,t]$. 
Recall that in case $t-s\geq|x-\bar x|^2$ we used Taylor's theorem to rewrite the spatial increments of $P^{(n)}$ and, thanks to Lemma~\ref{lemma:DHeatKernels}, we can do the same here. When on the other hand, we considered the situation $t-s\geq\bar t- t$ we exploited again Taylor's formula and the equality $\partial_t P_t=\Delta P_t$, which of course now means that we have to rewrite the increment $(P^{(\eps,n)}_{\bar t-s}-P^{(\eps,n)}_{t-s})(x-\cdot)$ as a telescopic sum from $t$ to $\bar t$ and use the fact that $P^\eps$ solves the discrete heat equation \eqref{eq:DHeat} and that there exists $\pi$ such that $\Delta_\eps = D_{x,\eps}^2$ (and consequently refer to Lemma~\ref{lemma:DHeatKernels}).

For $t-s<c\eps^2$, $P^{(\eps,n)}_t(x)=\eps^{-n-1}\delta_{0,\cdot}$, where $\delta_{\cdot,\cdot}$ is the Kronecker's delta function on $\Z$, therefore the $L^\infty$ estimate reads
\begin{equ}
\eps^2\sum_{\substack{s\in\Lambda_{\eps^2,t}\\ t-s<c\eps^2}} \big|\langle\zeta_{z}(s,\cdot), P^{(n,\eps)}_{t-s}(x-\cdot)\rangle_\eps|=\eps^{2-n}|\zeta_z(z)|
=\eps^{2-n}\big|\langle\zeta_{z}(t,\cdot), \varphi^\eps_x\rangle_\eps|\lesssim \eps^{2-n+\beta}\enorm{t}^{\varrho}\lesssim\enorm{t}^{\bar \varrho +\bar \beta}\enorm{T}^\theta
\end{equ}
where $\varphi$ is {\it any} function in $\CB^r_0$ so that the passage from the first to the second line is justified by the fact that $\varphi^\eps$ is compactly supported in a ball of radius $\eps$. The same argument works for the spatial increments. Indeed, obviously one has $t-s<c\eps^2\leq |x-\bar x|^2$ for $x\neq\bar x\in\T_\eps$, so that it suffices to bound each of the terms in (the discrete version of)~\eqref{e:SpaceIncrements} and obtain the required estimate. 

For the time increments instead, if one looks at (the discrete version of)~\eqref{e:TimeIntegral2}, the first summand can be treated as before, while for the second notice that, for $\bar t> t\in\Lambda_{\eps^2,T}$, $t-s< c\eps^2\leq \bar t - t$ we have
\begin{multline*}
\eps^2\sum_{\substack{s\in\Lambda_{\eps^2,t}\\t-s< c\eps^2}}\langle \zeta_z^\eps(s,\cdot),(P^{(\eps,n)}_{\bar t-s}-P^{(\eps,n)}_{t-s})(x-\cdot)\rangle\dd s=\eps^2 \langle \zeta_z^\eps(t,\cdot), P^{(\eps,n)}_{\bar t-t}-\eps^{-n} \delta_{x,\cdot}\rangle\\
\lesssim \eps^2t^{\frac{\rho}{2}}\big((\bar t -t)^{\frac{\beta-n}{2}} + (\bar t-t)^{\frac{\alpha-n}{2}} \eps^{\beta-\alpha}\big) + \eps^{\beta+2-n}t^{\frac{\rho}{2}}\lesssim \eps^\theta t^{\frac{\rho}{2}} (\bar t-t)^{\frac{\bar \beta}{2}}
\end{multline*}
which in turn concludes the proof of~\eqref{e:DConvBound}. 
\end{proof}

 And again, as a straightforward corollary of the previous proposition combined with Lemma~\ref{l:DSchwartz}, we have the following discrete Schauder estimate. 

\begin{corollary}\label{c:DSchauder}
Let $T>0$, $\alpha, \beta,\,\varrho$ and $n$ as in the statement of Proposition~\ref{p:HeatBounds}. If $\zeta^\eps$ is a function on $\Lambda_{\eps^2, T}\times \T^\eps$ such that $\|\zeta^\eps\|_{\CC^{\beta}_{\rho,T}}^{(\eps)}\leq M$ for some finite $M>0$ independent of $\eps$, then for the same values of $\theta_1,\,\theta_2,\,\theta,\,\bar \beta$ and $\bar \rho$ as in Corollary~\ref{c:Schauder}, we have
\begin{equ}[e:DConvBound]
\sup_{z} \enorm{t}^{-(\bar \rho+\bar \beta)} \bigl|\bigl(P^{(\eps,n)} \ast_\eps \zeta^\eps\bigr)(z)\bigr| +  \sup_{z }\pnorm{P^{(\eps,n)} \ast_\eps \zeta^\eps}_{ \bar\rho, \bar \beta; T, z}^{(\eps)} \lesssim C \enorm{T}^{\theta}
\end{equ}
uniformly in $\eps$, where the suprema run over $z=(t,x)\in \Lambda_{\eps^2,T}\times\T_\eps$. 
\end{corollary}

\begin{proof}
The proof is identical to the one of Corollary~\ref{c:Schauder}, but instead of Lemma~\ref{l:Schwartz} and Proposition~\ref{p:HeatBounds} one has to use Lemma~\ref{l:DSchwartz} and Proposition~\ref{p:DHeatBounds}. 
\end{proof}

In the following we show how to apply Proposition~\ref{p:DHeatBounds} in the context of the Reconstruction Theorem~\ref{t:DReconstruct}. 

\begin{lemma}\label{l:DReconstructBound}
Let $\ST=(\CA,\CT, \CG)$ be a regularity structure according to Definition~\ref{def:RS}, $\alpha \eqdef \min \CA$, $Z^\eps = (\Pi^\eps, \Gamma^\eps, \Sigma^\eps)$ a discrete periodic model on $\ST$ and $\CR^\eps$ the corresponding reconstruction operator. 
Let $\gamma,\,\eta,\,\beta$ and $n$ be as in the statement of Lemma~\ref{l:ReconstructBound} and $T\in(0,1]$. For $z = (t,x) \in \Lambda_{\eps^2,T}\times\T_\eps$ and $H^\eps \in \CD^{\eta, \gamma}_{\eps,T}(Z^\eps)$, define the distribution
\begin{equ}[e:DReconstructCut]
\bigl(\CJ^\eps H^\eps\bigr)_{\beta, z}(s, \cdot) \eqdef \bigl(\CR^\eps_s H^\eps_s - \Pi^{\eps,s}_x \CQ_{< \beta} H^\eps_t(x)\bigr)(\cdot)
\end{equ}
where $\CQ_{<\beta}$ is the projection onto $\CT_{<\beta}$ (see Remark~\ref{r:QOperators}). Then, for $\theta_1,\,\theta_2,\,\theta$ and $\bar\beta$ as in the same Lemma mentioned above, the following bound holds
\begin{equ}[b:DReconstructionCut]
\sup_z\enorm{t}^{-(\eta-\gamma+\theta_2+\bar \beta)} \bigl|\bigl(P^{(\eps,n)} \ast_\eps \bigl(\CJ^\eps H^\eps\bigr)_{\beta, z}\bigr)(z)\bigr| + \sup_z\pnorm{P^{(\eps,n)} \ast_\eps \bigl(\CJ^\eps H^\eps\bigr)_{\beta, z}}_{ \eta -  \gamma+\theta_2, \bar \beta; T, z}^{(\eps)} \lesssim C \enorm{T}^{\theta}\;,
\end{equ}
uniformly in $\eps$, where the suprema run over $z=(t,x)\in \Lambda_{\eps^2,T}\times\T_\eps$. 

If furthermore $\bar Z^\eps = (\bar \Pi^\eps, \bar \Gamma^\eps, \Sigma^\eps)$ is another discrete model, $\bar \CR_t^\eps$ the associated family of discrete reconstruction operators and $\bar H^\eps \in \CD_{\eps,T}^{\gamma, \eta}(\bar Z)$, then~\eqref{e:DReconstructCut} still holds upon replacing $\CJ^\eps H^\eps$ by $\CJ H-\bar \CJ^\eps \bar H^\eps$ and with the constant $C$ proportional to 
\[
\bigl(\Vert H^\eps; \bar H^\eps \Vert^{(\eps)}_{\gamma, \eta; T} \Vert \Pi^\eps \Vert^{(\eps)}_{\gamma; T} + \Vert \bar H^\eps \Vert^{(\eps)}_{\gamma, \eta; T} \Vert \Pi^\eps - \bar \Pi^\eps \Vert^{(\eps)}_{\gamma; T} \bigr)\,.
\]
\end{lemma}

\begin{proof}
Also in this case the proof is completely analogous to the one of Lemma~\ref{l:ReconstructBound}, but using the discrete counterpart of the lemmas and propositions mentioned in the proof of the latter. 
\end{proof}

At last, we recall the analogous of Lemma~\ref{l:InitialCondition} in the discrete setting whose proof can be easily deduced by Lemma~\ref{l:DSchwartz} and some of the tools exploited in Proposition~\ref{p:DHeatBounds}.

\begin{lemma}\label{l:DInitialCondition}
Let $\eta\in\R\setminus\Z$ and $u^\eps_0 \in \CC^{\eta,\eps}(\T^\eps)$ be a periodic family parametrized by $\eps$. Then $P^\eps u^\eps_0 \in \CC^{\gamma,\s,\eps}_{\eta, T}$ for any $\gamma \in (0,1)$, and one has the following bounds uniformly in $\eps$:
\begin{equ}
 \Vert P u_0 \Vert_{\CC^{\gamma,\s}_{\eta, T}}^{(\eps)} \lesssim \Vert u_0 \Vert_{\CC^{\eta}}^{(\eps)}\;.
\end{equ} 
\end{lemma}

\subsection{Uniform bounds on the discrete solution map}\label{sec:unif_bounds}

We now want to bound, uniformly in $\eps$, the map $\CM^\eps$ defined   in~\eqref{DFPMap} and that we now recall. 
Let $u_0^\eps$ be a periodic function on $\Lambda_\eps$ and $\X^\eps\in\CX^\eps$ a discrete controlling process according to Definition~\ref{def:DControlProc}. We set $\CM^\eps(\cdot)\eqdef\CM^\eps(u_0^\eps,\,\X^\eps,\,\cdot)$. 
For $V^\eps=(v^\eps,v'^{\,\eps},R^\eps)\in\CH^{\eta,\gamma}_{\eps,\X^\eps,T}$, $\CM^\eps(V^\eps)=(\tilde v^\eps,\tilde v'^{\,\eps},\tilde R^\eps)$ is given by
\begin{subequations}\label{DFixedPointMap}
\begin{align}
\tilde{v}^\eps &= P^\eps u_0^\eps + 4 X^{\<tree1222>,\,\eps} + D_{x,\eps}P^\eps \ast_\eps \bigl(2 \CR^\eps \bigl(\BV\<bigtree1>\bigr)^\eps + F_{v^\eps}^{(\eps)}\bigr) + Q^{(\eps)}\;,\label{e:DFPMapv} \qquad \tilde{v}'^{\,\eps} = 4 X^{\<tree122>,\,\eps} + 2 v^\eps\;,\\
\tilde{R}(z, \bar z) &= \delta_{z, \bar z} \big(P^\eps u^\eps_0\big) +4 R^{\<tree1222>,\,\eps}_z(z,\bar z) + \delta_{z, \bar z} D_{x,\eps}P^\eps \ast \big(2 \bigl(\CJ^\eps (\BV\<bigtree1>)^\eps\bigr)_{\beta, z} + F_{v^\eps}^{(\eps)}  \big) + \delta_{z, \bar z}  Q^{(\eps)}\notag\\
&\qquad\qquad\qquad\qquad\qquad + \int \left(\tilde v'^{\,\eps}(z_1) \delta_{z_2,\bar z_2} \hat X^{\<tree11>,\eps}+ \delta_{z_1,z} \tilde v'^{\,\eps} \delta_{z_2,\bar z_2} \hat X^{\<tree11>,\eps}\right)\mu_\eps(\dd y_1,\dd y_2)\;,\label{e:DFPMapr}
\end{align}
\end{subequations}
where $\hat X^{\<tree11>,\,\eps}\eqdef D_{x,\eps}\hat K^\eps\ast_\eps X^{\<tree1>}$ and the terms $F_{v^\eps}^{(\eps)}$ and $Q^{(\eps)}$ are given in~\eqref{e:DF}. The equation for the remainder $\tilde R^\eps$ was obtained, as for~\eqref{FixedPointMap}, by unwrapping the relation~\eqref{e:DVExpansion}, exploiting the definitions of $\tilde v'^{\,\eps}$, $X^{\<tree11>,\eps}$ and~\eqref{e:DReconstructCut} but in this case an extra term appears which is due to the fact that the product we are dealing with is {\it twisted}. Anyway this term does not create any trouble since it is sufficiently regular and allows us to conclude that bounds similar to the ones of Proposition~\ref{prop:FixedPoint} hold also in the discrete setting. 

\begin{proposition}\label{prop:DFixedPoint}
Let $\alpha_\star ,\,\gamma $ and $\eta$ be as in the statement of Proposition~\ref{prop:FixedPoint}. Let $u_0^\eps \in \CC^{\eta,\eps}$ be a family of periodic functions parametrized by $\eps$ and $\X^\eps \in \CX^\eps$. Then there exists $\theta>0$ such that the map $\CM^\eps$, defined in~\eqref{DFixedPointMap}, satisfies, for all $V^\eps$, $\bar V^\eps \in \CH^{\eta, \gamma}_{\eps,\X^\eps,T}$ the following bounds 
\minilab{DFixedPointBounds}
\begin{gather}
\Vert \CM^\eps (V^\eps) \Vert _{\eta, \gamma; T}^{(\eps)} \lesssim \Vert u_0^\eps \Vert _{\CC^{\eta}}^{(\eps)} + \enorm{T}^\theta \bigl(1+\Vert \X^\eps \Vert _{\CX}^{(\eps)}\bigr)^2\big(1 + \Vert V^\eps \Vert _{\eta, \gamma; T}^{(\eps)} \bigr)^2\;,\label{e:DFixedPointBoundsFirst}\\
\Vert \CM^\eps (V^\eps) - \CM^\eps (\bar V^\eps) \Vert _{\eta, \gamma; T}^{(\eps)} \lesssim \enorm{T}^\theta \Vert V^\eps - \bar V^\eps \Vert _{\eta, \gamma; T}^{(\eps)} (1+\|\X\|^{(\eps)}_{\CX})^2\label{e:DFixedPointBoundsSecond}
\end{gather}
uniformly in $\eps$, where the second bound holds provided that $\Vert V^\eps\Vert^{(\eps)} _{\eta, \gamma; T},\,\Vert \bar V^\eps\Vert^{(\eps)} _{\eta, \gamma; T}\leq M<\infty$, $M$ independent of $\eps$, and the hidden constant depends only on $M$. 
\end{proposition}

The following lemma represents the discrete counterpart of Lemma~\ref{l:RemV} and will play its same role in the proof of the previous proposition. Since the proof is identical to the one already given we omit it. 

\begin{lemma}\label{l:DRemV}
Let $\alpha_{\star}, \,\gamma$ and $\eta$ be as in the statement of Lemma~\ref{l:RemV}. Let $\X^\eps\in\CX^\eps$ and $Z^\eps=(\Pi^\eps,\Gamma^\eps,\Sigma^\eps)$ be the model given in~\eqref{e:DTModel}. Let $V^\eps=(v^\eps,v'^{\,\eps},R^\eps)\in\CH^{\gamma,\eta}_{\eps,\X^\eps, T}$ and $(\BV\<bigtree1>)^\eps$ be defined as in~\eqref{e:DTVDot}. 
For $z=(t,x)\in\Lambda_{\eps^2,T}\times\T_\eps$, let
\[
\zeta^\eps_z(s,\cdot)\eqdef \Pi_x^{\eps,s}\CQ_{<\beta_{\<tree21>}}\big(\BV\<bigtree1>\big)^\eps_t(x)(s,\cdot)
\]
where $\CQ_{<\beta_{\<tree21>}}$ is the projection onto $\CT_{<\beta_{\<tree21>}}$. Then the following bound holds
\begin{equ}[b:DPiVBeta]
\sup_z\enorm{t}^{-\eta} \bigl|\bigl(P^{(\eps,n)} \ast_\eps \zeta^\eps_z\bigr)(z)\bigr| + \sup_z\pnorm{P^{(\eps, n)} \ast_\eps \zeta^\eps_z}_{ \eta-\alpha_\star, \alpha_\star; T, z}^{(\eps)} \lesssim C \enorm{T}^{\theta} \;,
\end{equ}
uniformly in $\eps$, where, the suprema run over $z=(t,x)\in \Lambda_{\eps^2,T}\times\T_\eps$, and $\theta$ can be chosen as in Lemma~\ref{l:RemV}. 
Moreover, if $\bar \X^\eps\in\CX^\eps$, $\bar Z^\eps = (\bar \Pi^\eps, \bar \Gamma^\eps,\Sigma^\eps)$, $\bar V^\eps\in\CH^{\eta,\gamma}_{\eps,\bar \X^\eps, T}$, $\bigl(\bar{\BV}\<bigtree1>\bigr)^\eps$ are another controlling process, model, controlled process and modelled distribution respectively,  then, for $\bar \zeta^\eps$ defined as above,~\eqref{b:DPiVBeta} still holds upon replacing $\zeta^\eps$ with $\zeta^\eps-\bar \zeta^\eps$.
\end{lemma}

\begin{proof}[Proof of Proposition~\ref{prop:DFixedPoint}]
As in the case of Proposition~\ref{prop:FixedPoint}, the proof consists in bounding each of the norms in Definition~\ref{e:DNorms}, which in turn boils down to verify that for each summand appearing in the equations for $\tilde V^\eps$ the assumptions of Proposition~\ref{p:DHeatBounds} are satisfied. 
The terms in~\eqref{DFixedPointMap} which have an expression analogous to the corresponding ones in~\eqref{FixedPointMap} can be treated in the same way. 
\end{proof}

We conclude this section with the proof of Theorem~\ref{t:DFixedPoint}. 

\begin{proof}[Proof of Theorem~\ref{t:DFixedPoint}]
Notice that in the discrete context, the problem is not existence and uniqueness of solution for {\it fixed} $\eps$ since one can explicitly construct the solution at all times. The point is to obtain a bound, {\it uniform} in $\eps$ of its $\|\cdot\|_{\eta,\gamma;T}^{(\eps)}$ norm. 

By assumption, $\X^\eps\in\CX^\eps$ and $u_0^\eps\in\CC^{\eta,\,\eps}$, hence the sequences $\{\X^\eps\}_\eps$ and $\{u_0^\eps\}_\eps$ are such that $M_{cp}\eqdef \sup_{\eps}\|\X^\eps\|_{\CX^\eps}<\infty$ and $M_0\eqdef\sup_{\eps}\|u_0^\eps\|_{\CC^{\eta,\eps}}^{(\eps)}<\infty$. Assume $\Vert V^\eps \Vert _{\eta, \gamma; T}^{(\eps)}\leq M$ for all $\eps>0$
Then,~\eqref{e:DFixedPointBoundsFirst} tells us that there exists $C>0$ such that 
\begin{align*}
\Vert \CM^\eps (V^\eps) \Vert _{\eta, \gamma; T}^{(\eps)} &\leq C\Big( M_0 +\enorm{T}^\theta \bigl(1+M_{cp}\bigr)^2\big(1 + M\bigr)^2\Big)\\
&\leq \frac{1}{2} M\Big(1+\enorm{T}^\theta \bigl(1+M\bigr)^2\Big)=\frac{1}{2} M\Big(1+T^{\frac{\theta}{2}} \bigl(1+M\bigr)^2\Big)
\end{align*}
where the passage from the first to the second line holds upon taking $M\eqdef2C(M_0+(1+M_{cp})^2)$ while in the last inequality we are imposing $\eps<(1+M)^{-2/\theta}$. Hence, for $T=(1+M)^{-\theta}=\bar C (M_0 + (1+ M_{cp})^2)^{-\theta}$ and $\eps<(1+M)^{-2/\theta}$, $\CM^\eps$ leaves the ball of radius $M$ in $\CH^{\eta,\gamma}_{\eps,\X^\eps;T}$ invariant, uniformly in $\eps$, and for smaller $T$ and $\eps$, thanks to~\eqref{e:DFixedPointBoundsSecond} we conclude that $\CM^\eps$ is there a contraction. This in turn implies that the unique fixed point (in the whole space $\CH^{\eta,\gamma}_{\eps,\X^\eps;T}$, invoking again the same argument as~\cite[Thm. 4.8]{Hai13}) $V^\eps$ is such that its $\|\cdot\|_{\eta,\gamma;T}^{(\eps)}$ is uniformly bounded in $\eps$ by $2C(M_0+(1+M_{cp})^2)$. 

By iterating the previous argument, we get the maximal time for which the sequence of fixed points $V^\eps$ is uniformly bounded. The proof of the local Lipschitz continuity can then be obtained following the scheme suggested in the proof of Theorem~\ref{t:FixedPoint}.
\end{proof}

\section{Convergence of controlling processes}\label{sec:Stochastic}

The focus of this section is on the controlling processes we introduced in Sections~\ref{sec:ControllingProcesses} and~\ref{sec:DControllingProcesses}. We aim at showing that it is indeed possible to consistently {\it enhance} a space-time white noise $\xi$ to a controlling process and that, if the family of rescaled normal random variables $\{\xi^\eps(z)\}_{z\in\Lambda_\eps^\s}$ converges to $\xi$ in a suitable topology, then so does the discrete controlling process associated to it. 

\subsection{A convergence criterion for random distributions}\label{sec:Criterion}

Let us begin by recalling some basic facts about Wiener-chaos decomposition, on which our subsequent analysis is based. Let $(\Omega, \SF, \Prob)$ be a probability space and $\xi$ be a space-time white noise on it, by which we mean a Gaussian random field whose covariance structure is given by
\[
\E\big[\langle \xi,\varphi\rangle\langle \xi,\psi\rangle\big]=\langle \varphi, \psi\rangle_{L^2}
\]
where $\varphi,\,\psi\in H\eqdef L^2(\R\times\T)$. It is well-known (see, for example,~\cite[Ch. 1]{Nua06}) that $L^2(\Omega, \Prob)=\bigoplus_{n\in\N}\CH_n$, $\CH_n$ being defined as the closure of the linear subspace of $L^2(\Omega, \Prob)$ generated by the random variables $\{H_n(\xi(\varphi)): \varphi\in H\}$, where $H_n$ is the Hermite polynomial of degree $n$ (\cite[Thm. 1.1.1]{Nua06}). 

It is also possible to show that, for every $n\in\N$, there exists a bijective isometry, $I_n$ (which we will refer to as Wiener-It\^o isometry), between the space $H^{\otimes_s n}$ of symmetric functions in $L^2((\R\times\T)^{\otimes n})$ and $\CH_n$. 
Composing the projection that assigns to every element in $H^{\otimes n}$ its symmetrization belonging to $H^{\otimes_s n}$, we obtain a map (that we will still denote by $I_n$) such that 
\begin{equ}[b:ItoIso]
\E[I_n(\varphi)^2]\lesssim \|\varphi\|^2_{H^{\otimes n}}
\end{equ}
for every $\varphi\in H^{\otimes n}$ and $I_n(\varphi)$ is called Wiener-It\^o integral of order $n$. We will say that a random variable $Y\in L^2(\Omega,\Prob)$ belongs to the $n$-th homogeneous Wiener chaos if $Y\in\CH_n$, while if $Y\in\bigoplus_{k\leq n}\CH_k$ we will say that it belongs to the $n$-the inhomogeneous Wiener chaos. 
\newline

Let $\xi_\eps\eqdef \xi\ast\varrho_\eps$, where $\varrho$ a symmetric compactly supported smooth function integrating to 1 and $\varrho_\eps(t,x)\eqdef \eps^{-3}\varrho(\eps^{-2} t, \eps^{-1} x)$. 
At this point, it should be clear that, upon taking $\eta=\xi_\eps$, each of the stochastic objects introduced in~\eqref{Control} lies in a finite inhomogeneous Wiener chaos, whose order is given by the number of occurrences of $\xi$ in their expression (e.g. $X^{\<tree1>}$ is in the first, $X^{\<tree2>}$ in the second and so on). 

In order to show that there exist suitable constants for which $\X(\xi_\eps, C_\eps^{\<tree2>}, C_\eps^{\<tree21>})$ converges to a well-defined object, independent of $\varrho$, in a suitable topology, we will decompose each of its terms in its Wiener-chaos expansion and separately bound its chaos components. From the discussion above, one can deduce that such components will be of the form
\begin{equ}[e:XWienerIntegral]
Y_\eps(\varphi) = I_k\left(\int_{\R^{2}} \varphi(z)\, \CW_{\eps}(z)\, dz\right),
\end{equ}
for a test function $\varphi$, where $I_k$ is the Wiener integral of  $k$-th order with respect to $\xi$, the function $\CW_{\eps}$ takes values in $H^{\otimes k}$ and is such that $\mathcal{W}_\eps(z_1) \mathcal{W}_\eps(z_2) \in L^1((\R\times\T)^{\otimes k})$, for every $z_1 \neq z_2 \in \R^2$. 
Moreover, we define its limit, $Y$, in the same way, but via some $H^{\otimes k}$-valued function $\mathcal{W}$, such that $\mathcal{W}(z_1) \mathcal{W}(z_2)$, $\mathcal{W}(z_1) \mathcal{W}_\eps(z_2) \in L^1((\R\times\T)^{\otimes k})$, for every $z_1 \neq z_2 \in \R^2$. 
By~\eqref{b:ItoIso}, the covariance functions, given by
\begin{equ}[e:KDef]
 \CK(z_1, z_2) \eqdef \langle \mathcal{W}(z_1), \mathcal{W}(z_2) \rangle_{L^2}\;, \qquad \delta \CK_{\eps}(z_1, z_2) \eqdef \langle \delta \mathcal{W}_{\eps}(z_1), \delta \mathcal{W}_{\eps}(z_2) \rangle_{L^2}\;,
\end{equ}
for every $z_1 \neq z_2$, where $\delta \mathcal{W}_{\eps} \eqdef \mathcal{W}_{\eps} - \mathcal{W}$ and the scalar product is taken in $H^{\otimes k}$, will play an extremely important role in our analysis. 
In all the cases we will consider these functions can be written as $\CK(z_1, z_2) = \CK(z_1 - z_2)$, where $\CK$ is a kernel that can have a singularity at the origin and is smooth away from it. 
It is this singularity that determines the regularity and convergence properties of the distributions $Y$ and $Y_\eps$. The following definition provides a quantitative description of such functions.

\begin{definition}\label{d:SingularFunction}
We say that a smooth compactly supported function $\CK : \R^{2} \setminus \{0\} \rightarrow \R$ has a singularity of order $\zeta \in \R$, if for some $m \in \N$ and for all multiindices $k=(k_0,k_1)\in\N^2$, such that $|k|_{\mathfrak{s}}\eqdef 2|k_0|+|k_1| \leq m$, the function $D^k \CK$ is defined on $\R^2$, and there exists a constant $C > 0$ such that the bound
\begin{equ}[e:KBound]
|D^k \CK(z)| \leq C \Vert z \Vert_{\mathfrak{s}}^{\zeta - |k|_{\mathfrak{s}}}\;,
\end{equ}
holds uniformly over $z \neq 0$. Here, for a point $z = (t,x)$, we use the norm $\Vert z \Vert_{\mathfrak{s}} = |t|^{1/2} \vee |x|$.
\end{definition}

\begin{remark}
It is not difficult to see that, if the function $\CK$ has singularity of order $\zeta > 0$ such that $\zeta \notin \N$ and $m \geq \lfloor \zeta \rfloor$ in the sense of Definition \ref{d:SingularFunction}, then the function
\begin{equ}[e:TaylorOperator]
T^{k, \zeta} \CK(z) \eqdef D^{k} \CK(z) - \sum_{|\ell|_{\mathfrak{s}} < \zeta - |k|_{\s}} \frac{z^\ell}{\ell!} D^{k+\ell} \CK(0)
\end{equ}
is well-defined, for $z \neq 0$ and for all multiindices $k$ with $|k|_{\s} < \zeta$, and satisfies the bound
\begin{equ}[e:TKBound]
|T^{k, \zeta} \CK(z)| \leq C \Vert z \Vert_{\mathfrak{s}}^{\zeta - |k|_{\mathfrak{s}}}\;,
\end{equ}
uniformly in $z \neq 0$ (but not in $k$). The sum in the definition of $T^{k, \zeta} \CK$ runs over multiindices $\ell \in \N^2$. (For a reference, see the proof of~\cite[Lem. 10.14]{Hai14}.) 
\end{remark}

The following proposition provides a criterion to show the convergence of random variables $Y_\eps$ as in~\eqref{e:XWienerIntegral} to $Y$. 

\begin{proposition}\label{p:ChaosKernelsConv}
In the setting described above, let $\mathcal{K}$ and $\delta \mathcal{K}_{\eps}$ be compactly supported functions on $\R^2$ defined in \eqref{e:KDef} for distributions $Y$ and $Y_\eps$ of the form \eqref{e:XWienerIntegral}. Then we have the following results
\begin{enumerate}
\item if $\mathcal{K}$ is singular of order $\zeta \in (-3, -1]$, then the process $Y$ almost surely belongs to the space $\CC_1^{\alpha,\s}$, for any $\alpha < \frac{\zeta}{2}$. If furthermore the function $\delta \mathcal{K}_{\eps}$ has a singularity of order $\zeta - 2\theta$, for some $\theta > 0$, and satisfies the estimates \eqref{e:KBound}, with the proportionality constants of order $\eps^{2\theta}$, then, for any $p \geq 1$ and $\alpha < \frac{\zeta}{2}$, the following bound holds
\begin{equ}[e:KEpsConverge]
 \mathbb{E} \left[\Vert Y - Y_\eps \Vert^p_{\CC_1^{\alpha - \theta,\s}}\right] \lesssim \eps^{\theta p}\;.
\end{equ}
 \item\label{item:third} If $\mathcal{K}$ is singular of order $\zeta \in (-1, \infty) \setminus \N$, then the process $Y$ almost surely belongs to the space $\CC^{\beta/2} ([-1,1], \CC^{\alpha - \beta})$, for any $\alpha < \frac{\zeta}{2}$ and any $\beta\in(0,\alpha)$ if $\zeta>0$ and any $\beta\in\big(0, \frac{\zeta + 1}{2}\big)$ if $\zeta<0$. Moreover, if the function $\delta \mathcal{K}_{\eps}$ has singularity of order $\zeta - 2\theta$, for some $\theta \in \big(0, \frac{\zeta+1}{2}\big)$, and satisfies the estimates \eqref{e:KBound}, with the proportionality constants of order $\eps^{2\theta}$, then, for any $p \geq 1$ and $\alpha < \frac{\zeta}{2}$, the following bound holds
\begin{equ}[e:KEpsConvergeTime]
 \mathbb{E} \left[\Vert Y - Y_\eps \Vert^p_{\CC^{\beta/2}([-1,1], \CC^{\alpha - \beta - \theta})}\right] \lesssim \eps^{\theta p}\;,
\end{equ}
where $\beta > 0$ is such that $\beta + \theta\in(0,\alpha)$ if $\zeta>0$ and $\beta+\theta\in\big(0, \frac{\zeta + 1}{2}\big)$ if $\zeta<0$.
\end{enumerate}
\end{proposition}

The proof of the previous proposition is based on the following technical lemma. For any two time points $s, t \in \R$, we define the operator 
\begin{equ}[e:Deltaz]
\delta_{s, t}^2 \CK(x) \eqdef \sum_{\epsilon_1, \epsilon_2 \in \{0, 1\}} (-1)^{\epsilon_1 - \epsilon_2} \CK\big((\epsilon_1 - \epsilon_2)(t - s), x\big)\;,
\end{equ}

\begin{lemma}
\label{l:KDeltaz}
Let the function $\CK : \R^{2} \setminus \{0\} \rightarrow \R$ have a singularity of order $\zeta \leq 0$. Then, for any $\beta \in [0,2]$ and any points $s, t \in \R$ and $x \neq 0$, the following bound holds:
\begin{equs}
 |\delta^2_{s, t} \CK(x)| \leq C | s- t |^{\beta} |x|^{\zeta - 2 \beta}\;.
\end{equs}
\end{lemma}

\begin{proof}
First, we consider the regime $|t - s| \geq |x|^2$. In this case we use the brutal bound
\begin{equs}
|\delta^2_{s, t} \CK(x)| \lesssim \sum_{\epsilon \in \{-1, 0, 1\}} |\CK(\epsilon(t - s), x)| \lesssim |x|^{\zeta} \lesssim | t - s |^\beta |x|^{\zeta - 2\beta}\;,
\end{equs}
which holds for any $\beta \geq 0$. In the case $|t-s| < |x|^2$, we use the identity
\begin{equ}
\delta^2_{s, t} \CK(x) = \int_{s}^{t} \int_{s}^{t} \partial_{s_1} \partial_{s_2} \CK(s_1 - s_2, x)\, d s_1\, d s_2\;,
\end{equ}
from which the next estimate for $t > s$ follows
\begin{equs}
|\delta^2_{s, t} \CK(x)| \lesssim \int_{s}^{t} \int_{s}^{t} |x|^{\zeta - 4}\, d s_1\, d s_2 \lesssim |t-s|^2 |x|^{\zeta - 4} \lesssim | t-s |^\beta |x|^{\zeta - 2\beta}\;,
\end{equs}
for any $\beta \leq 2$. This is precisely the required bound.
\end{proof}

\begin{proof}[Proof of Proposition \ref{p:ChaosKernelsConv}]
In the proof of this proposition we will use some elements of wavelet analysis that are briefly recalled in the Appendix. 

Let $\varphi$ be a father wavelet on $\R^2$ of H\"{o}lder regularity $r > 3$ and $\psi \in \Psi$ the corresponding finite set of mother wavelets. Furthermore, we denote by $\varphi_{z}^{n, \mathfrak{s}}$ and $\psi_{z}^{n, \mathfrak{s}}$ the respective rescaled and recentered functions with parabolic scaling $\s=(2,1)$. 
Then we can write
\begin{equ}
\mathbb{E} |\langle Y, \psi_{z}^{n, \mathfrak{s}} \rangle|^2 = \int \psi_{z}^{n, \mathfrak{s}} (z_1) \psi_{z}^{n, \mathfrak{s}}(z_2)\, T^{0, \zeta} \mathcal{K}(z_1 - z_2)\, d z_1\, d z_2\;,
\end{equ}
where we have used the operator \eqref{e:TaylorOperator} and the fact that the functions $\psi$ annihilate polynomials. Exploiting \eqref{e:TKBound}, the last integral can be estimated by
\begin{equation}\label{e:IntegralBound}
\int |\psi_{z}^{n, \mathfrak{s}} (z_1) \psi_{z}^{n, \mathfrak{s}}(z_2)| \Vert z_1 - z_2 \Vert_{\mathfrak{s}}^{\zeta}\, d z_1\, d z_2 \lesssim 2^{3 n} \int_{\substack{\Vert z_1 \Vert_{\mathfrak{s}} \lesssim 2^{-n} \\ \Vert z_2 \Vert_{\mathfrak{s}} \lesssim 2^{-n}}} \Vert z_1 - z_2 \Vert_{\mathfrak{s}}^{\zeta}\, d z_1\, d z_2
\lesssim \int_{\Vert z \Vert_{\mathfrak{s}} \lesssim 2^{-n + 1}} \Vert z \Vert_{\mathfrak{s}}^{\zeta}\, d z \lesssim 2^{-(\zeta + 3) n}\;,
\end{equation}
where, in order to have the integral finite, we require $\zeta > -3$. Keeping in mind that $\varphi$ do not annihilate polynomials, we can similarly obtain the bound $\mathbb{E} |\langle Y, \varphi_{z}^{n, \mathfrak{s}} \rangle|^2 \lesssim 2^{-(\zeta \wedge 0 + 3) n}$. Now, our aim is to use the characterisation of the H\"{o}lder-Besov norm via the wavelet expansion which for $\zeta < 0$ is proved in \cite[Prop.3.20]{Hai14} and for non-integer $\zeta > 0$ is provided in \cite[Thm.~6.4.5]{Mey92}(it is easy to see that the result holds also for the parabolic scaling). Thus, denoting $\Psi_\star \eqdef \Psi \cup \{\varphi\}$ and $\Lambda_n^\s \eqdef \Lambda_{2^{-2n}} \times \Lambda_{2^{-n}}$, using these estimates and the equivalence of moments of Wiener integrals \cite{Nel73} one gets, for a compact set $\mathfrak{K} \subset \R^2$,
\begin{align}
 \mathbb{E} \Vert Y \Vert^{2p}_{\CC^{\alpha,\s}_{\mathfrak{K}}} &\leq \sum_{\psi \in \Psi_\star} \sum_{n \geq 0} \sum_{z \in \Lambda^{\mathfrak{s}}_n \cap \bar{\mathfrak{K}}} 2^{(2 \alpha + 3) np} \mathbb{E} | \langle Y, \psi_z^{n, \mathfrak{s}} \rangle |^{2p}\label{e:XEstim}\\
&\lesssim \sum_{\psi \in \Psi_\star} \sum_{n \geq 0} \sum_{z \in \Lambda^{\mathfrak{s}}_n \cap \bar{\mathfrak{K}}} 2^{(2 \alpha + 3) np} \left(\mathbb{E} | \langle Y, \psi_z^{n, \mathfrak{s}} \rangle |^{2} \right)^p\lesssim \sum_{n \geq 0} \sum_{z \in \Lambda^{\mathfrak{s}}_n \cap \bar{\mathfrak{K}}} 2^{(2 \alpha - \zeta) n p} \lesssim \sum_{n \geq 0} 2^{- np (\zeta - 2 \alpha - \frac{3}{p})}\;,\notag
\end{align}
which is finite if $\alpha < \frac{\zeta}{2}$ and $p$ is large enough, where we indicated with $\CC^{\alpha.\s}_{\mathfrak{K}}$ the space of $\CC^{\alpha,\s}$ functions on $\mathfrak{K}$, $\bar{\mathfrak{K}}$ is the $1$-fattening of $\mathfrak{K}$ and the proportionality constant depends on $p$. Here, we have used the fact that the set $\Psi_\star$ is finite. In case $\zeta < 0$, this is exactly the first claim of Proposition \ref{p:ChaosKernelsConv}.

For $\zeta > 0$, the estimate \eqref{e:XEstim} immediately implies the required result of the second claim of the above mentioned proposition (in fact, we have proved a stronger result, that $X$ belongs to the parabolic H\"{o}lder space). 
If instead $\zeta \in(-1,0)$, then a special argument is required. 
To be able to evaluate $Y$ at a fixed time point, we need the bound
\begin{equ}[e:XPointEval]
 \mathbb{E} |\langle Y, \varphi_{t,x}^{\delta, \lambda} \rangle|^2 \leq C \lambda^{\zeta}
\end{equ}
to hold uniformly over $\delta, \lambda \in (0,1]$, for all functions $\varphi_{t,x}^{\delta, \lambda}(s,y) \eqdef \eta_{t}^{\delta}(s) \psi_{x}^{\lambda}(y)$, where $\eta, \psi \in \bar \CC_0^r(\R)$ with $r > 3$. The expectation at the left hand side of the previous is equal to
\begin{equs}[e:IntPointEval]
\int \varphi_{t,x}^{\delta, \lambda} (z_1) \varphi_{t,x}^{\delta, \lambda}(z_2) \mathcal{K}(z_1, z_2)\, d z_1\, d z_2 \lesssim \int |\varphi_{t,x}^{\delta, \lambda} (z_1) \varphi_{t,x}^{\delta, \lambda}(z_2)| \Vert z_1 - z_2 \Vert_{\mathfrak{s}}^{\zeta}\, d z_1\, d z_2\;.
\end{equs}
Then, we can simply bound $\Vert z_1 - z_2 \Vert_{\mathfrak{s}}^{\zeta} \leq | x_1 - x_2 |^{\zeta}$ (recall that $\zeta < 0$), where as before we treat space and time variables separately, i.e. $z_i = (t_i, x_i)$. Therefore, we can then bound the integral \eqref{e:IntPointEval} by
\begin{equs}
 \int |\psi_{x}^{\lambda}(x_1) \psi_{x}^{\lambda}(x_2)| | x_1 - x_2 |^{\zeta}\, d x_1\, d x_2 \lesssim \lambda^{\zeta}\;,
\end{equs}
where we have estimated the integral in the same way as in \eqref{e:IntegralBound}, under the assumption that $\zeta>-1$. Letting now $\delta$ in \eqref{e:XPointEval} go to $0$ and using the dominated convergence theorem we obtain
\begin{equ}
 \mathbb{E} | \langle Y(t), \psi_{x}^{\lambda} \rangle|^2 \leq C \lambda^{\zeta}\;,
\end{equ}
for any time point $t \in \R$. 
Using this bound one can proceed in the same way as in \eqref{e:XEstim}, but now exploiting the wavelet expansion only in the spatial variable, to obtain for $\alpha < \frac{\zeta}{2}$,
\begin{equs}
 \mathbb{E} \Vert Y(t) \Vert^{2p}_{\CC^{\alpha}} \leq C\;.
\end{equs} 

At last, we  investigate the regularity of $Y$ in time. Let us denote for brevity $\delta_{s,t} \mathcal{W}(x) \eqdef \mathcal{W}(t, x) - \mathcal{W}(s, x)$. Then we can write
\begin{equ}
 \langle \delta_{s,t} \mathcal{W}(x_1), \delta_{s,t} \mathcal{W}(x_2) \rangle = \delta_{s,t}^2 \mathcal{K}(x_1 - x_2)\;,
\end{equ}
where we used the operator $\delta_{s,t}^2$ defined in \eqref{e:Deltaz}. Furthermore, it follows from Lemma~\ref{l:KDeltaz} that
\begin{equ}
 |\delta_{s,t}^2 \mathcal{K}(x_1 - x_2)| \lesssim C |t- s|^\beta |x_1 - x_2|^{\zeta - 2\beta}\;,
\end{equ}
for any $\beta \in [0,2]$. Taking a function $\psi_x^n$ from the wavelet basis on $\R$ and using the previous bound we derive the estimates
\begin{multline*}
\mathbb{E} | \langle Y(t) - Y(s), \psi_x^n \rangle |^2 = \int \psi_x^n (x_1) \psi_x^n(x_2)\, \delta_{s,t}^2 \mathcal{K}(x_1 - x_2)\, d x_1\, d x_2\\
\lesssim |t - s|^\beta \int |\psi_x^n (x_1) \psi_x^n(x_2)| | x_1 - x_2 |^{\zeta - 2\beta}\, d x_1\, d x_2 \lesssim |t - s|^\beta 2^{- (\zeta - 2\beta) n - n}\;,
\end{multline*}
where the integral is estimated as in \eqref{e:XEstim} with the condition $\zeta - 2\beta > - 1$. From this estimate we obtain in the same way as above
\begin{equs}
 \mathbb{E} \Vert X(t) - X(s) \Vert^{2p}_{\CC^{\alpha - \beta}} \lesssim |t-s|^{\beta p}\;,
\end{equs}
for any $\alpha < \frac{\zeta}{2}$. Now one can apply the Kolmogorov continuity criterion for Banach-valued random variables \cite{Kal02}, to prove the existence of a modification of $Y$ which belongs to $\CC^{\beta/2} ([-1,1], \CC^{\alpha - \beta})$, for any $\beta < \frac{\zeta + 1}{2}$.

The required estimates on the approximating processes $Y_\eps$ can be obtained by following the same steps, recalling that the proportionality constants depend on $\eps$.
\end{proof}

In the following corollary we deal with the special case in which a random variable $Y$ can be written as $K\ast\bar Y$, where $K$ is a singular kernel according to Definition~\ref{d:SingularFunction} and $\bar Y$ is a random variable of the form~\eqref{e:XWienerIntegral}.

\begin{corollary}\label{c:KConvol}
In the settings of Proposition~\ref{p:ChaosKernelsConv} with $\zeta \in (-3, 0)$, let a compactly supported function $K$ be singular of order $\chi \in (-\frac{7 + \zeta}{2}, 0)$. Then the distribution $K \ast Y$ belongs almost surely to the space $\CC^{\beta/2} ([-1,1], \CC^{\alpha - \beta}(\R))$ for any $\alpha < \frac{\zeta}{2} + \chi + 3$ and $\beta$ as in the second part of the statement of the above-mentioned proposition with $\zeta + 2\chi + 6$ in place of $\zeta$.
\end{corollary}

\begin{proof}
The random distribution $K \ast Y$ is of the form \eqref{e:XWienerIntegral} with the respective function as in \eqref{e:KDef} given by
\begin{equ}
\hat{\mathcal{K}}(z_1, z_2) = \hat{\mathcal{K}}(z_1 - z_2) = \int \int K(z_1 - \bar z_1) \mathcal{K}(\bar z_1 - \bar z_2) K(z_2 - \bar z_2)\,d \bar z_1 d\bar z_2\;,
\end{equ}
where $\mathcal{K}$ is as in Proposition~\ref{p:ChaosKernelsConv}. It follows immediately from \cite[Lem.~10.14]{Hai14} that $\hat{\mathcal{K}}$ has singularity of order $\zeta + 2\chi + 6 > -1$. The claim now follows from point \ref{item:third} of Proposition~\ref{p:ChaosKernelsConv}.
\end{proof}

We conclude this section with the following lemma which will come at hand when dealing with the remainder terms introduced in~\eqref{e:Remainders}. 

\begin{lemma}\label{l:Control}
Let us be given a function $\R \ni x \mapsto R_x$ such that $R_x$ is of the form \eqref{e:XWienerIntegral} with the respective function $\mathcal{K}_x$, defined in \eqref{e:KDef}, and, for some $\zeta \in (-1, 0)$, satisfies
\begin{equ}[eControlBound]
|\mathcal{K}_x(z_1, z_2)| \lesssim \sum_{\delta \geq 0} |x_1 - x_2|^{\zeta - \delta} \left( |x_1 - x_2|^{\delta} + |x_1 - x|^{\delta} + |x_2 - x|^{\delta} \right),
\end{equ}
uniformly over the variables $z_i = (t, x_i)$, and where the sum runs over finitely many values of $\delta \in [0, 1 + \zeta)$. Let furthermore $R_x$ has a strictly positive H\"{o}lder regularity in time variable. Then $(t, x) \mapsto R_x(t, \cdot)$ belongs almost surely to the space $\CL_1^\beta(\R)$ for any $\beta < {\zeta \over 2}$.
\end{lemma}

\begin{proof}
Similarly to \eqref{e:XEstim}, but using the wavelet expansion only in the spatial variable, we can prove that the bounds
\begin{equ}
\mathbb{E} \Vert R(t) \Vert_{\CL^\beta}^p \lesssim 1\;, \qquad \mathbb{E} \Vert R(t) - R(s) \Vert_{\CL^\beta}^p \lesssim |t-s|^{\gamma p}\;,
\end{equ}
hold for all $p \geq 1$, $s, t \in [-1,1]$ and $\beta < \frac{\zeta}{2}$, and for some $\gamma > 0$. The claim then follows from the Kolmogorov continuity criterion for Banach-valued random variables \cite{Kal02}.
\end{proof}

\subsection{Enhancing the space-time white noise}\label{sec:control}

In the setting of the previous section, let $\xi$ be a periodic space-time white noise and $\xi_\eps\eqdef \xi\ast\varrho_\eps$, where $\varrho$ a symmetric compactly supported smooth function integrating to 1 and $\varrho_\eps(t,x)\eqdef \eps^{-3}\varrho(\eps^{-2} t, \eps^{-1} x)$.
The core of this section is the proof of Proposition~\ref{p:Control}. In order to represent the chaos components of the stochastic processes in $\X(\xi_\eps)$ and $\X(\xi)$, we will use the expression in \eqref{e:XWienerIntegral} and we will denote their covariance by~\eqref{e:KDef}.  
Wick's lemma \cite[Prop.~1.1.3]{Nua06} will allow us to work with the product of multiple Wiener-It\^o integrals, while we will apply Proposition~\ref{p:ChaosKernelsConv} to prove convergence of the smooth controlling processes defined in Section~\ref{sec:ControllingProcesses}. 
Notice that, this latter proposition tells us that it suffices to determine the {\it order} of the singular kernel appearing in~\eqref{e:KDef}, which in turn are given by products, convolutions and derivatives of the kernel $K$, that is the singular part of the heat kernel, as we saw at the beginning of Section~\ref{sec:ControllingProcesses}. 
To deal with these operations, we will extensively use the results of~\cite[Sec.~10.3]{Hai14}, together with the fact that $K$ and $K_\eps\eqdef  K \ast\rho_\eps$ have singularities of order $-1$ in the sense of Definition~\ref{d:SingularFunction} and that, for every $\theta\in[0,1]$, $K - K_\eps$ has a singularity of order $-1 - \theta$ with the constant $C$ in \eqref{e:KBound} proportional to $\eps^\theta$.

\begin{proof}[Proof of Proposition~\ref{p:Control}]

\cite[Lem.~10.2]{Hai14} shows that $\xi$ belongs to $\CC^{-\frac{3}{2} - \kappa,\s}$ a.s. for every $\kappa > 0$ and that, for all $p\geq 1$, $\xi_\eps$ converges to $\xi$ in $L^p(\Omega,\CC^{-\frac{3}{2} - \kappa,\s})$ as $\eps \to 0$. 

The process $X_\eps^{\<tree1>}$ can be written as $X_\eps^{\<tree1>} = \partial_x Y_\eps^{\<tree1>}$, where $Y_\eps^{\<tree1>}$ belongs to the first Wiener chaos, and whose kernel is $\mathcal{W}_\eps^{\<tree1>}(z; y) = K_\eps(z - y)$. Analogously, the limiting process $X^{\<tree1>}$ is given by $X^{\<tree1>} = \partial_x Y^{\<tree1>}$, and $Y^{\<tree1>}$ is defined via the kernel $\mathcal{W}^{\<tree1>}(z; y) = K(z - y)$. 
Then it follows from \cite[Lem.~10.14]{Hai14} that the covariance functions $\mathcal{K}^{\<tree1>}$ and $\delta \mathcal{K}^{\<tree1>}_\eps$, which are defined in \eqref{e:KDef} for $\mathcal{W}_\eps^{\<tree1>}$ and $\mathcal{W}^{\<tree1>}$, satisfy the assumptions of Proposition \ref{p:ChaosKernelsConv} with the parameters $\zeta = 1$ and $\theta < \frac{1}{2}$. This implies that $Y^{\<tree1>} \in \CC^{\beta/2} ([-1,1], \CC^{\alpha - \beta}(\R))$, for any $\alpha \in (0, \frac{1}{2})$ and $\beta \in (0, \alpha)$. The regularity of $X^{\<tree1>}$ is then $\CC^{\beta/2} ([-1,1], \CC^{\alpha - 1 - \beta}(\R))$, and the estimate \eqref{e:KEpsConvergeTime} holds for $\beta < \frac{1}{2} - \theta$.

It follows from this result and Corollary~\ref{c:KConvol} that the distribution $X^{\<tree11>} = K' \star X^{\<tree1>}$ belongs to $\CC^{\beta/2} ([-1,1], \CC^{\alpha - \beta}(\R))$, for any $\alpha \in (0, \frac{1}{2})$ and $\beta \in (0, \alpha)$. The respective bound \eqref{e:KEpsConvergeTime} follows trivially.

To lighten the notation, in what follows we will represent our kernels via diagrams, similarly to what was done in the proof of \cite[Thm.~10.22]{Hai14}. 
By 
$\begin{tikzpicture}[baseline={([yshift=-.5ex]current bounding box.center)}] \node [plainNode] (C) {}; \node [xshift=0.15cm] {$z$}; \node [plainNode] (CL) [left=0.5cm] {} edge [arrow] (C); \node [xshift=-0.15cm] at (CL) {$y$}; \end{tikzpicture}$ and $\begin{tikzpicture}[baseline={([yshift=-.5ex]current bounding box.center)}] \node [plainNode] (C) {}; \node [xshift=0.15cm] {$z$}; \node [plainNode] (CL) [left=0.5cm] {} edge [dashed] (C); \node [xshift=-0.15cm] at (CL) {$y$}; \end{tikzpicture}$ we will denote the kernels $K(z - y)$ and $K_\eps(z - y)$ respectively, while the difference $(K_\eps - K)(z - y)$ by $\begin{tikzpicture}[baseline={([yshift=-.5ex]current bounding box.center)}] \node [plainNode] (C) {}; \node [xshift=0.15cm] {$z$}; \node [plainNode] (CL) [left=0.5cm] {} edge [wavy] (C); \node [xshift=-0.15cm] at (CL) {$y$}; \end{tikzpicture}$. Furthermore, we will use a double-headed arrow for the partial derivative in space of a function, e.g. $\begin{tikzpicture}[baseline={([yshift=-.5ex]current bounding box.center)}] \node [plainNode] (C) {}; \node [xshift=0.15cm] {$z$}; \node [plainNode] (CL) [left=0.5cm] {} edge [arrow, deriv] (C); \node [xshift=-0.15cm] at (CL) {$y$}; \end{tikzpicture}$ will be $\partial_x K(z - y)$. Each node in such diagrams corresponds to a point in $\R^2$, and if a variable is integrated out, we will draw the corresponding node in gray.

Now, we can turn to the process $X_\eps^{\<tree2>}$. Let us begin by defining its $0$-th chaos component, which will be the constant we need to remove in order to obtain a well-defined object in the limit. By shift-invariance, this is given by 
\begin{equ}\label{e:Constant1}
C_\eps^{\<tree2>} \eqdef \E [X_\eps^{\<tree1>}(0)]^2 = \int_{\R^2} \bigl(K_\eps'(w)\bigr)^2\dd w
\end{equ}
and it is not difficult to see that $C_\eps^{\<tree2>} \sim \eps^{-1}$. Then upon subtracting $C_\eps^{\<tree2>}$ in \eqref{e:Control1}, we conclude that $X_\eps^{\<tree2>}$ is defined by the kernel
\begin{equ}
 \mathcal{W}_\eps^{\<tree2>}(z) = 
\begin{tikzpicture}[baseline={([yshift=-.5ex]current bounding box.center)}]
\node [plainNode] (C) {}; \node [yshift=-0.15cm] at (C) {$z$};
\node [plainNode] (CL) [above left=0.5cm of C] {} edge [dashed, deriv] (C);
\node [plainNode] (CR) [above right=0.5cm of C] {} edge [dashed, deriv] (C);
\end{tikzpicture}\;.
\end{equ}
and, correspondingly, the limiting process $X^{\<tree2>}$ is defined by the kernel $\mathcal{W}^{\<tree2>}$ which is obtained by replacing dashed arrows by solid ones in the above diagram. The covariance function $\mathcal{K}^{\<tree2>}$ is then given by
\begin{equ}
\mathcal{K}^{\<tree2>}(z_1, z_2) = 
\begin{tikzpicture}[baseline={([yshift=-.5ex]current bounding box.center)}]
\node [plainNode] (C) {}; 
\node [xshift=-0.2cm] at (C) {$z_1$};
\node [grayNode] (CRU) [above right=0.5cm of C] {} edge [arrow, deriv] (C);
\node [grayNode] (CRD) [below right=0.5cm of C] {} edge [arrow, deriv] (C);
\node [plainNode] (D) [below right=0.5cm of CRU] {};
\draw [arrow, deriv] (CRU) -- (D);
\draw [arrow, deriv] (CRD) -- (D);
\node [xshift=0.25cm] at (D) {$z_2$};
\end{tikzpicture},
\end{equ}
which implies, using \cite[Lem.~10.14]{Hai14}, that it has a singularity of order $\zeta = -2$. Applying now Proposition \ref{p:ChaosKernelsConv} we conclude that $X^{\<tree2>}$ belongs to $\CC_1^{\alpha,\s}$ for $\alpha < -1$. Writing now the kernel $\delta \mathcal{W}_\eps^{\<tree2>} = \mathcal{W}_\eps^{\<tree2>} - \mathcal{W}^{\<tree2>}$ as 
\begin{equ}
\mathcal{W}_\eps^{\<tree2>}(z) = 
\begin{tikzpicture}[baseline={([yshift=-.5ex]current bounding box.center)}]
\node [plainNode] (C) {}; \node [yshift=-0.15cm] at (C) {$z$};
\node [plainNode] (CL) [above left=0.5cm of C] {} edge [wavy, deriv] (C);
\node [plainNode] (CR) [above right=0.5cm of C] {} edge [dashed, deriv] (C);
\end{tikzpicture}
+
\begin{tikzpicture}[baseline={([yshift=-.5ex]current bounding box.center)}]
\node [plainNode] (C) {}; \node [yshift=-0.15cm] at (C) {$z$};
\node [plainNode] (CL) [above left=0.5cm of C] {} edge [arrow, deriv] (C);
\node [plainNode] (CR) [above right=0.5cm of C] {} edge [wavy, deriv] (C);
\end{tikzpicture},
\end{equ}
and treating each term in this sum separately as before, we conclude that the estimate \eqref{e:KEpsConverge} holds for $X^{\<tree2>}$ and $X_\eps^{\<tree2>}$ with $\theta < \frac{1}{2}$ and $\alpha < -1 - \theta$.

Combining the result above concerning $X^{\<tree2>}$ with Corollary~\ref{c:KConvol}, we can show that the process $X^{\<tree12>} = K' \ast X^{\<tree2>}$ belongs almost surely to the space $\CC^{\beta/2} ([-1, 1], \CC^{\alpha - \beta}(\R))$ for $\alpha < 0$ and $\beta \in (0, \frac{1}{2})$. The required bound \eqref{e:KEpsConvergeTime} for $X^{\<tree12>}_\eps$ and $X^{\<tree12>}$ can be obtained straightforwardly.

Now, we will define the process $X^{\<tree22>}$. It follows from Wick's lemma, that we can write $X_{\eps}^{\<tree22>} = X_{\eps,3}^{\<tree22>} + X_{\eps,1}^{\<tree22>}$, where the last two processes belong to the third and first homogeneous Wiener chaoses respectively, and are defined via the kernels
\begin{equs}
\mathcal{W}_{\eps,3}^{\<tree22>}(z) =
\begin{tikzpicture}[baseline={([yshift=-.5ex]current bounding box.center)}]
\node [plainNode] (C) {}; \node [yshift=-0.15cm] at (C) {$z$};
\node [grayNode] (CU) [above=0.5cm of C] {} edge [arrow, deriv] (C);
\node [plainNode] (D) [above right=0.5cm of C] {} edge [dashed, deriv] (C);
\node [plainNode] (CL) [above left=0.5cm of CU] {} edge [dashed, deriv] (CU);
\node [plainNode] (CR) [above right=0.5cm of CU] {} edge [dashed, deriv] (CU);
\end{tikzpicture}, \qquad
\mathcal{W}_{\eps,1}^{\<tree22>}(z) = 2
\begin{tikzpicture}[baseline={([yshift=-.5ex]current bounding box.center)}]
\node [plainNode] (C) {}; \node [yshift=-0.15cm] at (C) {$z$};
\node [grayNode] (CU) [above=0.5cm of C] {} edge [arrow, deriv] (C);
\node [grayNode] (D) [above right=0.4cm of C] {} edge [dashed, deriv] (C);
\node [plainNode] (CL) [above left=0.5cm of CU] {} edge [dashed, deriv] (CU);
\draw [dashed, deriv] (D) -- (CU);
\end{tikzpicture}
\;-\; 2 C^{\<tree21>}_\eps 
\begin{tikzpicture}[baseline={([yshift=.5ex]current bounding box.center)}]
\node [plainNode] (C) {}; \node [yshift=-0.15cm] at (C) {$z$};
\node [grayNode] (CU) [above=0.5cm of C] {} edge [dashed, deriv] (C);
\end{tikzpicture},
\end{equs}
where the renormalization constant $C^{\<tree21>}_\eps$ is given by
\begin{equ}[e:Constant2]
 C^{\<tree21>}_\eps =
\begin{tikzpicture}[baseline={([yshift=0ex]current bounding box.center)}]
\node [plainNode] (C) {}; \node [yshift=-0.2cm] at (C) {\scriptsize$0$};
\node [grayNode] (CR) [above right=0.5cm of C] {} edge [dashed, deriv] (C);
\node [grayNode] (CL) [above left=0.5cm of C] {} edge [arrow, deriv] (C);
\draw [dashed, deriv] (CR) -- (CL);
\end{tikzpicture}
= \int_{\R^2} K'\ast K_\eps'(w)K_\eps'(w)\dd w\;.
\end{equ}
It is evident that $C^{\<tree21>}_\eps=0$ since the quantity inside the integral is odd (because the heat kernel is even and $\varrho$ symmetric). It is natural to set the corresponding kernels of the limiting object $X_{3}^{\<tree22>}$ and $X_{1}^{\<tree22>}$ to be given by the same expression but replacing the dashed arrows by solid ones, and, consequently, $X^{\<tree22>} = X_{3}^{\<tree22>} + X_{1}^{\<tree22>}$. Using \cite[Lem.~10.14]{Hai14}, we conclude that the function $\mathcal{K}_{3}^{\<tree22>}$ is singular of order $\zeta = -1$, which thanks to Proposition \ref{p:ChaosKernelsConv} implies that $X_3^{\<tree22>}$ belongs to the space $\CC_1^{\alpha,\s}$ with $\alpha < -\frac{1}{2}$. Furthermore, we note that $\mathcal{W}_{\eps,1}^{\<tree22>}$ can be written as
\begin{equ}
 \mathcal{W}_{\eps,1}^{\<tree22>}(z) = 2 \int_{\R^2} \int_{\R^2} \SR Q_\eps (z - \bar z)\, K_\eps'(\bar z - z_1) \xi(z_1) \dd \bar z\, \dd z_1\;,
\end{equ}
where the kernel $Q_\eps$ is defined by the diagram
$
Q_\eps (z - \bar z) = 
\begin{tikzpicture}[baseline={([yshift=0ex]current bounding box.center)}]
\node [plainNode] (C) {}; \node [yshift=-0.2cm] at (C) {\scriptsize$z$};
\node [grayNode] (CR) [above right=0.5cm of C] {} edge [dashed, deriv] (C);
\node [plainNode] (CL) [above left=0.5cm of C] {} edge [arrow, deriv] (C); 
\node [xshift=-0.2cm] at (CL) {\scriptsize$\bar z$};
\draw [dashed, deriv] (CR) -- (CL);
\end{tikzpicture}
$,
and the operator $\SR$ acts, for any test function $\psi$, as
\begin{equ}[eq:RQ]
 \langle \SR Q_\eps, \psi \rangle = \langle Q_\eps, \psi - \psi(0)\rangle.
\end{equ}
Now, \cite[Lem.~10.16]{Hai14} implies that the kernel $z \mapsto \int_{\R^2} \SR Q_\eps (z - \bar z)\, K_\eps'(\bar z)\, \dd \bar z$ has singularity of order $-2$. Hence, the function $\mathcal{K}_{1}^{\<tree22>}$ is of order $\zeta = -1$, and Proposition \ref{p:ChaosKernelsConv} yields $X_1^{\<tree22>} \in \CC_1^{\alpha,\s}$ with $\alpha < -\frac{1}{2}$. Thus the process $X^{\<tree22>}$ has the same regularity and the estimate \eqref{e:KEpsConverge} can be obtained as above.

As for $X^{\<tree122>}$, applying Corollary~\ref{c:KConvol} to the result above, we get that $X^{\<tree122>} = K' \ast X^{\<tree22>}$ belongs to $\CC^{\beta/2} ([-1, 1], \CC^{\alpha - \beta}(\R))$ for $\alpha \in (0, \frac{1}{2})$ and $\beta \in (0, \alpha)$. The bound \eqref{e:KEpsConvergeTime} can be shown similarly.

Now, we will define the product $X^{\<tree21>}$. To this end, we use Wick's lemma to write $X_\eps^{\<tree11>} X_\eps^{\<tree1>} = X_\eps^{\<tree21>} + C^{\<tree21>}_\eps$, where the terms $X_\eps^{\<tree21>}$ and $C^{\<tree21>}_\eps$ belong to the second and zeroth homogeneous Wiener chaoses respectively, and they are given by
\begin{equ}
\mathcal{W}_\eps^{\<tree21>}(z) = 
\begin{tikzpicture}[baseline={([yshift=-.5ex]current bounding box.center)}]
\node [plainNode] (C) {}; \node [yshift=-0.15cm] at (C) {$z$};
\node [plainNode] (CR) [above right=0.5cm of C] {} edge [dashed, deriv] (C);
\node [grayNode] (CL) [above left=0.5cm of C] {} edge [arrow, deriv] (C);
\node [plainNode] (CLU) [above=0.5cm of CL] {} edge [dashed, deriv] (CL);
\end{tikzpicture}
\end{equ}
and \eqref{e:Constant2} respectively. In particular, this implies that $X^{\<tree21>}$ is defined via the kernel $\mathcal{W}^{\<tree21>}$ which is obtained from $\mathcal{W}_\eps^{\<tree21>}$ by replacing all the approximating kernels $K_\eps$ by $K$. Then \cite[Lem.~10.14]{Hai14} implies that the respective function $\mathcal{K}^{\<tree21>}$ has singularity of order $\zeta = - \kappa$ for every $\kappa > 0$, which thanks to Proposition~\ref{p:ChaosKernelsConv} yields $X^{\<tree21>} \in \CC^{\beta/2} ([-1,1], \CC^{\alpha - \beta}(\R))$ with $\alpha < 0$ and $\beta \in (0, \frac{1}{2})$. The bound \eqref{e:KEpsConvergeTime} can be proved as before.

In order to treat the process $X^{\<tree124>}$, we will first define the product $(X^{\<tree12>})^2$. Applying Wick's lemma, we conclude that we can write $(X_\eps^{\<tree12>})^2 = X_{\eps, 4}^{\<tree12>\<tree12>} + X_{\eps, 2}^{\<tree12>\<tree12>} + X_{\eps, 0}^{\<tree12>\<tree12>}$, where these terms belong to the $4$th, $2$nd and $0$th Wiener chaoses respectively. Since the spatial derivative will kill the constant term, we will not take it into consideration. The first two processes are defined by the kernels
\begin{equ}
\mathcal{W}_{\eps,4}^{\<tree12>\<tree12>}(z) =
\begin{tikzpicture}[baseline={([yshift=-.5ex]current bounding box.center)}]
\node [plainNode] (C) {}; \node [yshift=-0.15cm] at (C) {$z$};
\node [grayNode] (CU) [above left=0.5cm of C] {} edge [arrow, deriv] (C);
\node [plainNode] (CL) [above left=0.4cm of CU] {} edge [dashed, deriv] (CU);
\node [plainNode] (CR) [above right=0.4cm of CU] {} edge [dashed, deriv] (CU);
\node [grayNode] (D) [above right=0.5cm of C] {} edge [arrow, deriv] (C);
\node [plainNode] (DL) [above left=0.4cm of D] {} edge [dashed, deriv] (D);
\node [plainNode] (DR) [above right=0.4cm of D] {} edge [dashed, deriv] (D);
\end{tikzpicture}, \qquad 
\mathcal{W}_{\eps,2}^{\<tree12>\<tree12>}(z) =
\begin{tikzpicture}[baseline={([yshift=-.5ex]current bounding box.center)}]
\node [plainNode] (C) {}; \node [yshift=-0.15cm] at (C) {$z$};
\node [grayNode] (CU) [above left=0.5cm of C] {} edge [arrow, deriv] (C);
\node [plainNode] (CL) [above left=0.5cm of CU] {} edge [dashed, deriv] (CU);
\node [grayNode] (CR) [above right=0.5cm of CU] {} edge [dashed, deriv] (CU);
\node [grayNode] (D) [above right=0.5cm of C] {} edge [arrow, deriv] (C);
\node [plainNode] (DR) [above right=0.5cm of D] {} edge [dashed, deriv] (D);
\draw [dashed, deriv] (CR) -- (D);
\end{tikzpicture},
\end{equ}
and we define the kernels $\mathcal{W}_{4}^{\<tree12>\<tree12>}$ and $\mathcal{W}_{2}^{\<tree12>\<tree12>}$ of the limiting object via the same diagrams, but replacing all the dashed arrows by solid. Then \cite[Lem.~10.14]{Hai14} implies that the corresponding covariance functions $\mathcal{K}_{4}^{\<tree12>\<tree12>}$ and $\mathcal{K}_{2}^{\<tree12>\<tree12>}$ have singularities of order $\zeta = -\kappa$ for every $\kappa > 0$. Applying now consecutively Proposition~\ref{p:ChaosKernelsConv} first and Corollary~\ref{c:Schauder}, we conclude that the process $P'\ast X^{\<tree12>\<tree12>}$ belongs to $\CC^{\alpha,\s}_1$ with $\alpha \in (0,1)$.

To treat the process $X^{\<tree1222>}$, the same strategy as for $X^{\<tree124>}$ applies so that it suffices to obtain suitable bounds on $X^{\<tree222>}$, formally given by the product of $X^{\<tree122>}$ and $X^{\<tree1>}$. Using Wick's lemma, we can write $X_\eps^{\<tree222>} = X_{\eps,4}^{\<tree222>} + X_{\eps,2}^{\<tree222>}$, where these terms belong to the $4$th and $2$nd homogeneous Wiener chaoses respectively and we omitted the $0$-th chaos component since the spatial derivative will anyway kill it. The kernels defining $X_{\eps,4}^{\<tree222>}$ and $X_{\eps,2}^{\<tree222>}$ are given by
\begin{equ}
\mathcal{W}_{\eps,4}^{\<tree222>}(z) =
\begin{tikzpicture}[baseline={([yshift=-.5ex]current bounding box.center)}]
\node [plainNode] (C) {}; \node [xshift=-0.15cm] at (C) {$z$};
\node [grayNode] (CU) [right=0.5cm of C] {} edge [arrow, deriv] (C);
\node [plainNode] (D) [below right=0.5cm of C] {} edge [dashed, deriv] (C);
\node [grayNode] (CUU) [right=0.5cm of CU] {} edge [arrow, deriv] (CU);
\node [plainNode] (M) [below right=0.5cm of CU] {} edge [dashed, deriv] (CU);
\node [plainNode] (L) [above right=0.5cm of CUU] {} edge [dashed, deriv] (CUU);
\node [plainNode] (R) [below right=0.5cm of CUU] {} edge [dashed, deriv] (CUU);
\end{tikzpicture}\;,
\end{equ}
as well as
\begin{equ}
\mathcal{W}_{\eps,2}^{\<tree222>}(z) =
\begin{tikzpicture}[baseline={([yshift=-.5ex]current bounding box.center)}]
\node [plainNode] (C) {}; \node [xshift=-0.15cm] at (C) {$z$};
\node [grayNode] (CU) [right=0.5cm of C] {} edge [verywavy] (C);
\node [grayNode] (CUU) [right=0.5cm of CU] {} edge [arrow, deriv] (CU);
\node [plainNode] (L) [above right=0.5cm of CUU] {} edge [dashed, deriv] (CUU);
\node [plainNode] (R) [below right=0.5cm of CUU] {} edge [dashed, deriv] (CUU);
\end{tikzpicture}
+ 2 \Big(
\begin{tikzpicture}[baseline={([yshift=-.5ex]current bounding box.center)}]
\node [plainNode] (C) {}; \node [xshift=-0.15cm] at (C) {$z$};
\node [grayNode] (CU) [right=0.5cm of C] {} edge [arrow, deriv] (C);
\node [grayNode] (CUU) [right=0.5cm of CU] {} edge [verywavy] (CU);
\node [plainNode] (L) [above right=0.5cm of CUU] {} edge [dashed, deriv] (CUU);
\node [plainNode] (R) [below right=0.5cm of C] {} edge [dashed, deriv] (C);
\end{tikzpicture}
+
\begin{tikzpicture}[baseline={([yshift=-.5ex]current bounding box.center)}]
\node [plainNode] (C) {}; \node [xshift=-0.15cm] at (C) {$z$};
\node [grayNode] (CU) [right=0.5cm of C] {} edge [arrow, deriv] (C);
\node [grayNode] (CUU) [right=0.5cm of CU] {} edge [arrow, deriv] (CU);
\node [plainNode] (M) [below right=0.5cm of CU] {} edge [dashed, deriv] (CU);
\node [plainNode] (R) [below right=0.5cm of CUU] {} edge [dashed, deriv] (CUU);
\node [grayNode] (N) [above=0.3cm of CU] {} edge [dashed, deriv] (C);
\draw [dashed, deriv] (N) -- (CUU);
\end{tikzpicture}
\;\Big)\;,
\end{equ}
where we use the kernel 
$
\SR Q_\eps (z - \bar z) = 
\begin{tikzpicture}[baseline={([yshift=-0.5ex]current bounding box.center)}]
\node [plainNode] (CL) {};
\node [xshift=-0.2cm] at (CL) {\scriptsize$z$};
\node [plainNode] (CR) [right=0.5cm of CL] {} edge [verywavy] (CL); 
\node [xshift=0.2cm] at (CR) {\scriptsize$\bar z$};
\end{tikzpicture}
$
 defined in \eqref{eq:RQ}. Using \cite[Lem.~10.14, 10.16]{Hai14} we conclude that the orders of the singularities of the corresponding covariance functions are $\zeta = -\kappa$ for every $\kappa > 0$. Proposition \ref{p:ChaosKernelsConv} now yields $X^{\<tree222>} \in \CC^{\beta/2} ([-1,1], \CC^{\alpha - \beta}(\R))$ with $\alpha < 0$ and $\beta \in (0, \frac{1}{2})$. Convergence of the approximate processes we show in the usual way. 

Now, we turn to analysis of the remainders defined in \eqref{e:Remainders}. Since these definitions are continuous with respect to the controlling processes, the identities \eqref{e:Remainders} will automatically hold in the limit $\eps \to 0$. It remains only to show that the processes belong to the required spaces. In order to show that $X^{\<tree21>}$ is ``controlled'' by $X^{\<tree11>}$, we write 
\begin{equ}
R^{\<tree21>}(t, x; y) \eqdef X^{\<tree21>}(t, y) - X^{\<tree11>}(t,x)\, X^{\<tree1>} (t,y) = R^{\<tree21>}_{2, x}(t,y) + R^{\<tree21>}_{0,x}(t,y)\;, 
\end{equ}
where the processes $R^{\<tree21>}_{2,x}$ and $R^{\<tree21>}_{0,x}$ belong to the $2$nd and $0$th homogeneous Wiener chaoses respectively and are characterised by the kernels 
\begin{equ}
\mathcal{W}_{2, x}^{\<tree21>}(t,y) = \Big(
\begin{tikzpicture}[baseline={([yshift=0ex]current bounding box.center)}]
\node [plainNode] (C) {}; \node [yshift=-0.15cm] at (C) {\scriptsize$(t,y)$};
\node [grayNode] (CU) [above=0.5cm of C] {} edge [arrow, deriv] (C);
\node [plainNode] (CUU) [above=0.5cm of CU] {} edge [arrow, deriv] (CU);
\end{tikzpicture} 
-
\begin{tikzpicture}[baseline={([yshift=0ex]current bounding box.center)}]
\node [plainNode] (C) {}; \node [yshift=-0.15cm] at (C) {\scriptsize$(t,x)$};
\node [grayNode] (CU) [above=0.5cm of C] {} edge [arrow, deriv] (C);
\node [plainNode] (CUU) [above=0.5cm of CU] {} edge [arrow, deriv] (CU);
\end{tikzpicture} 
\Big)
\begin{tikzpicture}[baseline={([yshift=0ex]current bounding box.center)}]
\node [plainNode] (C) {}; \node [yshift=-0.15cm] at (C) {\scriptsize$(t,y)$};
\node [plainNode] (CU) [above=0.5cm of C] {} edge [arrow, deriv] (C);
\end{tikzpicture}
\;, \qquad
\mathcal{W}^{\<tree21>}_{0, x}(t,y) =
\begin{tikzpicture}[baseline={([yshift=0ex]current bounding box.center)}]
\node [plainNode] (C) {}; \node [yshift=-0.2cm] at (C) {\scriptsize$(t,x)$};
\node [grayNode] (CU) [above=0.5cm of C] {} edge [arrow, deriv] (C);
\node [grayNode] (CUR) [right=0.5cm of CU] {} edge [arrow, deriv] (CU);
\node [plainNode] (C2) [below=0.5cm of CUR] {}; \node [yshift=-0.2cm] at (C2) {\scriptsize$(t,y)$};
\draw [arrow, deriv] (CUR) -- (C2);
\end{tikzpicture}\;.
\end{equ}
For the variables $z_i = (t, x_i)$, we can write the covariance function of $R_{2,x}^{\<tree21>}$ as
\begin{equ}
\mathcal{K}^{\<tree21>}_{2, x}(z_1, z_2) = \Big(\CK^{\<tree11>}(z_1,z_2)-\CK^{\<tree11>}(z_1,z)-\CK^{\<tree11>}(z,z_2)+\CK^{\<tree11>}(z,z)\Big)\mathcal{K}^{\<tree1>}(z_1, z_2)\;.
\end{equ}
We have already established above that the functions $\mathcal{K}^{\<tree11>}$ and $\mathcal{K}^{\<tree1>}$ have singularities of orders $\zeta_1 = 1-\kappa$ and $\zeta_2=-1$ respectively, which implies, subtracting to each $\CK^{\<tree11>}$, $\CK^{\<tree11>}$, and applying~\cite[Lem. 10.14]{Hai14} the bound
\begin{equ}
|\mathcal{K}_{2, x}^{\<tree21>}(z_1, z_2)| \lesssim \Bigl( |x_1 - x_2|^{\zeta_1} + |x_1 - x|^{\zeta_1} + |x_2 - x|^{\zeta_1} \Bigr) |x_1 - x_2|^{\zeta_2}\;.
\end{equ}
The same result holds for $R_{0,x}^{\<tree21>}$, because it is deterministic and the kernel $\mathcal{W}_{0, x}^{\<tree21>}$ satisfies $|\mathcal{W}^{\<tree21>}_{0, x}(t,y)| \lesssim |y-x|^{-\kappa}$ for every $\kappa > 0$. It follows furthermore from the previous results that $R^{\<tree21>}$ has a strictly positive H\"{o}lder regularity in time, which yields from Lemma~\ref{l:Control} that $R^{\<tree21>}(t) \in \CL_1^\beta(\R)$ almost surely for any $\beta < {\zeta_1 + \zeta_2 \over 2}$.

Finally, we turn to the limiting process of $R^{\<tree1222>}_\eps$. Notice that, by definition $R^{\<tree1222>}_\eps(z, \cdot)=P'\ast R^{\<tree222>}_\eps(z, \cdot)$, where the convolution is taken in the second variable with fixed $z$, and where $R_\eps^{\<tree222>}(z; \bar z) = X_\eps^{\<tree222>} (\bar z) - X_\eps^{\<tree122>}(z)\, X_\eps^{\<tree1>} (\bar z)$. We can define $R^{\<tree222>}$ by removing $\eps$ in the last definition and we can write
\begin{equ}
R^{\<tree222>}(z; \bar z) = R^{\<tree222>}_{4,z}(\bar z) + R^{\<tree222>}_{2,z}(\bar z) \;,
\end{equ}
where the process $R^{\<tree222>}_{i,z}(\bar z)$ belongs to the $i$th Wiener chaos and is of the form \eqref{e:XWienerIntegral}. Notice that we neglected the $0$-th chaos component because it will be anyway killed by the spatial derivative. Acting as before, it is not difficult to show that the respective functions \eqref{e:KDef} satisfy the bound \eqref{eControlBound} with $\zeta < 0$. Applying now Lemma~\ref{l:Control} we conclude that $R^{\<tree222>}_{(t, \cdot)}(t, \cdot) \in \CL_1^\beta(\R)$ almost surely for any $\beta < 0$.  As a consequence, setting $\zeta_z(s,\cdot)\eqdef R^{\<tree222>}_z(s,\cdot)$, we have
\begin{equ}
\big| \langle \zeta_z(t,\cdot), \varphi_x^\lambda \rangle \big| \leq \lambda^{\beta_{\<tree222>}}\Vert \X \Vert_{\CX} \;,
\end{equ}
uniformly over $z=(t,x),\in (0,T]\times\T$, $\lambda \in (0,1]$ and $\varphi \in \CB_0^r(\R)$. 
Moreover, if $\bar z=(\bar t,\bar x)$ is another point in $(0,T]\times\T$, $\lambda \in (0,1]$, then 
\begin{equ}
\big|\langle \zeta_{\bar z}(t,\cdot)-\zeta_z(t,\cdot), \varphi_x^\lambda \rangle \big|= \big|X^{\<tree122>}(\bar z)-X^{\<tree122>}(z)\big|\big|\langle X^{\<tree1>} (t, \cdot), \varphi_x^\lambda \rangle\big|\lesssim \lambda^{\alpha_{\<tree1>}} \snorm{\bar z - z}^{\alpha_{\<tree122>}}\Vert \X \Vert^2_{\CX}
\end{equ}
uniformly over $z=(t,x),\,\bar z=(\bar t,\bar x) \in (0,T]\times\T$ such that $|t-\bar t|\leq \onorm{t,\bar t}$, $\lambda \in (0,1]$ and $\varphi \in \CB_0^r(\R)$. 
Now, notice that, in the expression for the increments, the only term that is tested against $\varphi_x^\lambda$ is $X^{\<tree1>}$, for which it is straightforward to verify that the assumptions of Lemma~\ref{l:Schwartz} are matched so that the previous holds also for $\varphi$ Schwartz. 
In conclusion, upon choosing $\alpha=\alpha_{\<tree1>}$, $\beta=\beta_{\<tree222>}$ and $\rho=0$ in Proposition~\ref{p:HeatBounds} (as well as $\theta_1=\theta_2=0$), the proof of the proposition is concluded. 
\end{proof}

\subsection{Discrete singular kernels}\label{sec:DKernels}

As we saw in the previous section, in order to obtain suitable bounds for the stochastic processes under study the crucial notion we needed is that of kernels with a singularity at $0$ of prescribed order. In the discrete setting, Definition~\ref{d:SingularFunction} becomes

\begin{definition}\label{d:DSingKer}
Let $\{\CK^\eps\}_\eps$ be a sequence of functions on $\Lambda_\eps^\s$ supported in a ball around the origin. We say that it is of order $\zeta\in\R$ if for some $m\in\N$ the quantity
\begin{equation}\label{def:DSingKer}
\VERT \CK^\eps\VERT_{\zeta;m}^{(\eps)}\eqdef \max_{|k|_\s\leq m}\sup_{z\in\Lambda_\eps^\s} \frac{\left|\bar{D}^k_\eps \CK^\eps(z)\right|}{\|z\|_{\s,\eps}^{\zeta-|k|_\s}}
\end{equation}
is bounded uniformly in $\eps$. In the previous, $k=(k_0,k_1)\in\N^2$, $|k|_\s\eqdef 2k_0+k_1$, $\|z\|_{\s,\eps}=\|z\|_\s\vee \eps$ and $\bar D^k_\eps$ is the forward discrete derivative in space and time, i.e. $\bar D^k_\eps\eqdef \bar{D}_{t,\eps^2}^{k_0}\bar{D}_{x,\eps}^{k_1}$, where for $f$ on $\Lambda_\eps^\s$ we set $\bar D_{t,\eps^2}f(x) = \eps^{-2}\big(f(t+\eps^2,x)-f(t,x)\big)$ and $\bar D_{x,\eps}f(x) = \eps^{-1}\big(f(t,x+\eps)-f(t,x)\big)$.
\end{definition}

\begin{remark}
Notice that in~\eqref{def:DSingKer}, we are not taking a general discrete derivative operator as in Lemma~\ref{lemma:DHeatKernels} but we made the specific choice of $\bar{D}$. Even if not crucial, it is a convention that will turn out to be convenient when proving what it means to multiply or convolve singular kernels. Moreover, this does not affect in \textit{any} way the definition of the stochastic objects under study. 
\end{remark}

Before stating a number of lemmas representing the discrete counterpart of those in~\cite[Sec. 10.3]{Hai14}, let us show that the discrete heat kernel (and its discrete derivatives), as given in~\eqref{e:DSTHeatKernel} can indeed be decomposed as the sum of a singular and a ``regular" part. 

\begin{lemma}\label{l:DKernelDec}
The kernel $P^\eps$ can be written as $P^\eps=K^\eps+\hat K^\eps$, in such a way that the identity
\[
P^\eps\ast_\eps u(z)=K^\eps\ast_\eps u(z)+\hat K^\eps\ast_\eps u(z)
\]
holds for all $z\in\Lambda_\eps^\s$ and all spatially periodic functions $u$ on $\Lambda_{\eps^2}^+\times\Lambda_\eps$ (here $\Lambda_{\eps^2}^+\eqdef \Lambda_{\eps^2}\cap \R^+$). In the previous decomposition, $K^\eps$ is a singular kernel of order $-1$ according to Definition~\ref{d:DSingKer} and $\hat K^\eps$ is compactly supported, non-anticipative and such that, for any $r>0$, $\|\hat K^\eps\|_{\CC^{r,\s}}^{(\eps)}$ is bounded uniformly in $\eps$.
\end{lemma} 

\begin{proof}
The statement above is similar in spirit to Lemma 5.5 in~\cite{Hai14} and Lemma 5.4 in~\cite{HM15}, but, in a sense, we only need to \textit{localize} our discrete heat kernel around the origin as we did in~\eqref{def:Kernel}. Indeed, it suffices to consider a smooth compactly supported function $\chi:\R^2\to[0,1]$ such that $\chi(z)=1$ for $\|z\|_\s\leq 1$ and $\chi(z)=0$ for $\|z\|_\s> \frac{3}{2}$ (of course the choice of $\frac{3}{2}$ is completely arbitrary). Then we set, for $z\in\Lambda_\eps^\s$
\[
K^\eps (z)=\chi(z)P^\eps(z)\qquad\text{and}\qquad \tilde{K}^\eps(z)= (1-\chi(z))P^\eps(z)\;.
\]
At this point the decomposition in the statement is obvious and it remains to prove that indeed $K^\eps$ and $\hat K^\eps$ satisfy the correct properties. Now, $\tilde{K}^\eps$ is clearly non-anticipative and such that, for any $r>0$, $\|\tilde K^\eps\|_{\CC^{r,\s}}^{(\eps)}$ because so is $P^\eps$ outside of a ball centered at the origin. In order to make it compactly supported, we recall the procedure outlined in Lemma 7.7 of~\cite{Hai14}. Assume for simplicity that $u$ has period 1. Fix a function $\varrho:\R\to[0,1]$, compactly supported in a ball of radius $C$ centered at $0$ such that $\sum_{k\in\Z}\varrho(x+k)=1$ for every $x\in\Lambda_\eps$. Then, one can straightforwardly verify that
\[
\hat K^\eps(t,x)\eqdef \sum_{k\in\Z}\tilde{K}^\eps(t,x+k)\varrho(x)
\] 
does the job. Concerning $K^\eps$ we have to prove that there exists a $c>0$ such that the quantity in~\eqref{def:DSingKer} is bounded upon taking  $\zeta=-1$. Let $k=(k_0,k_1)\in\N^2$ and recall that, for $\|z\|_\s\leq 1 - \eps |k|_\s$, $\bar{D}^k_\eps K^\eps(z)$ coincides with $\bar{D}^k_\eps P^\eps(z)$ so that we can focus on the latter. The proof proceeds along the same steps as the one of Lemma~\ref{lemma:DHeatKernels} hence we will simply point out the main steps. Fix $c>0$ such that $\lfloor c^{-1}\rfloor>1+|k|_\s$ and, for $(t,x)=z\in\Lambda_{\eps^2}\times \R$, set
\[
F^\eps_t(x)\eqdef |t|_\eps^{1+|k|_\s} \left(\bar{D}_\eps^k P^\eps_t\right)(|t|_\eps x)\;.
\]
We want to obtain uniform bounds on its Fourier transform, which is given by
\[
\widehat{F^\eps_t}(\xi)=\hat{\psi}(-\bar{\xi})\left(\frac{\widehat{\nu}(\bar{\xi})}{(\eps|t|_\eps^{-1})^2}\right)^{k_0}\left(\frac{\widehat{\pi_d}(-\bar{\xi})}{\eps|t|_\eps^{-1} }\right)^{k_1} \left(1+\frac{\widehat{\nu}(\bar{\xi})}{2\bar{\nu}}\right)^{\lfloor t\eps^{-2}\rfloor}
\]
where, to streamline the notation, we have set $\bar{\xi}\eqdef \eps |t|_\eps^{-1} \xi$, and  $\pi_d\eqdef \delta_1(\dd x)-\delta_0(\dd x)$, which obviously satisfies Assumption~\ref{a:pi}. Notice that the simplifications that lead to the expression above were possible only because  $t\in\Lambda_{\eps^2}$.

Since all the quantities in the previous expression were already considered in the proof of the aforementioned Lemma, we reach the same conclusions, namely, for any $m\in\N$, $m< \lfloor c^{-1}\rfloor$, there exists $C>0$ such that
\begin{equ}[b:DSingKer]
\bigl|\bar{D}_\eps^k P^\eps_t (x)\bigr| \leq C   |t|_{\eps}^{-1- |k|_\s+m}|x|^{-m}\;.
\end{equ}
uniformly over $z=(t,x)\in\R^2$ such that $\|z\|_{\s}\geq  c\eps$, and $\eps>0$, where $|t|_\eps= |t|^{\frac{1}{2}}$, for $|t|\geq c\eps^2$, and is equal to $\eps$ otherwise. If instead $\|z\|_\s <c \eps$ the previous bound holds with $m=0$. To see how to get~\eqref{def:DSingKer}, we proceed as follows. In case $\|z\|_\s<c \eps$, there is nothing to prove. Otherwise, recall that $\|z\|_\s=|t|^{\frac{1}{2}}\vee |x|$, hence if $|t|^{\frac{1}{2}}\geq |x|$ we choose $m=0$,  while if $|x|>|t|^{\frac{1}{2}}$ we take $m=1+|k|_\s$. In both cases the right-hand side of~\eqref{b:DSingKer} becomes $\|z\|_\s^{-1-|k|_\s}$ which is exactly what requested, hence the proof is complete.  
\end{proof}

\begin{remark}\label{rem:cDSingKer}
It might seem that in the previous proof we showed a slightly different result, namely that there exists a $c>0$ such that, for $\zeta\eqdef -1$ and $k\in\N^2$
\[
|\bar{D}^k K^\eps(z)|\lesssim
\begin{cases}
\|z\|_\s^{\zeta-|k|_\s}, &\text{if $\|z\|_\s\geq c \eps$,}\\
\eps^{\zeta-|k|_\s}, &\text{if $\|z\|_\s< c \eps$.}
\end{cases}
\] 
The point is that if the previous holds for {\it some} $c>0$ than it holds for {\it any} $c$. Indeed, if for example (the other case is completely analogous) we take a $\bar{c}>c$, then for $\|z\|_\s\geq \bar{c} \eps$ and $\|z\|_\s< c \eps$ the previous bound still holds, so that the only situation to discuss is when $c\eps\leq\|z\|_\s< \bar{c}\eps$, i.e. $\|z\|_\s\sim \eps$. But in this case, one immediately sees that $\|z\|_\s^{\zeta-\|k\|_\s}\lesssim \eps^{\zeta-\|k\|_\s}$.
\end{remark}

In the following lemmas, we collect some results that tell us how these discrete singular kernels behave under various operations. Their proof is identical to the one of their continuous counterparts given in~\cite[Sec. 10.3]{Hai14} so that we limit ourselves to point out the differences.

\begin{lemma}\label{lemma:OpDSingKer1}
Let $K_1^\eps$ and $K_2^\eps$ be functions on $\Lambda_\eps^\s$ of order $\zeta_1$ and $\zeta_2\in\R$ respectively, according to Definition~\ref{d:DSingKer}. Let $B^\eps$ be the operator defined in~\eqref{e:Dproduct}, where $\mu$ satisfies Assumption~\ref{a:MeasureMu}. Then for any $m\in\N$ the following bounds hold uniformly in $\eps$:
\begin{equation}\label{b:ProdDSingKer}
\VERT B^\eps (K_1^\eps,K_2^\eps)\VERT^{(\eps)}_{\zeta_1+\zeta_2;m}\lesssim \VERT K_1^\eps\VERT^{(\eps)}_{\zeta_1;m}\VERT K_2^\eps\VERT^{(\eps)}_{\zeta_2;m}\;.
\end{equation}

If $\zeta_1\wedge \zeta_2>-|\s|$ and $\bar{\zeta}\eqdef \zeta_1+\zeta_2+|\s|$ is strictly negative, then $K_1^\eps\ast_\eps K_2^\eps$ is of order $\bar{\zeta}$ and, for $m\in\N$ such that $4m\eps<1$, we have
\begin{equation}\label{b:ConvDSingKer1}
\VERT K_1^\eps\ast_\eps K_2^\eps\VERT^{(\eps)}_{\bar{\zeta};m}\lesssim \VERT K_1^\eps\VERT^{(\eps)}_{\zeta_1;m}\VERT K_2^\eps\VERT^{(\eps)}_{\zeta_2;m}
\end{equation}
uniformly in $\eps$. If instead $\bar{\zeta}\in\R_+\setminus\N$, the function
\[
\bar{K}^\eps(z)\eqdef K_1^\eps\ast_\eps K_2^\eps(z)-\sum_{|k|_\s<\bar{\zeta}} \frac{(z)_{\eps,k}}{k!}\bar{D}_\eps^k (K_1^\eps\ast_\eps K_2^\eps) (0)
\]
where $k=(k_0,k_1)\in\N^2$ and $(z)_{k,\eps}=(z_0,z_1)_{k,\eps}=\prod_{i=0,1}\prod_{0\leq j<k_i}(z_i-\eps^{s_i} j)$, satisfies 
\begin{equation}\label{b:ConvDSingKer2}
\VERT \bar{K}^\eps\VERT^{(\eps)}_{\bar{\zeta};m}\lesssim \VERT K_1^\eps\VERT^{(\eps)}_{\zeta_1;\bar{m}}\VERT K_2^\eps\VERT^{(\eps)}_{\zeta_2;\bar{m}}
\end{equation}
for $4m\eps<1$ and $\bar{m}=m\vee(\lfloor\bar{\zeta}\rfloor +2)$. 
\end{lemma}

\begin{proof}
While the second and third bounds can be obtained along the lines of the proof of (10.10) and (10.12) in~\cite{Hai14},~\eqref{b:ProdDSingKer} requires a special treatment. 
So, fix $m\in\N$ and let $k\in\N^2$ be such that $|k|_\s\leq m$. Thanks to the discrete Leibniz rule and Definition~\ref{d:DSingKer}, we have
\begin{multline*}
\bar{D}_\eps^k B^\eps(K_1^\eps,K_2^\eps)(z)=\sum_{l\leq k} \binom{k}{l} \int \bar{D}^{l}_\eps K_1^\eps(z_1)\bar{D}_\eps^{k-l} K_2^\eps(z_2)\mu_\eps(\dd y_1,\dd y_2)\\
\lesssim \VERT K_1^\eps\VERT^{(\eps)}_{\zeta_1;m}\VERT K_2^\eps\VERT^{(\eps)}_{\zeta_2;m} \sum_{l\leq k} \binom{k}{l} \int 
\|z_1\|_{\s,\eps}^{\zeta_1-|l|_\s}\|z_2\|_{\s,\eps}^{\zeta_2-|k-l|_\s}\mu_\eps(\dd y_1,\dd y_2)
\end{multline*}
where the sum runs over $l=(l_0,l_1)\in\N^2$ and, to simplify the notation, we set $z_1\eqdef z+(0,y_1)$ and $z_2\eqdef z+(0,y_2)+(\eps^2 l_0,\eps l_1)$
In order to conclude, it clearly suffices to show that both $\|z_1\|_{\s,\eps}^{\zeta_1-|l|_\s}$ and $\|z_2\|_{\s,\eps}^{\zeta_2-|k-l|_\s}$ can be bounded by the corresponding quantity in which $z_i$, $i=1,2$, is replaced by $z$.
We will prove it for $z_2$, being the proof for $z_1$ similar and easier. 

Fix $c\eqdef 1+ 3m+3R_\mu$. Notice that for $\|z\|_\s\geq c\eps$, $\|z_2\|_\s\geq \eps$ and, since trivially $\eps m$ and $\eps R_\mu$ are less or equal to $\frac{1}{3}\|z\|_\s$, by the triangular inequality we have $\|z_2\|_\s\sim\|z\|_\s$. 
Hence we immediately get $\|z_2\|_\s^{\zeta_2-|k-l|_\s}\lesssim \|z\|_\s^{\zeta_2-|k-l|_\s}$. On the other hand, in case $\|z\|_\s<c\eps$, either $\|z_2\|_\s\geq \eps$ so that $\|z_2\|_\s^{\zeta_2-|k-l|_\s}\lesssim \eps^{\zeta_2-|k-l|_\s}$, or $\|z_2\|_\s<\eps$, but in this situation $\|z_2\|_{\s,\eps}^{\zeta_2-|k-l|_\s}=\eps^{\zeta_2-|k-l|_\s}$. 
By Remark~\ref{rem:cDSingKer} and the fact that, by assumption, $\mu$ has finite mass, the conclusion follows. 
\end{proof}

In the analysis of the stochastic terms, we will run into the convolution of two discrete singular kernels whose order is less or equal than $-|\s|$, hence falling out of the range of applicability of the previous lemma. Nevertheless, we will need to obtain suitable bounds for these objects and, under specific assumptions, this is still possible. Let us begin with the following definition, which represents the discrete counterpart of \cite[Def. 10.15]{Hai14}. 

\begin{definition}\label{def:RenSingKer}
Let $-|\s|-1<\zeta\leq-|\s|$ and $K^\eps$ be a function on $\Lambda_\eps^\s$ of order $\zeta$. We define the discrete renormalized distribution $\SR_\eps K^\eps$ corresponding to $K^\eps$ as 
\[
\langle\SR_\eps K^\eps,\psi\rangle_\eps=\langle K^\eps, \psi-\psi(0)\rangle_\eps
\] 
for every smooth compactly supported test function $\psi$, where the pairing $\langle\cdot,\cdot\rangle_\eps$ is a Riemann sum approximation of the usual $L^2$ scalar product. 
\end{definition}

The proof of the following lemma is analogous to~\cite[Lem. 10.16]{Hai14}.

\begin{lemma}\label{lemma:OpDSingKer2}
Let $K_1^\eps$ and $K_2^\eps$ be two functions on $\Lambda_\eps^\s$ of order $\zeta_1$ and $\zeta_2$ respectively. Assume $-|\s|-1<\zeta_1\leq-|\s|$ and $-2|\s|-\zeta_1<\zeta_2\leq 0$ and set $\bar{\zeta}\eqdef \zeta_1+\zeta_2+|\s|$. Then, the function $\SR_\eps K_1^\eps\ast_\eps K_2^\eps$ is of order $\bar{\zeta}$ and, for $m\in\N$ such that $4m\eps<1$, we have uniformly in $\eps$:
\begin{equation}\label{b:ConvDSingKer3}
\VERT \SR_\eps K_1\ast_\eps K_2\VERT^{(\eps)}_{\bar{\zeta};m}\lesssim \VERT K_1\VERT^{(\eps)}_{\zeta_1;m}\VERT K_2\VERT^{(\eps)}_{\zeta_2;m+2}.
\end{equation} 
\end{lemma}

In the next lemma, whose proof is analogous to \cite[Lem. 10.18]{Hai14}, we show how it is possible to control the increment of a discrete singular kernel. 

\begin{lemma}\label{lemma:OpDSingKer3}
Let $K^\eps$ be a function on $\Lambda_\eps^\s$ of order $\zeta\leq 0$. Then for every $\kappa\in[0,1]$ and $z,\,\bar{z}\in\Lambda_\eps^\s$ we have uniformly in $\eps$:
\begin{equation}\label{b:IncDSingKer}
|K^\eps(z)-K^\eps(\bar{z})|\lesssim \|z-\bar{z}\|_{\s,\eps}^\kappa \left(\|z\|_{\s,\eps}^{\zeta-\kappa}+\|\bar{z}\|_{\s,\eps}^{\zeta-\kappa}\right)\VERT K^\eps\VERT_{\zeta,2}^{(\eps)}.
\end{equation} 
\end{lemma}

We conclude this section with a useful lemma which shows how a discrete singular kernel behaves when convolved with a discrete mollifier. Its proof is identical to that of \cite[Lem. 10.17]{Hai14}. 

\begin{lemma}\label{lemma:OpDSingKer4}
Let $\bar{\eps}\in[\eps,1]$ and $\psi^{\eps,\bar{\eps}}:\Lambda_\eps^\s\to\R_+$ be a function supported in a ball of radius $\bar{R}\eps$, $\bar{R}>0$, such that
\[
\eps^{|\s|}\sum_{w\in\Lambda_\eps^\s}\psi^{\eps,\bar{\eps}}(w)=1,\qquad\text{and}\qquad |\bar{D}_\eps^k\psi^{\eps,\bar{\eps}}(z)|\lesssim \bar{\eps}^{-|\s|-|k|_\s}
\]
for all $z\in\Lambda_\eps^\s$ and all $k\in\N^2$, $|k|_\s\leq m+2$ for some $m\in\N$. Let $K^\eps$ be a function on $\Lambda_\eps^\s$ of order $\zeta\in(-|\s|,0]$ and set $K^\eps_{\bar{\eps}}(z)\eqdef K^\eps\ast_\eps\psi^{\eps,\bar{\eps}}(z)$. Then, for all $\kappa\in[0,1]$, we have uniformly in $\eps$:
\begin{equation}\label{b:ApproxDSingKer}
\VERT K^\eps-K^\eps_{\bar{\eps}}\VERT_{\zeta-\kappa,m}^{(\eps)}\lesssim \bar{\eps}^\kappa \VERT K\VERT_{\zeta,m+2}
\end{equation} 
\end{lemma}

\subsection{Uniform bounds on the discrete controlling processes}
\label{sec:dnif_bounds_disc}

Let $\{\xi^\eps(z)\}_{z\in\Lambda_\eps^\s}$ be a family of i.i.d. normal random variables with mean $0$ and variance $\eps^{-3}$ on a probability space $(\Omega,\SF,\Prob)$ and consider the discrete controlling process $\X^\eps(\xi^\eps, \BC^\eps)$ introduced in Definition~\ref{def:DControlProc}. Arguing as in Section~\ref{sec:Criterion}, it is clear that each component of $\X^\eps$ belongs to a finite inhomogeneous Wiener chaos and each chaos component will be of the form
\begin{equ}[e:DXWienerIntegral]
Y^\eps(\varphi) = I^\eps_k\left(\langle \varphi\,, \CW^{\eps}\rangle_\eps\right),
\end{equ}
for a suitable discrete kernel $\CW^\eps$ on $H^{\otimes k}_\eps$, where $I^\eps_k$ is the $k$-th order Wiener-It\^o integral (sum) with respect the family $\{\xi^\eps(z)\}_z$, $\langle \cdot,\cdot\rangle_\eps$ is the usual discrete pairing (read, Riemann-sum approximation of the integral) and $H_\eps\eqdef \ell^2(\Lambda_{\eps^2}\times \T_\eps)$. 
Let $\varrho$ be a symmetric compactly supported smooth function integrating to 1 and, for $\bar \eps\geq \eps$, $\varrho_{\bar \eps}(t,x)\eqdef \bar \eps^{-3}\varrho(\bar \eps^{-2} t, \bar \eps^{-1} x)$ be its rescaled version. Set  $\xi^{\eps,\bar \eps}\eqdef \xi^\eps\ast_\eps\varrho_{\bar \eps}$ and define $Y^{\eps}_{\bar \eps}$ as in~\eqref{e:DXWienerIntegral} but with the kernel $\CW^\eps$ replaced by $\CW^{\eps,\,\bar\eps}$. 
For $z_1\neq z_2\in\Lambda_{\eps}^\s$, we also introduce
\begin{equ}[e:DKDef]
 \CK^\eps(z_1, z_2) \eqdef \langle \mathcal{W}^\eps(z_1), \mathcal{W}^\eps(z_2) \rangle_{\eps}\;, \qquad \delta \CK^{\eps,\bar\eps}(z_1, z_2) \eqdef \langle \delta \mathcal{W}^{\eps,\bar \eps}(z_1), \delta \mathcal{W}^{\eps,\bar\eps}(z_2) \rangle_{\eps}\;,
\end{equ}
where $\delta \mathcal{W}^{\eps,\bar\eps} \eqdef \mathcal{W}^{\eps,\bar \eps} - \mathcal{W}^\eps$.  In all the cases we will consider these functions can be written as $\CK^\eps(z_1, z_2) = \CK^\eps(z_1 - z_2)$, where $\CK^\eps$ is a kernel on $\Lambda_\eps^\s$ and $\{\CK^\eps\}_\eps$ is a family of discrete singular kernels according to Definition~\ref{d:DSingKer}. 
Thanks to the notations introduced above, we now have everything we need to state and prove the following proposition, representing the discrete counterpart of Proposition~\ref{p:ChaosKernelsConv}.

\begin{proposition}\label{p:DChaosKernelsConv}
In the setting described above, for all $N\in\N$, $\eps\eqdef 2^{-N}$ and $\bar\eps\geq\eps$, let $\CK^\eps$ and $\delta\CK^{\eps,\bar\eps}$ be the compactly supported kernels on $\Lambda_\eps^\s$ defined in~\eqref{e:DKDef} for distributions $Y^\eps$ and $Y^{\eps,\bar\eps}$ of the form~\eqref{e:DXWienerIntegral}. Then the following results hold
\begin{enumerate}
\item if $\{\CK^\eps\}_\eps$ is a family singular of order $\zeta \in (-3, -1]$, then for any $\alpha< \frac{\zeta}{2}$, $Y^\eps\in\CC_1^{\alpha,\s,\eps}$ almost surely. If furthermore the family $\delta \mathcal{K}^{\eps,\bar\eps}$ has a singularity of order $\zeta - 2\theta$, for some $\theta > 0$, and the quantity in~\eqref{def:DSingKer} is bounded by $\bar\eps^{2\theta}$, then for any $p\geq 1$ and $\alpha < \frac{\zeta}{2}$, the following bound holds uniformly in $\eps$:
\begin{equ}[e:DKEpsConverge]
 \mathbb{E} \left[\left(\Vert Y^\eps - Y^{\eps,\bar\eps} \Vert^{(\eps)}_{\CC_1^{\alpha - \theta,\s}}\right)^p\right] \lesssim \bar\eps^{\theta p}\;.
\end{equ} 
\item\label{item:Dthird} If $\{\mathcal{K}^\eps\}_\eps$ is a family singular of order $\zeta \in (-1, \infty) \setminus \N$, then the process $Y^\eps\in\CC^{\beta/2,\eps} ([-1,1]\cap\Lambda_{\eps^2}, \CC^{\alpha - \beta,\eps})$, for any $\alpha < \frac{\zeta}{2}$ and any $\beta\in(0,\alpha)$ if $\zeta>0$ and any $\beta\in\big(0, \frac{\zeta + 1}{2}\big)$ if $\zeta<0$. Moreover, if the family $\delta \mathcal{K}^{\eps,\bar\eps}$ is singular of order $\zeta - 2\theta$, for some $\theta \in \big(0, \frac{\zeta+1}{2}\big)$, and the quantity in~\eqref{def:DSingKer} is bounded by $\bar\eps^{2\theta}$, then, for any $p \geq 1$ and $\alpha < \frac{\zeta}{2}$, the following bound holds uniformly in $\eps$:
\begin{equ}[e:DKEpsConvergeTime]
 \mathbb{E} \left[\left(\Vert Y^\eps - Y^{\eps,\bar\eps} \Vert^{(\eps)}_{\CC^{\beta/2}([-1,1]\cap\Lambda_{\eps^2}, \CC^{\alpha - \beta - \theta})}\right)^p\right] \lesssim \bar\eps^{\theta p}\;,
\end{equ}
where $\beta > 0$ is such that $\beta + \theta\in(0,\alpha)$ if $\zeta>0$ and $\beta+\theta\in\big(0, \frac{\zeta + 1}{2}\big)$ if $\zeta<0$.
\end{enumerate}
\end{proposition}

In the proof of the previous proposition we will need the discrete counterpart of Lemma~\ref{l:KDeltaz}. For any two time points $s, t \in \Lambda_\eps$, we define the operator 
\begin{equ}[e:DDeltaz]
\delta_{s, t}^2 \CK^\eps(x) \eqdef \sum_{\epsilon_1, \epsilon_2 \in \{0, 1\}} (-1)^{\epsilon_1 - \epsilon_2} \CK^\eps\big((\epsilon_1 - \epsilon_2)(t - s), x\big)\;,
\end{equ}
and then we have
\begin{lemma}
\label{l:DKDeltaz}
Let the family $\{\CK^\eps\}_\eps$ have a singularity of order $\zeta \leq 0$. Then, for any $\beta \in [0,2]$ and any points $s, t \in \Lambda_\eps$ and $x \neq 0$, the following bound holds uniformly in $\eps$:
\begin{equs}
 |\delta^2_{s, t} \CK^\eps(x)| \leq C | s- t |^{\beta} |x|^{\zeta - 2 \beta}\;.
\end{equs} 
\end{lemma}
\begin{proof}
The proof is identical to the of Lemma~\ref{l:KDeltaz}. The only difference being that in case $|t-s|<|x|^2$, one has to rewrite $\delta_{s, t}^2 \CK^\eps(x)$ as a double telescopic sum
\[
\delta_{s, t}^2 \CK^\eps(x)=\eps^4\sum_{ s_1=s}^t\sum_{ s_2=s}^t \bar D_{s_1,\eps^2}\bar D_{s_2,\eps^2}\CK^\eps(s_1-s_2, x)
\]
and then proceed as in the above mentioned proof.
\end{proof}

\begin{proof}[Proof of Proposition~\ref{p:DChaosKernelsConv}]
Thanks to the characterization of the discrete H\"older spaces through the discrete wavelets given in Proposition~\ref{prop:DCalpha} and the lemma above, it is immediate to see that in order to prove the result it suffices to follow the same steps and use the same arguments as in the proof of Proposition~\ref{p:ChaosKernelsConv}, replacing the integrals with the discrete pairing and exploiting the analogous bounds.
\end{proof}

\begin{corollary}\label{c:DKConvol}
In the settings of Proposition~\ref{p:DChaosKernelsConv} with $\zeta \in (-3, 0)$, let a compactly supported family of functions $\{\CK^\eps\}_\eps$ be singular of order $\chi \in (-\frac{7 + \zeta}{2}, 0)$. Then the family of distributions $\CK^\eps \ast_\eps Y^\eps$ belongs almost surely to the space $\CC^{\beta/2,\eps} ([-1,1]\cap\Lambda_{\eps^2}, \CC^{\alpha - \beta,\eps})$ for any $\alpha < \frac{\zeta}{2} + \chi + 3$ and $\beta$ as in the second part of the statement of the above-mentioned proposition with $\zeta + 2\chi + 6$ in place of $\zeta$.
\end{corollary}

\begin{proof} 
Analogous to the one of Corollary~\ref{c:KConvol}, applying though Lemma~\ref{lemma:OpDSingKer1} and Proposition~\ref{p:DChaosKernelsConv} in place of~\cite[Lem.~10.14]{Hai14} and Proposition~\ref{p:ChaosKernelsConv}.
\end{proof}

We conclude this section with the following lemma, whose proof is identical to that of Lemma~\ref{l:Control}, which allows to bound the stochastic processes defined  in~\eqref{e:DRemainders}. 

\begin{lemma}\label{l:DControl}
Let us be given a family of functions (parametrized by $\eps$) $\Lambda_\eps \ni x \mapsto R^\eps_x$ such that $R^\eps_x$ is of the form \eqref{e:DXWienerIntegral} with the respective function $\mathcal{K}^\eps_x$, defined in \eqref{e:DKDef}, and, for some $\zeta \in (-1, 0)$, satisfies
\begin{equ}[eDControlBound]
|\mathcal{K}^\eps_x(z_1, z_2)| \lesssim \sum_{\delta \geq 0} |x_1 - x_2|^{\zeta - \delta} \left( |x_1 - x_2|^{\delta} + |x_1 - x|^{\delta} + |x_2 - x|^{\delta} \right),
\end{equ}
uniformly over the variables $z_i = (t, x_i)$, and where the sum runs over finitely many values of $\delta \in [0, 1 + \zeta)$. Let furthermore $R^\eps_x$ have a strictly positive (discrete) H\"{o}lder regularity in time. Then $R^\eps$ belongs almost surely to the space $\CC^\eps([-1,1], \CL^{\beta, \eps})$ for any $\beta < {\zeta \over 2}$.
\end{lemma}

Now that we derived the discrete version of all the results we needed in the proof of Proposition~\ref{p:Control}, we are ready for the following. 

\begin{proof}[Proof of Proposition~\ref{p:DControl}]
Let $\{\xi^\eps(z)\}_{z\in\Lambda_\eps^\s}$ be a family of i.i.d. normal random variables with mean $0$ and variance $\eps^{-3}$ and set, for $\bar\eps\geq\eps$, $\xi^{\eps}_{\bar \eps}\eqdef \xi^\eps\ast_\eps\varrho_{\bar\eps}$, where $\varrho$ is a symmetric compactly supported smooth function integrating to 1 and $\varrho_{\bar \eps}(t,x)\eqdef \bar \eps^{-3}\varrho(\bar \eps^{-2} t, \bar \eps^{-1} x)$ its rescaled version. 
Notice that the definition of the discrete stochastic processes in $\X^\eps(\xi_{\bar\eps}^\eps)$ (or in $\X^\eps(\xi^\eps)$) can be obtained by that of the ones in $\X(\xi^\eps)$ (or $\X(\xi)$) simply by replacing every occurrence of $K$ with $K^\eps$, $K_\eps$ with $K^\eps_{\bar\eps}\eqdef K^\eps\ast_\eps \varrho_{\bar\eps}$, $\partial_x$ with $D_{x,\eps}$ and the pointwise product with $B_\eps$. Moreover, all the estimates involving these quantities, as well as those we exploited in the proof of Proposition~\ref{p:Control} are available in this discrete setting (Section~\ref{sec:DKernels}, Proposition~\ref{p:DChaosKernelsConv} and Corollary~\ref{c:DKConvol}) and are uniform in $\eps$. This ensures that the validity of Proposition~\ref{p:DControl} can be shown by retracing the steps followed in the proof of Proposition~\ref{p:Control}. 

The only difference lies in the constants we need in order to ``renormalize" two objects (the other constants coming from the $0$-th chaos component of the intermediate objects we introduced in the other proof disappear also in this case because of the presence the discrete spatial derivative). More precisely, consider $X^{\<tree2>,\,\eps}$ and $X^{\<tree21>,\,\eps}$, and let us look at the definition of $C^{\<tree2>,\,\eps}$ and $C^{\<tree21>,\,\eps}$. Similarly to~\eqref{e:Constant1}, after Wick's contraction we have\footnote{In the definition of the constants, for convenience, we will use the full discrete heat kernel, and not the $K^\eps$ introduced in Lemma~\ref{l:DKernelDec} but, since they agree in a neighbourhood of $0$, this does not change anything.}
\begin{equ}
C^{\<tree2>,\,\eps}=\E\left[B_\eps(X^{\<tree1>,\,\eps},X^{\<tree1>,\,\eps})(0)\right]=\eps^2\sum_{s\in\Lambda_{\eps^2}}\eps \sum_{x\in\Lambda_\eps}B_\eps(D_{x,\eps}P^\eps(s,\cdot),D_{x,\eps}P^\eps(s,\cdot))(x)\;,
\end{equ}
where we split the space and time convolution for reasons that will be clear later. We now apply~\eqref{e:DFTProd} to the right hand side of the previous so that it becomes
\begin{equs}
\eps^2\sum_{s\in\Lambda_{\eps^2}^+}\int_{-\frac{1}{2\eps}}^{\frac{1}{2\eps}}\frac{\hat\pi(\eps k)\hat\pi(-\eps k)}{\eps^2}&\left(1+\frac{\hat\nu(\eps k)}{2\bar\nu}\right)^{2s\eps^{-2}}\hat\mu(-\eps k ,\eps k)\dd k\\
&=\eps^{-1}\sum_{n\in\N}\int_{-\frac{1}{2}}^{\frac{1}{2}}\hat\pi( k)\hat\pi(- k)\left(1+\frac{\hat\nu( k)}{2\bar\nu}\right)^{2n}\hat\mu(- k , k)\dd k\;,
\end{equs}
where $\Lambda_{\eps^2}^+\eqdef\Lambda_{\eps^2}\cap[0,+\infty)$ and the equality is obtained via change of variables. Notice that, for $k\neq 0$, we can write
\begin{equ}
\hat\pi( k)\hat\pi(- k)\sum_{n\in\N}\left(1+\frac{\hat\nu( k)}{2\bar\nu}\right)^{2n} = g(k)g(-k)\frac{4\bar \nu^2}{f(k)\big(4\bar\nu+\hat\nu(k)\big)}\;,
\end{equ}
where we have set $g(k)\eqdef \hat\pi( k)/(i k)$ and, as in the proof of Lemma~\ref{lemma:DHeatKernels}, $f(k)\eqdef -\hat\nu( k)/k^2$. But now, thanks to Assumption~\ref{a:nu}, we know that there exists a constant $c>0$ such that, for $k\in[-1/2,1/2]$, $f(k)>c$ hence we can exchange sum and integral, and get
\begin{equ}\label{e:DConstant1}
C^{\<tree2>,\,\eps}=\eps^{-1}\int_{-\frac{1}{2}}^{\frac{1}{2}}g(k)g(-k)\frac{4\bar \nu^2}{f(k)\big(4\bar\nu+\hat\nu(k)\big)}\hat\mu(- k , k)\dd k\;,
\end{equ}
where the last integral is of course finite. Let us now turn to the second constant. Acting as before and similar to~\eqref{e:Constant2}, we have
\begin{equ}
C^{\<tree21>,\,\eps}=\eps^2\sum_{s\in\Lambda_{\eps^2}}\eps \sum_{x\in\Lambda_\eps}B_\eps(B_\eps(1,D_{x,\eps}P^\eps(s,\cdot)),D_{x,\eps}P^\eps\ast_\eps D_{x,\eps}P^\eps(s,\cdot))(x)\;.
\end{equ}
Thanks to~\eqref{e:DFTProd} and~\eqref{e:DFTConv}, the previous coincides with
\begin{equs}
\eps^2\sum_{s\in\Lambda_{\eps^2}^+} s \int_{-\frac{1}{2\eps}}^{\frac{1}{2\eps}}\frac{\hat\pi(-\eps k)^2\hat\pi(\eps k)}{\eps^3} &\left(1+\frac{\hat\nu(\eps k)}{2\bar\nu}\right)^{2s\eps^{-2}} \hat\mu(-\eps k ,0)\hat\mu(-\eps k ,\eps k)\dd k\\
&=\sum_{n\in\N} n\int_{-\frac{1}{2}}^{\frac{1}{2}}\hat\pi(- k)^2 \hat\pi( k) \left(1+\frac{\hat\nu( k)}{2\bar\nu}\right)^{2n} \hat\mu(- k , 0) \hat\mu(- k , k)\dd k\;.
\end{equs}
Now, notice that for $x\in(0,1)$, we have
\begin{equ}
\sum_{n\in\N}n(1-x)^{2n}=\frac{1}{2}(1-x)\sum_{n\in\N}2n(1-x)^{2n-1}=\frac{1}{2}(1-x)\partial_x\sum_{n\in\N}(1-x)^{2n}=\frac{(1-x)^2}{x^2(2-x)^2}\;,
\end{equ}
so that, using the notations introduced above, the integrand can be rewritten as
\begin{equ}
\hat\pi(- k)^2 \hat\pi( k)\sum_{n\in\N} n \left(1+\frac{\hat\nu( k)}{2\bar\nu}\right)^{2n} = \frac{g(k)}{k}g(-k)^2 \frac{4\bar\nu^2(2\bar\nu+\hat\nu(k))^2}{f(k)^2(4\bar\nu+\hat\nu(k))^2}=i\frac{g(-k)}{k}|g(k)|^2 \frac{4\bar\nu^2(2\bar\nu+\hat\nu(k))^2}{f(k)^2(4\bar\nu+\hat\nu(k))^2}\;.
\end{equ}
where the last equality is due to the fact that the complex conjugate of $g(k)$ is $g(-k)$. Now, arguing as before, apart from the first, all the other factors at the right hand side are clearly bounded uniformly in $k\in[-\frac{1}{2},\frac{1}{2}]$. For the first instead it suffices to notice that $g$, by Assumption~\ref{a:pi} is differentiable and such that $g(0)=0$, which in particular implies that $g(k)/k$ is a bounded continuous function. Therefore we conclude that
\begin{equs}\label{e:DConstant2}
C^{\<tree21>,\,\eps} &= i\int_{-\frac{1}{2}}^{\frac{1}{2}}\frac{g(-k)\hat\mu(- k , 0)}{k}\hat\mu(- k , 0) |g(k)|^2 \frac{4\bar\nu^2(2\bar\nu+\hat\nu(k))^2}{f(k)^2(4\bar\nu+\hat\nu(k))^2} \hat\mu(- k , k)\dd k\\
&=-\int_{-\frac{1}{2}}^{\frac{1}{2}}\frac{\Imma(g(-k)\hat\mu(- k , 0))}{k} |g(k)|^2 \frac{4\bar\nu^2(2\bar\nu+\hat\nu(k))^2}{f(k)^2(4\bar\nu+\hat\nu(k))^2} \hat\mu(- k , k)\dd k
\end{equs}
where $\Imma$ denotes the imaginary part and we used parity in order to pass from the first to the second line (see also~\cite[Lemma 10.2]{Reloaded} where a similar computation is carried out). Notice that the constant is indeed finite and independent of $\eps$.
\end{proof}

\section{Convergence of discrete solutions}\label{sec:Conv}

In this section we want to prove the main result of this paper, namely Theorem~\ref{t:Convergence}, which states that, upon choosing the sequence of constants $\BC^\eps$ as in Proposition~\ref{p:DControl}, the sequence of solutions to our family of discrete models converge to the solution of the renormalized SBE. 

\begin{proof}[Proof of Theorem~\ref{t:Convergence}]
Given the stability results obtained in Sections~\ref{sec:Analytic} and~\ref{sec:Stochastic}, in order to prove the statement we will exploit a diagonal argument similar to the one, used in analogous contexts, in~\cite{HM15,HQ15, HS15} and many others. 
To be more specific, let $\psi$ be a smooth, symmetric  compactly supported function on $\R^2$ which integrates to $1$ and, for some $\bar\eps\geq \eps$, define $\xi_{\bar\eps}\eqdef \xi\ast\psi^{\bar\eps}$, where $\psi^{\bar\eps}(t,x)=\bar\eps^{-3}\psi(t/\bar\eps^2,x/\bar\eps)$. Let $u_{\bar\eps}$ be the unique smooth solution to 
\begin{equ}\label{e:SBEfinal}
\partial_t u_{\bar\eps}=\Delta u_{\bar\eps} +\partial_x (u_{\bar\eps})^2 +C\partial_x u_{\bar\eps} +\partial_x\xi_{\bar\eps}\,,\qquad u_{\bar\eps}(0,\cdot)=u_0(\cdot)\;,
\end{equ}
and define the family of controlling processes $\X(\xi_{\bar\eps};\BC_{\bar\eps})\eqdef \X(\xi_{\bar\eps}, C_{\bar\eps}^{\<tree2>},\,C_{\bar\eps}^{\<tree21>})$ as in the proof of Proposition~\ref{p:Control}, where $ C_{\bar\eps}^{\<tree2>}$ is as in~\eqref{e:Constant1} and $C_{\bar\eps}^{\<tree21>}$ is given by \eqref{e:DConstant2}. Let $\X(\xi;\BC)$ be the enhancement of white noise determined in the just mentioned proposition, in which though we subtract to the term $X^{\<tree21>}$ the constant $C$. 
Thanks to Theorem~\ref{t:FixedPoint}, we know that the solution map $\CS$, assigning to any $(u_0,\X)\in\CC^\eta\times\CX$ the solution to~\ref{e:SBEfinal} is jointly locally Lipschitz continuous, which in particular implies that if $T_\infty$ is the explosion time for $u\eqdef \CS(u_0,\X(\xi;\BC))$ then, for all $T<T_\infty$, we have
\begin{equ}
\E\left[\|u;u_{\bar\eps}\|_{\CC^{\alpha}_{\eta,T}}^p\right]\lesssim \E\left[\|\X(\xi;\BC);\X(\xi_{\bar\eps};\BC_{\bar\eps})\|_{\CX}^p\right]
\end{equ}
and, by Proposition~\ref{p:Control}, for all $p\geq1$, the right hand side converges to $0$ uniformly in $\eps$. 
In order to discretize the noise $\xi_{\bar\eps}$, for $\bar\eps\geq\eps$, we set 
\begin{equ}
\xi_{\bar\eps}^\eps(z)\eqdef\psi_{\bar\eps}\ast_\eps\xi^\eps(z)\;,
\end{equ}
for $z\in\Lambda_{\eps^2,T}\times\T_\eps$, where we recall that $\xi^\eps(z)\eqdef \eps^{-3} \langle \xi, \1_{|\eps^{-\s}(\cdot-z)| \leq 1/2}\rangle$. Let $u_{\bar\eps}^\eps$ and $u^\eps$ be the unique solutions to~\eqref{e:DiscreteSBE} both with initial condition $u_0^\eps$ but with respect to the noises $\xi_{\bar\eps}^\eps$ and $\xi^\eps$, and with the respective renormalization constants $C^\eps_{\bar \eps}$ and $C^\eps$. 
Let $\X^\eps(\xi^\eps_{\bar\eps}, \BC^\eps_{\bar\eps})\eqdef\X^\eps(\xi_{\bar\eps}^\eps, C_{\bar\eps}^{\<tree2>,\,\eps},\,C_{\bar\eps}^{\<tree21>,\,\eps})$ and $\X^\eps(\xi^\eps, \BC^\eps)\eqdef\X^\eps(\xi^\eps, C^{\<tree2>,\,\eps},\,C^{\<tree21>,\,\eps})$ be the discrete enhancement of $\xi^\eps_{\bar\eps}$ and $\xi^\eps$ respectively, where, in each case, the constants are as in~\eqref{e:DConstant1} and~\eqref{e:DConstant2} respectively. 
Arguing as before but applying Theorem~\ref{t:DFixedPoint} and the (uniform) local Lipschitz continuity of the discrete solution map $\CS^\eps$, we have that, for all $T< T^d_\infty\wedge T_\infty$, 
\begin{equ}
\E\left[\Big(\|u^\eps;u_{\bar\eps}^\eps\|_{\CC^{\alpha}_{\eta,T}}^{(\eps)}\Big)^p\right]\lesssim \E\left[\Big(\|\X^\eps(\xi^\eps;\BC^\eps);\X^\eps(\xi_{\bar\eps}^\eps;\BC_{\bar\eps}^\eps)\|_{\CX}^{(\eps)}\Big)^p\right]
\end{equ}
where $T_\infty^d$ is the stopping time determined in Theorem~\ref{t:DFixedPoint} and which is independent of $\eps$, and, by Proposition~\ref{p:DControl}, for all $p\geq1$, the latter converges to $0$ uniformly in $\eps$. Hence, we obtain, for all $T< T^d_\infty\wedge T_\infty$, 
\begin{equ}
\E\left[\Big(\|u;u^\eps\|_{\CC^{\alpha}_{\eta,T}}^{(\eps)}\Big)^p\right]\lesssim \E\left[\|u;u_{\bar\eps}\|_{\CC^{\alpha}_{\eta,T}}^p\right] + \E\left[\Big(\|u_{\bar\eps};u_{\bar\eps}^\eps\|_{\CC^{\alpha}_{\eta,T}}^{(\eps)}\Big)^p\right]+\E\left[\Big(\|u^\eps;u_{\bar\eps}^\eps\|_{\CC^{\alpha}_{\eta,T}}^{(\eps)}\Big)^p\right]\;,
\end{equ}
and taking the limit $\eps\to0$ and then $\bar\eps\to 0$, the above discussion guarantees that the first and the last summand at the right hand side converge to $0$.  For the second summand instead, it suffices to notice that, for $\bar\eps$ fixed, $u_{\bar\eps}$ is the solution of a parabolic semilinear PDE driven by the smooth noise $\xi_{\bar\eps}$ (hence, smooth itself) therefore the convergence of its discretization is ensured by the convergence of the discrete noise (see~\cite{BS08} and references therein). But now, by our definitions we have that
\begin{equ}
\xi_{\bar\eps}=\psi_{\bar\eps}\ast\xi\qquad\text{and}\qquad \xi_{\bar\eps}^\eps=\psi_{\bar\eps}^\eps\ast\xi\;,
\end{equ}
where we set $\psi_{\bar\eps}^\eps(z)\eqdef \eps^{-3} \langle \psi_{\bar\eps}, \1_{|\cdot-z|_\s\leq \eps/2}\rangle$. It is easy to see that
\[
|\bar D^k(\psi_{\bar\eps}(z)-\psi_{\bar\eps}^\eps(z))|\lesssim \eps\,\bar\eps^{-|\s|-|k|_\s-1}\;,
\]
for all $k\in\N^2$ and uniformly in $z\in\Lambda_\eps^\s$, which implies the required convergence. 
\end{proof}

\appendix

\section{Elements of wavelet analysis and its discrete counterpart}\label{Appen:Wavelet}

In this appendix we want to present the basic elements of wavelet analysis we need in the definition of the stochastic terms appearing in the description of the solution to our equation and its discrete version. We make no claim of exhaustiveness and the interested reader is addressed to the relevant literature, namely~\cite{Dau},~\cite{Mey92}. In the discrete case, we recall the construction in~\cite{HM15} and explicitly state and prove the results necessary in our context. 
\newline

A multiresolution analysis of $L^2(\R)$ (see~\cite[Def. 1, Sec. 2]{Mey92}) is a sequence $\{V_n\}_{n\in\Z}$ of closed subsets of $L^2(\R)$ such that 
\begin{itemize}[noitemsep,,]
\item $V_n\subset V_{n+1}$ for all $n\in\Z$, the intersection of all $V_n$'s is $\{0\}$ and their union dense in $L^2(\R)$
\item for all $n\in\Z$, $f\in V_0$ if and only if $f(2^n\cdot)\in V_n$
\item there exists a function $\varphi\in L^2(\R)$ (called \textit{father wavelet}) such that $\{\varphi(\cdot-k)\}_{k\in\Z}$ is an orthonormal basis of $V_0$. 
\end{itemize}

We now let $\Lambda_n\eqdef 2^{-n}\Z$ and define $\varphi_x^n=2^{n/2}\varphi(2^n(\cdot-x))$, the rescaled version of the father wavelet that preserves its $L^2$-norm, then it is immediate to see that, for every $n\in\Z$, the family $\{\varphi_x^n\}_{x\in\Lambda_n}$ forms an orthonormal basis of $V_n$. 

It is even possible to understand whether a given function $\varphi$ can be used as father wavelet for a multiresolution analysis. Indeed, if $\varphi\in L^2(\R)$ is such that 
\begin{enumerate}[noitemsep,,]
\item $\langle \varphi\,,\,\varphi_k\rangle=0$ for all $k\in \Z$ ($\varphi_k(\cdot)\eqdef \varphi(\cdot-k)$),
\item there exists a sequence $\{a_k\}_{k\in\Z}\in \ell^2(\Z)$ for which $\varphi(\cdot)=\sum_k a_k \varphi_k(2\cdot)$
\item for some $r\geq 0$, $\varphi$ is a compactly supported $\CC^r$ function integrating to 1,
\end{enumerate}
then defining  $V_n$, for $n\in\Z$, as the closure of span$\{\varphi_x^n\}_{x\in\Lambda_n}$, we have that $\{V_n\}_{n\in\Z}$ forms a multiresolution analysis of $L^2(\R)$. 
In the following proposition we collect (without proof) other important results concerning wavelets. 

\begin{proposition}\label{prop:CWave}
For every $r\geq 0$ there exists a function $\varphi\in L^2(\R)$ such that properties 1,2 and 3 stated above hold and, in 2, the number of non-zero $a_k$'s is finite. Moreover, take $n\in\Z$ and let $W_n$ be the orthogonal complement of $V_n$ in $V_{n+1}$. Then there exists a finite set of coefficients $\{b_k\}_{k\in\CK}$, $\CK\subset \Z$, such that, upon setting $\psi(x)=\sum_{k\in\CK} b_k \varphi(2x-k)$, the family $\{\psi_x^n\}_{x\in\Lambda_n}$ forms an orthonormal basis of $W_n$, where $\psi_x^n=2^{n/2}\psi(2^n(\cdot-x))$. As a consequence, for any $n\in\Z$, 
\begin{equation}\label{def:CWaveBasis}
\{\varphi_x^n\,:\,x\in\Lambda_n\}\cup\{\psi_x^m\,:\,m\geq n,\,x\in\Lambda_m\}
\end{equation}
is an orthonormal basis of the whole $L^2(\R)$. At last, $\psi$ is such that, for any $k\in\N$, $k\leq r$ 
\begin{equation}\label{eq:WavePol}
\int_\R \psi(x) x^k\dd x=0\;.
\end{equation}
We will refer to the function $\psi$ introduced above as ``mother wavelet".
\end{proposition}

Even if the proposition above is stated for wavelets on $\R$, it can be easily generalized to $\R^d$ with general scaling (for us it will be the parabolic scaling in which time counts twice), following the procedure outlined in Section 3.1 of~\cite{Hai14} and that we here briefly recall.  
Let $\s=(s_1,\cdots,s_d)\in\N^d$ be a scaling and set $\Lambda_n^\s\eqdef \{ 2^{-s_in}k_i e_i\,:\,i=1,...,d\, k_i\in\Z\}$, where $e_1,\,...,\,e_d$ is the canonical basis of $\R^d$. Then it suffices to define $\varphi_x^{n,\s}$, $x=(x_i)_{i=1,...,d}\in\Lambda_n^\s$, as the product of $\varphi_{x_i}^{s_i n}$, and it can be proved that there exists a finite collection, $\Psi$, of functions $\psi$, satisfying the properties of the mother wavelet, such that $\{\psi_x^{n,\s}\}_{x\in\Lambda_n^\s}$ generates $W_n$, where $\psi_x^{n,\s}(\cdot)=2^{n|\s|/2}\psi(2^{n\s}(\cdot-x))$. 

Moreover, it is possible to characterize the spaces $\CC^{\alpha,\s}$ using the family of wavelets we just introduced. Let $\alpha\in\R$ and $r>|\alpha|$, then, Proposition 3.20 in~\cite{Hai14} states that $\xi\in\CC^{\alpha,\s}$ if and only if for every compact set $\FK$ in $\R^d$, the bounds
\[
\langle \xi,\varphi_y^{n,\s}\rangle\lesssim 2^{-\frac{n|\s|}{2}-n\alpha}\qquad \text{and}\qquad \langle \xi,\psi_x^{m,\s}\rangle\lesssim 2^{-\frac{m|\s|}{2}-m\alpha}
\]
hold for $n\in\N$, $m\geq n$ and $y\in\Lambda_n^\s\cap\FK$, $x\in\Lambda_m^\s\cap\FK$. 
\newline

We now would like to define a set of functions that have similar properties to the ones of mother and father wavelet but can be used to give a characterization of H\"older functions/distributions on the grid $\Lambda_\eps^\s$, for $\eps\eqdef 2^{-N}$ and $N\in\N$, which, in a sense, is uniform in $\eps$. The construction carried out in Section 4.1.2 of~\cite{HM15} serves, among the various, also this purpose. 

Let $r\geq 0$ and $\varphi\in\CC^{r,\s}(\R^d)$ be the father wavelet of a multiresolution analysis in $L^2(\R^d)$ as given in Proposition~\eqref{prop:CWave} and the construction just below, and $\psi\in\Psi$ be the finite family of mother wavelet associated to it. 
For $n\in\N$, $n\leq N$ define the functions
\begin{equation}\label{def:DWave}
\varphi_x^{N,n,\s}(\cdot)\eqdef 2^{\frac{N|\s|}{2}}\langle \varphi_\cdot^{N,\s},\varphi_x^{n,\s}\rangle\qquad\text{and}\qquad \psi_x^{N,n,\s}(\cdot)\eqdef2^{\frac{N|\s|}{2}}\langle \varphi_\cdot^{N,\s},\psi_x^{n,\s}\rangle
\end{equation}
for $\psi\in\Psi$, $x\in\Lambda_n^\s$ and where $\langle\cdot,\,\cdot\rangle$ is the usual $L^2(\R^d)$-pairing. By~\eqref{def:DWave}, it is easy to see that the functions $\varphi^{N,n,\s}$ and $\psi^{N,n,\s}$ inherit many of the properties of $\varphi^{n,\s}$ and $\psi^{n,\s}$, for example, as functions on $\R^d$, they are supported in a ball of radius $\CO(2^{-n|\s|})$, belong to $\CC^{r,\s}$, and possess the same scaling features. Moreover, they also play essentially the same role in the description of discrete functions, in the sense that, for a function $f$ on $\Lambda_\eps^\s$, $\varphi^{N,n,\s}$ allows to look at it at scales $2^{-n}$ while the $\psi^{N,m,\s}$'s improve our knowledge of $f$ at finer scales.
The following proposition makes this idea more precise and provides the discrete counterpart of Proposition 3.20 in~\cite{Hai14}. 
Before stating it, we recall that we say that a family of maps $f^\eps$, $f^\eps$ on $\Lambda_\eps^\s$, belongs to $\CC^{\alpha,\s,\eps}$ if $\|f^\eps\|_{\CC^{\alpha,\s,\eps}}^{(\eps)}\leq M$ and $M$ is independent of $\eps$. 

\begin{proposition}\label{prop:DCalpha}
Let $\alpha\in\R$, $\eps=2^{-N}$ and $\xi$ be a function on the grid $\Lambda_\eps^\s$. Let $r\geq-|\alpha|$, $\varphi,\in\CC^{r,\s}$ be a father wavelet and $\psi\in\Psi$ be the finite family of mother wavelets associated to it. Then, $\xi^\eps\in\CC^{\alpha,\s,\eps}$ if and only if for every compact set $\FK$ in $\R^d$, the bounds
\begin{equation}\label{b:DCalpha}
\langle \xi^\eps,\varphi_y^{N,0,\s}\rangle_\eps\lesssim 1\qquad \text{and}\qquad \langle \xi^\eps,\psi_x^{N,n,\s}\rangle_\eps\lesssim 2^{-\frac{n|\s|}{2}-n\alpha}
\end{equation}
hold for $n\in\N$, $0\leq n\leq N$ and all $y\in\Lambda_0^\s\cap\FK$, $x\in\Lambda_n^\s\cap\FK$, uniformly in $\eps$.
\end{proposition}
\begin{proof}
We are going to prove the result for $\alpha<0$ and the proof is very similar to the one of Proposition 3.20 in~\cite{Hai14}. 
Of course, if $\xi^\eps\in\CC^{\alpha,\s,\eps}$ according to~\eqref{e:DParabolicHolderNegative}, then the bounds above are satisfied by definition (recall that the scaling of mother and father wavelet preserve the $L^2$-norm, not the $L^1$). Assume then that~\eqref{b:DCalpha} hold. 
Notice that, by definition~\eqref{def:DWave} and the orthonormality of the elements in~\eqref{def:CWaveBasis}, we have
\begin{equation}\label{eq:BWave}
\varphi_x^{N,N,\s}(y)=2^{\frac{N|\s|}{2}}\langle \varphi_y^{N,\s},\varphi_x^{N,\s}\rangle_\eps=2^{\frac{N|\s|}{2}}\delta_{x,y}\,,\qquad x,\,y\in\Lambda_\eps^{\s}\;.
\end{equation}
Moreover, Proposition~\ref{prop:CWave}, guarantees that $V_N=V_0\oplus\bigoplus_{n=0}^{N-1} W_n$ which in particular implies
\[
\varphi_x^{N,\s}=2^{-\frac{N|\s|}{2}}\sum_{z\in\Lambda_0^\s}  \varphi_z^{N,0,\s}(x)\varphi_z^{0,\s}+2^{-\frac{N|\s|}{2}}\sum_{\psi,n,z}\psi_z^{N,n,\s}(x)\psi_w^{n,\s}\;,
\]
where in the second sum $\psi, \,n$ and $z$ belong respectively to $\Psi$, $\{0,...,N-1\}$ and $\Lambda_n^\s$. Hence, since $\xi^\eps(\cdot)=2^{-\frac{N|\s|}{2}}\langle \xi^\eps,\varphi_\cdot^{N,N,\s}\rangle_\eps$, we can rewrite $\xi^\eps$ as
\[
\xi^\eps(\cdot) =2^{-\frac{N|\s|}{2}}\sum_{z\in\Lambda_0^\s} \langle \xi^\eps,\varphi_z^{N,0,\s}\rangle_\eps \varphi_z^{N,0,\s}(\cdot)+2^{-\frac{N|\s|}{2}}\sum_{\psi,n,z}\langle \xi^\eps,\psi_z^{N,n,\s}\rangle_\eps\psi_z^{N,n,\s}(\cdot)\;.
\]
Now, let $\eta\in\CB_0^r(\R^d)$, $\lambda\in[\eps,1]$, $x\in\Lambda_\eps^\s$ and consider $\langle \xi^\eps\,,\eta_x^\lambda\rangle_\eps$, whose expression can be immediately deduced by the one for $\xi^\eps$ given above. We will bound each of the two summands separately, beginning with the second. 
Split the sum in two and consider first the one on those $n$ such that $2^{-n|\s|}\leq \lambda$. Notice that 
\begin{equ}
\langle\psi_z^{N,n,\s},\eta_x^\lambda\rangle_\eps= 2^{\frac{N|\s|}{2}} \eps^{|\s|} \sum_{y\in\Lambda_\eps^\s} \eta^\lambda_x(y)\int_{\R^d} \varphi_y^{N,\s}(w)\phi_z^{n,\s}(w)\dd w=\int \varphi^{0,\s}(w)\lambda^{-1} \eps^{|\s|}\sum_{y\in\Lambda_\eps^\s}\eta(y/\lambda)\psi^{n,\s}_z(2^{-N} w+x-y)\dd w
\end{equ}
and upon replacing $\eta$ by its Taylor's expansion around $0$ we can rewrite the inner sum as
\[
\lambda^{-1}\sum_{|k|_\s=0}^{r-1}\frac{D^k \eta(0)}{k!}\sum_{y\in\Lambda_\eps^\s}(y/\lambda)^k\psi^{n,\s}_z(2^{-N} w+x-y)+\lambda^{-1-r}\sum_{y\in\Lambda_\eps^\s}\tilde R(\eta) y^r\psi^{n,\s}_z(2^{-N} w+x-y)\;,
\]
where we indicated by $\tilde R(\eta)$ the remainder. Now, by properties of the mother wavelet, we now that in the $\eps$-limit the only summand that does not vanish is the last, which can be immediately seen to be bounded by $\lambda^{-1-r}2^{-\frac{n|\s|}{2}}2^{-nr}$. Since the $L^1$ norm of $\varphi^{0,\s}$ is bounded, we obtain, for $2^{-n|\s|}\leq \lambda$
\[
\langle\psi_z^{N,n,\s},\eta_x^\lambda\rangle_\eps\lesssim  \lambda^{-1-r}2^{-\frac{n|\s|}{2}-nr}
\]
uniformly in $\eps$. Exploiting a similar, but simpler, strategy one can show that, for $\lambda\leq 2^{-n|\s|}$ one has
\[
\langle\psi_z^{N,n,\s},\eta_x^\lambda\rangle_\eps\lesssim 2^{\frac{n|\s|}{2}}\,,\qquad\text{and}\qquad \langle\varphi^{N,0,\s},\eta_x^\lambda\rangle_\eps\lesssim 2^{\frac{n|\s|}{2}}\;.
\]
Thanks to the last three bounds the conclusion of the proof is straightforward. 
\end{proof}

\section{Basic results in discrete Fourier analysis}

We hereby introduce the basic Fourier analysis tools we exploited in the proof of Proposition~\ref{p:DControl}. Let $f$ be a function in $\ell^1(\Lambda_\eps)$, then we define its (discrete) Fourier transform as
\begin{equ}[e:DFT]
\CF^\eps f(k)\eqdef \eps \sum_{x\in\Lambda_\eps} f(x) e^{-2\pi i  kx}\,,\qquad\text{so that}\qquad f(x)=\int_{-\frac{1}{2\eps}}^{\frac{1}{2\eps}}\CF^\eps f(k)e^{2\pi i  kx}\dd k\;,
\end{equ}
where the second equality holds for all $x\in\Lambda_\eps$. Clearly, if $f$ and $g$ are two functions in $\ell^1(\Lambda_\eps)$, then one has 
\begin{equ}\label{e:DFTConv}
\CF^\eps(f\ast_\eps g)(k)=\CF^\eps f(k)\CF^\eps g(k)\;,
\end{equ}
where $f\ast_\eps g(x)=\eps\sum_{y\in\Lambda_\eps}f(x-y)g(y)$ is the usual discrete convolution. The following lemma, whose proof is straightforward, is a variant of  Parseval's identity for the twisted product $B_\eps$. 
\begin{lemma}\label{l:DFT}
Let $f,\,g\in\ell^2(\Lambda_\eps)$, $B_\eps$ the operator defined in~\eqref{e:Dproduct} for a measure $\mu$ satisfying Assumption~\ref{a:MeasureMu}. Then
\begin{equ}[e:DFTProd]
\eps\sum_{x\in\Lambda_\eps} B_\eps(f,g)(x)=\int_{-\frac{1}{2\eps}}^{\frac{1}{2\eps}}\CF^\eps f(k) \CF^\eps g(-k)\hat \mu(-\eps k, \eps k)\dd k\;,
\end{equ}
where $\hat \mu$ has been defined in Assumption~\ref{a:MeasureMu}. 
\end{lemma}
\begin{proof}
The proof is immediate, notice that by the definition of the discrete Fourier transform in~\eqref{e:DFT} and thanks to a simple change of variables, we have that the right hand side of~\eqref{e:DFTProd} equals
\begin{equ}
\eps^2\int_{\R^2}\sum_{x,y\in\Lambda_\eps} f(x+\eps y_1)g(y+\eps y_2) \int_{-\frac{1}{2\eps}}^{\frac{1}{2\eps}}e^{-2\pi i  k(x-y)} \dd k\mu(\dd y_1,\dd y_2)\;,
\end{equ}
from which the result follows at once. 
\end{proof}


\bibliographystyle{./Martin}
\bibliography{./bibliography}

\end{document}